\documentclass[11pt,reqno]{amsart}

\usepackage[utf8]{inputenc}
\usepackage[T1]{fontenc}
\usepackage{lmodern}
\usepackage{microtype}
\usepackage{xcolor}

\usepackage{newunicodechar}
\newunicodechar{，}{,}

\usepackage{amsmath,amssymb,amsthm,mathtools}
\usepackage{mathrsfs}

\usepackage{enumitem}
\usepackage[margin=1.15in]{geometry}

\usepackage[
    colorlinks=true,
    linkcolor=blue,
    citecolor=blue,
    urlcolor=blue
]{hyperref}

\usepackage[
    nameinlink,
    capitalise,
    noabbrev
]{cleveref}

\numberwithin{equation}{section}

\newtheorem{theorem}{Theorem}[section]
\newtheorem{proposition}[theorem]{Proposition}
\newtheorem{lemma}[theorem]{Lemma}
\newtheorem{corollary}[theorem]{Corollary}
\newtheorem{claim}[theorem]{Claim}

\theoremstyle{definition}
\newtheorem{definition}[theorem]{Definition}

\theoremstyle{remark}
\newtheorem{remark}[theorem]{Remark}

\newcommand{\R}{\mathbb{R}}
\newcommand{\C}{\mathbb{C}}
\newcommand{\N}{\mathbb{N}}
\newcommand{\Z}{\mathbb{Z}}

\title[Fourier Decay of SRB Measures]
{Fourier Decay of SRB Measures for  Skew Products}

\author{Gaétan Leclerc}

\address{
Department of Mathematics and Statistics,
University of Helsinki,
Helsinki, Finland
}

\email{gaetan.leclerc@helsinki.fi}

\author{Haojie Ren}

\address{
Department of Mathematics,
Technion--Israel Institute of Technology,
Haifa, Israel
}

\email{haojie.ren@sjtu.edu.cn}

\date{}

\begin{document}

\begin{abstract}
We study Fourier decay of natural invariant measures for real analytic
skew products
$T(x,y)=(f(x),g(x,y)) $
on $\mathbb{S}^1\times U$, where $f$ is expanding and
$0<|\partial_y g(x,y)|<1$.
Let $\mathcal S$ denote the associated solenoidal attractor.
We prove a dichotomy: either $\mathcal S$ is a single horizontal graph,
or both the SRB measure and the measure of maximal entropy have
polynomial Fourier decay. No transversality assumption is imposed.
\end{abstract}

\maketitle

%\tableofcontents

% ------------------------------------------------------------

\section{Introduction}\label{sec:introduction}
Let $f:\mathbb{S}^1\to\mathbb{S}^1$ be a real analytic expanding map, and let $U=[-1,1]$. We consider the skew product dynamical system
$$
T:\mathbb{S}^1\times U\to\mathbb{S}^1\times U,
\qquad
T(x,y)=(f(x),g(x,y)),
$$
where $g$ is a real analytic function on $\mathbb{S}^1\times U$ and satisfies
$$
0<\left|\partial_y g(x,y)\right|<1
\qquad
\text{for all }(x,y)\in\mathbb{S}^1\times U.
$$

{ 
There exists a unique SRB measure $\omega=\omega_T$ on
$\mathbb{S}^1\times U$; see \cite[Section~2]{Rams03}.
Indeed, the unique absolutely continuous invariant probability measure of
$f$ has a unique $T$-invariant lift, whose basin has full Lebesgue measure;
see \cite[Theorem~3.1, Corollary~3.5]{Kloeckner2021}.
Consequently, }
 for Lebesgue-almost every $z\in\mathbb{S}^1\times U$, we have
$$
\frac{1}{n}\sum_{k=0}^{n-1}\delta_{T^k(z)}
\longrightarrow \omega
\qquad\text{weakly as }n\to\infty.
$$

 We denote its support
$$ \mathcal{S}:= \bigcap_{n=0}^\infty T^n(\mathbb{S}^1 \times U)   \subset \mathbb{S}^1 \times U, $$
and call $\mathcal{S}$ the {\em solenoidal attractor} associated with $T$.
{
Geometrically, $\mathcal{S}$ can be described as the union of unstable curves corresponding to the different backward itineraries of the expanding map $f$; see Corollary \ref{cor:S}.
Indeed, under iteration, $T$ stretches $\mathbb{S}^1\times U$ in the base direction and contracts it in the fiber direction, while the successive images wind around $\mathbb{S}^1\times U$. Their nested intersection is the solenoidal attractor $\mathcal{S}$.
}

We say that $\mathcal{S}$ is a single graph if there exists an analytic function
$ \theta:\mathbb{S}^1\to U$ such that
$\mathcal{S}=\{(x,\theta(x)):x\in\mathbb{S}^1\}.$ This graph is horizontal if $\theta$ is constant.

The projection of $\omega$ onto $\mathbb{S}^1$ is the $SRB$ measure for $(f,\mathbb{S}^1)$. This holds more generally for a wider class of measures, called Gibbs measures. 
{Let $\psi:\mathbb{S}^1\to\mathbb{R}$ be H\"older continuous, and let
$\mu_\psi$ be the corresponding Gibbs measure for $f$. We denote by
$\omega_\psi$ its unique $T$-invariant lift to $\mathcal{S}$ and call it the
Gibbs measure on $\mathcal{S}$ associated with $\psi$.
That is,
$$
\int_{\mathcal{S}} F(x)\,d\omega_\psi(x,y)
=
\int_{\mathbb{S}^1} F(x)\,d\mu_\psi(x)
\qquad\text{for every }F\in C(\mathbb{S}^1).
$$
 See \S \ref{sec:Gibbs} for more details.
}
 The SRB measure is an example of Gibbs measure; another example is the measure of maximal entropy (MME) for $(T,\mathcal{S})$, which is associated to constant potential $\psi$.

For the SRB measure and the MME supported on the solenoidal attractor $\mathcal S$, a natural question is whether the geometry of the unstable curves produces sufficient oscillation to yield Fourier decay in every frequency direction.
The following theorem is our main result.

\begin{theorem}\label{th:main}
Let $f:\mathbb{S}^1\to\mathbb{S}^1$ be a real analytic expanding map, let $g$ be a real analytic function on $\mathbb{S}^1\times U$ such that $T(\mathbb{S}^1\times U)\subset \mathbb{S}^1\times U$ and
$0<\left|\partial_y g(x,y)\right|<1$
for all $(x,y)\in\mathbb{S}^1\times U$. Suppose that $\mathcal{S}$
is not a single horizontal graph, and denote by $\omega_{\text{SRB}}$ the SRB measure and by $\omega_{\text{MME}}$ the measure of maximal entropy. \\
For all $* \in \{\text{SRB,MME}\}$, there exist constants $C>1$ and $\gamma\in(0,1)$ such that
$$
\left|\widehat{\omega_{*}}(k,\xi)\right|
\leq
C\left(1+|k|+|\xi|\right)^{-\gamma}
\qquad
\text{for all }(k,\xi)\in\mathbb{Z}\times\mathbb{R}.
$$
\end{theorem}

{
The exceptional case is necessary: if
$\mathcal S=\mathbb S^1\times\{y_0\}$ is a single horizontal graph, then
$|\widehat{\omega}(0,\xi)|=1$ for all $\xi\in\mathbb R$.
Thus Theorem \ref{th:main} gives a sharp dichotomy.
}

\begin{remark}
In fact, as soon as $f : \mathbb{S}^1\rightarrow \mathbb{S}^1$ is not smoothly conjugated to a linear map of the form $f(x) := b x \text{ mod } 1$, 
{
then the same conclusion holds for every Gibbs lift $\omega_\psi$
associated with a H\"older continuous potential
$\psi:\mathbb S^1\to\mathbb R$.
The argument in \S~\ref{sec:Reduction-Fourier-decay} already proves this more general statement.
}
This situation is generic in the choice of $f$. Notice further that the SRB measure and the MME are equal if and only if $f$ is smoothly conjugated to a linear map. 
Notice that when $f(x)=bx \text{ mod  } 1$, then the SRB measure (equal to the MME) is the only $T$-invariant measure exhibiting Fourier decay. Indeed, its projection onto $\mathbb{S}^1$ would then yield a $f$-invariant Rachjman measure on the circle, and the only such measure is the Lebesgue measure. This can be checked easily by noticing that all the Fourier coefficients of such a measure should vanish (except the zero-th coefficient). More generally, when $f$ is smoothly conjugated to $bx \text{ mod }1$ but not equal, the situation is more complicated, since Fourier decay of equilibrium states is expected to hold, but not in every local coordinates.
\end{remark}

\begin{remark}
In the setting where $\mathcal{S}$ is not a single graph, then our result immediatly yields power Fourier decay for every pushforward $\varphi_*\omega$, where $\varphi$ is an analytic diffeomorphism of the cylinder $\mathbb{S}^1 \times U$, adapted to the skew-product structure: $\varphi(x,y) := (\varphi_1(x),\varphi_2(x,y))$. Indeed, "being a single graph" is invariant under such changes of coordinates. Notice that "being a single horizontal graph" is not invariant under such changes of coordinates. Unfortunately, it is not immediate from our method that the exponent of decay can be uniformly bounded from below under such coordinate changes, even thought we might expect a uniform decay rate to hold.
\end{remark}

\subsection{Background and motivation}
{
The study of skew products has a long history in both dynamical systems
and fractal geometry, particularly in connection with solenoidal
attractors and SRB measures.
A classical example is the fat baker's transformation,
a piecewise affine skew product studied by Alexander and Yorke
\cite{AlexanderYorke84}. In this setting, absolute continuity of the SRB
measure is equivalent to absolute continuity of the corresponding
Bernoulli convolution.
}

{
Tsujii \cite{Tsujii_2001} studied the nonlinear skew product family
\begin{equation}\nonumber
T:\mathbb{S}^1\times\mathbb{R}\to\mathbb{S}^1\times\mathbb{R},
\qquad
T(x,y)
=
\bigl(bx,\lambda y+\phi(x)\bigr),
\end{equation}
where $b>1$ is an integer, $0<\lambda<1$, and
$\phi:\mathbb{S}^1\to\mathbb{R}$ is a $C^2$ function.
When $b\lambda>1$, the map is area expanding. Tsujii proved that  the associated SRB measure is absolutely continuous with respect to the two-dimensional Lebesgue measure for $C^2$-generic $\phi$.
Related transversality techniques were later used in the study of the Hausdorff dimension of graphs of classical Weierstrass functions; see
\cite{BaranskiBaranyRomanowska,Shen2018}. In particular, Shen
\cite{Shen2018} completed Mandelbrot's conjecture in the integer case.
}

{
Avila, Gou\"ezel and Tsujii \cite{AGT06} further studied the regularity of the density of the SRB measure  for $C^r$-generic $\phi$ under suitable
conditions on the parameters.
Bam\'on, Kiwi, Rivera-Letelier and Urz\'ua
\cite{BaJanJuan} studied the topology of these solenoidal attractors and proved that, for fixed $b$ and non-constant Lipschitz $\phi$, the attractor $\mathcal{S}$ is homeomorphic to a closed topological annulus when $\lambda$ is sufficiently close to $1$.
More recently, for this family in the real-analytic setting, the first
author \cite{REN24ADV} obtained a complete classification of the
dimension of the associated SRB measure.
}

{
For nonlinear baker maps, Rams \cite{Rams03} proved absolute continuity
of the SRB measure under area-expansion and transversality assumptions.
Higher-dimensional analogues and related questions concerning absolute continuity, regularity and dimension were studied in
\cite{Simon99,SiSo99,Rams04,MU08,BBR2018,BS22,BRE2022,MPR22,BockerBortolottiCastro2024,Ren2024CMP,GEA2026}.
}

{
For a general real analytic skew product
$T(x,y)=(f(x),g(x,y))$, a natural question is whether absolute
continuity holds under volume expansion without an additional
transversality assumption. More precisely, if the attractor is not
a single graph and
$|f'(x)\partial_y g(x,y)|>1$
everywhere, must the SRB measure be absolutely continuous with respect
to the two-dimensional Lebesgue measure?
}

{
We turn instead to Fourier decay, which is meaningful even for singular SRB measures.
For the real analytic skew products considered here, we prove that
a single horizontal graph is the only obstruction to polynomial
Fourier decay of the SRB measure and the measure of maximal entropy.
The result requires neither volume expansion nor any additional
transversality assumption.
}

{
Fourier decay bounds for fractal measures arising from chaotic dynamical
systems form a very active topic and have many applications; see, for
example, Sahlsten's survey \cite{Sah25} for an overview. In dimension
one, the Davenport--Erd\H{o}s--LeVeque criterion gives applications related to 
normal numbers \cite{DEL63}. Fourier decay also plays a role in the
study of sets of uniqueness
\cite{Salem1943SetsOU,KS64,Lai2026OnCO}, in Fourier restriction problems
for fractal measures, which are closely related to Kakeya-type questions
\cite{Ma15,LW18}, and in the fractal uncertainty principle
\cite{BD17,BS23}. In some dynamical settings, it is also related to
exponential mixing for dynamical extensions
\cite{Li2018FourierDR} or dispersion estimates for discrete quasiperiodic Schrödinger operators \cite{Le25}.
} \\
Fourier decay bounds in a dynamical setting have a long history but is getting renewed attention recently. Recent setting of interest includes Patterson-Sullivan measures \cite{BD17,BKS24,LequenSahlsten}, self-conformal measures \cite{KaufmanCF,QueffelecRamare,JS16,AlgomHertzWangLog,AHW23,BS23,SahlstenStevens,BKS24}, self-affine and stationary measures \cite{LindenstraussVarju,Li,Li2018FourierDR,LiSahlstenAffine,LiSahlsten,Rapaport}, and Gibbs measures invariant by hyperbolic dynamics \cite{LeclercJulia,LeclercBunched,BKS24,Le25}. Different approaches were developed to establish Fourier decay for fractal measures, some of them using additive combinatorics in some way to find cancellations in exponential sums, such as Jean Bourgain's \emph{sum-product phenomenon} \cite{Bo10} and Osama Khalil's \emph{$L^2$-flattening} method \cite{Kh23}, others using renewal theory \cite{AHW23,BKS24} or comparison to the Lebesgue measure in a large dimension regime \cite{KaufmanCF,QueffelecRamare,LequenSahlsten}. In our case, the main responsible for 2D Fourier decay is transversality, allowing us to upgrade a 1D-dimensional Fourier decay statement to a 2D-dimensional one. This is related to another setting getting a lot of attention, "nonlinear pushforwards" of measures. Such work in this direction include \cite{MosqueraShmerkin,ACWW,BakerBanaji,BanajiYu,BanajiYu2,AlgomBenOvadiaHertzShannon} for smooth pushforward, and \cite{Le22,LSS26} for pushforward by Hölder autosimilar functions, such as Weierstrass maps. A gap in the litterature is the case of pushforwards of $f$-invariant Gibbs measures when the analytic expanding map $f:\mathbb{S}^1 \rightarrow \mathbb{S}^1$ is smoothly conjugated, but not equal to, $f(x)=bx \text{ mod } 1$.

\subsection{Outline of the proof}
{
The main novelty of our paper is to establish the tree theorem,
Theorem \ref{thm:tree-theorem-general-case}, without any additional
transversality assumption. We write the unstable curves as graphs of
analytic functions $\theta_{\mathbf a}$; see
\S~\ref{subsec:Unstable-curves}. The key ingredient is the following
dichotomy: either $\mathcal S$ consists of a single graph, or, at every suitable scale, one can find suitable descendants whose
corresponding functions are uniformly separated in some derivative
of order at most $K$; see Proposition \ref{pro:Separation}.
Iterating this separation along the symbolic tree yields the polynomial
non-concentration estimate in the tree theorem.
}

\medskip

\noindent\textbf{Reduction to the tree theorem.}
{
To keep the notation simple, we explain the argument for the SRB measure
in the  case where
$f(x)=bx \bmod 1,$
assuming that $\mathcal S$ is not a single graph. See \S~\ref{sec:Reduction-Fourier-decay} for the general case.
}

{
The SRB measure $\omega$ admits a disintegration
$\omega=\int_{\mathbb S^1} \delta_x \otimes m_x\,dx,$
where $m_x$ is a probability measure supported on
$\mathcal S_x:=\mathcal S\cap(\{x\}\times U).$
Hence
$$
\widehat\omega(k,\xi)
=
\int_{\mathbb S^1}
e^{2\pi i kx}\widehat m_x(\xi)\,dx.
$$
By the Cauchy--Schwarz inequality, we have
\begin{equation}\label{eq:intro-cauchy-schwarz}
|\widehat\omega(k,\xi)|^2
\leq
\int_{\mathbb S^1}|\widehat m_x(\xi)|^2\,dx.
\end{equation}
The regime where $|k|$ dominates $|\xi|$ can be treated by standard
integration by parts in $x$. It remains to consider the regime
$|\xi|\gtrsim |k|$.
By \eqref{eq:intro-cauchy-schwarz}, it suffices to estimate
$$
\int_{\mathbb S^1}|\widehat m_x(\xi)|^2\,dx.
$$
}

{
In this case, let $m$ denote the uniform Bernoulli measure on the coding space of the unstable curves. Then
$m_x=(\theta_\bullet(x))_*m.$
By Fubini's theorem, 
$$
\int_{\mathbb S^1}|\widehat m_x(\xi)|^2\,dx
=
\iint
\left(
\int_{\mathbb S^1}
e^{2\pi i\xi(\theta_{\mathbf a}(x)-\theta_{\mathbf b}(x))}\,dx
\right)
dm^{\otimes 2}(\mathbf a,\mathbf b).
$$
Therefore the problem is reduced to estimating oscillatory integrals with
phases $\theta_{\mathbf a}-\theta_{\mathbf b}$.
}

{
For a general nonlinear fiber map $g$, uniform first-order
transversality cannot be expected for the phases
$\theta_{\mathbf a}-\theta_{\mathbf b}$. We therefore use higher-order
transversality. Following the formulation in \cite{Le22}, the tree theorem
gives a quantitative non-concentration estimate for the vectors
$$
\Theta_{\mathbf a}(x)
:=
\bigl(\theta_{\mathbf a}^{(k)}(x)\bigr)_{1\leq k\leq K}.
$$
This non-concentration estimate is the input for the van der Corput
argument.  Roughly speaking, there exists $\gamma\in(0,1)$ such that, uniformly
in $x$ and $t$,
\begin{equation}\label{eq:proba-decay}
m\left\{
\mathbf a:
\Theta_{\mathbf a}(x)\in B(t,\sigma)
\right\}
\lesssim
\sigma^\gamma.
\end{equation}
The precise discretized statement is given in Theorem
\ref{thm:tree-theorem-general-case}.
}

\medskip

\noindent\textbf{Reduction to transversality and the tree argument.}
{
The polynomial decay in \eqref{eq:proba-decay} is reduced to a
quantitative transversality estimate. For two symbolic sequences
$\mathbf a,\mathbf b$, write
$$
\Delta_{\mathbf a,\mathbf b}(x)
:=
\theta_{\mathbf a}(x)-\theta_{\mathbf b}(x).
$$
Since
$\Theta_{\mathbf a}(x)-\Theta_{\mathbf b}(x)
=
\bigl(\Delta_{\mathbf a,\mathbf b}^{(j)}(x)\bigr)_{1\leq j\leq K},$
if
$$
\max_{1\leq j\leq K}
\left|\Delta_{\mathbf a,\mathbf b}^{(j)}(x)\right|
>2\sigma,
$$
then $\Theta_{\mathbf a}(x)$ and $\Theta_{\mathbf b}(x)$ cannot both
belong to the same ball $B(t,\sigma)$.}

{
Proposition \ref{pro:Separation} allows us to iterate this observation along
the symbolic tree. For each $x$ and each prefix $\mathbf u$ at the relevant
scales, we can choose two words $\mathbf a,\mathbf b$ of a fixed length $N$
such that, for all continuations $\mathbf c,\mathbf d$, at least one of
the first $K$ derivatives of
$\theta_{\mathbf u\mathbf a\mathbf c}
-\theta_{\mathbf u\mathbf b\mathbf d}$
has absolute value greater than $2\sigma$ at $x$.
Hence at least one of the two corresponding subtrees is disjoint from
the event in \eqref{eq:proba-decay}. Since $m$ is the uniform Bernoulli
measure, each such subtree carries a fraction $b^{-N}$ of the mass of the
cylinder determined by $\mathbf u$. Thus a fixed proportion of the mass
is discarded at each step. After a number of iterations comparable to
$\log \sigma^{-1}$, this gives the required polynomial decay; see
\S~\ref{sec:proof-tree-theorem} for details.
} 

\medskip

\noindent\textbf{Transversality.}
{
The proof of Proposition \ref{pro:Separation} splits into two cases.
When $\partial_y^2 g\not\equiv 0$, the argument requires the
higher-order transversality developed above; see Proposition
\ref{pro:Separation-nozero-case}. When $\partial_y^2 g\equiv 0$, the
system is affine in the fiber variable, and a separate argument based
only on first-order transversality is sufficient; see Proposition
\ref{pro:Separation-zero-case}.
}

\medskip
\noindent\textbf{Case I: $\partial_y^2g\not\equiv0$.}
{
We now explain the main ideas in the proof of Proposition
\ref{pro:Separation-nozero-case}. The main difficulty is to understand how a long
common prefix $\mathbf u$ affects the derivatives of
$\Delta_{\mathbf u\mathbf a\mathbf c,\mathbf u\mathbf b\mathbf d}$.
To isolate this dependence, we write
\begin{equation}\label{eq:Intro-Lambda}
\Delta_{\mathbf u\mathbf a\mathbf c,\mathbf u\mathbf b\mathbf d}(x)
=
\Lambda_{\mathbf u;\mathbf a\mathbf c,\mathbf b\mathbf d}(x)
\Delta_{\mathbf a\mathbf c,\mathbf b\mathbf d}(\mathbf u[x]),
\end{equation}
where
$$
\Lambda_{\mathbf u;\mathbf a\mathbf c,\mathbf b\mathbf d}(x)
:=
\frac{
\Delta_{\mathbf u\mathbf a\mathbf c,\mathbf u\mathbf b\mathbf d}(x)
}{
\Delta_{\mathbf a\mathbf c,\mathbf b\mathbf d}(\mathbf u[x])
}.
$$
The second factor describes the difference between the two continuations at $\mathbf u[x]$, while $\Lambda_{\mathbf u;\mathbf a\mathbf c,\mathbf b\mathbf d}$ records the effect of the common prefix $\mathbf u$.
Here $\mathbf u[x]$ denotes the image of $x$ under the inverse branch
determined by $\mathbf u$.
We only consider
$\Delta_{\mathbf a\mathbf c,\mathbf b\mathbf d}\not\equiv 0$; in this
case the quotient extends analytically across the zeros of the denominator.
See \S~\ref{sub:def-normalization} for the precise definitions.
}

{
Differentiating \eqref{eq:Intro-Lambda} $k$ times, the leading term is
$$
\Lambda_{\mathbf u;\mathbf a\mathbf c,\mathbf b\mathbf d}^{(k)}(x)
\Delta_{\mathbf a\mathbf c,\mathbf b\mathbf d}(\mathbf u[x]),
$$
since all remaining terms contain derivatives of $\mathbf u[x]$ and
decay exponentially with $|\mathbf u|$; see
Lemma \ref{lem:higher-order-product-composition-formal}. The second factor is controlled separately by the choice of the
descendants, so the main point is to  control the higher-order derivatives of
$\Lambda_{\mathbf u;\mathbf a\mathbf c,\mathbf b\mathbf d}$.
After normalization and a suitable correction, the resulting function
is approximated by an analytic function $H_{\mathbf u}$ depending only on $\mathbf u$;
see Proposition \ref{pro:approximate-H-formal}. The required
transversality therefore follows from the corresponding transversality
of the family $\{H_{\mathbf u}\}_{\mathbf u}$.
}

{
The functions $H_{\mathbf u}$ satisfy an addition formula on the symbolic
space; see Lemma \ref{lem:addtion-H_u-formal}. In this case, this
implies that, if $\mathcal S$ is not a single graph, then
$
H_{\mathbf u}'\not\equiv 0
$
for every word $\mathbf u$; see Proposition
\ref{pro:non-degenecay-H_u-formal}. Indeed, if $H_{\mathbf v}'\equiv 0$ for some $\mathbf v$, the addition
formula, together with the assumption $\partial_y^2 g\not\equiv 0$,
implies that there are only finitely many distinct unstable curves.
This contradicts the fact that the family of unstable curves is infinite
whenever $\mathcal S$ is not a single graph.
Finally, the convergence of $H_{\mathbf u}$, together with the above
non-degeneracy and analyticity, yields the uniform finite-order
transversality required in Proposition \ref{pro:Separation}; see
\S~\ref{sec:transversality-I}.
}

\medskip
\noindent\textbf{Case II: $\partial_y^2 g\equiv 0$.}
{
In this case,
$g(x,y)=\lambda(x)y+\phi(x).$
The unstable curves $\theta_{\mathbf u}$ admit explicit convergent series
representations; see \eqref{eq:theta-series-lambda(x)}. Using this
summation structure, after a suitable normalization the first derivative
of $\Delta_{\mathbf u\mathbf a,\mathbf u\mathbf b}$ can be written in the
form
$$
t_1(\mathbf u,x)\,
\Delta_{\mathbf a,\mathbf b}(\mathbf u[x])
+
t_2(\mathbf u,x)\,
\Delta_{\mathbf a,\mathbf b}'(\mathbf u[x]),
\qquad
t_1(\mathbf u,x)^2+t_2(\mathbf u,x)^2=1.
$$
The problem is therefore reduced to controlling the projections of the
vectors
$$
\bigl(
\Delta_{\mathbf a,\mathbf b}(\mathbf u[x]),
\Delta_{\mathbf a,\mathbf b}'(\mathbf u[x])
\bigr)
$$
in different directions. An analysis of these projections gives the
required first-order transversality in the affine case; see
Section~\ref{sec:transversality-II}.
}

\subsection{Further discussion.}
{
A related question concerns the family
$$
T_N(x,y)=(f^N(x),g(x,y))
$$
and its solenoidal attractors $\mathcal S_N$.
Outside the single-graph case, does $\mathcal S_N$ have positive
Lebesgue measure for all sufficiently large $N$?
Does its dimension tend to $2$ as $N\to\infty$? It seems possible to adapt the strategy of \cite{LSS26} in this setting, to deal with "accelerated dynamics" in the base.

Another related question concerns exponential mixing of circle extensions over $\mathcal{S}$, of the form $ T_\tau : \mathcal{S} \times \mathbb{S}^1 \rightarrow \mathcal{S} \times \mathbb{S}^1$, $T_\tau(q,\theta) := (T(q),\theta+\tau(q))$, for some $\tau \in C^\omega(\mathcal{S},\mathbb{R})$.
The rate of mixing of $T_\tau$ is nontrivial to study because of the neutral direction in the $\theta$ variable. The extension $T_\tau$ leaves invariant the measure $\omega \otimes d\theta$. By exponential mixing, we mean that there exists $\rho \in (0,1)$ such that $$ \iint h_1 \circ T_\tau^n \ {h_2} \ d\omega \otimes d\theta = \iint h_1 d\omega \otimes d\theta \iint h_2 d\omega \otimes d\theta + \mathcal{O}(\rho^n) $$
for any smooth $h_1$ and $h_2$. Are such circle extensions exponentially mixing ? \\
The link with Fourier decay is given by Plancherel formula, which gives, denoting Birkhoff sums $S_n \tau(q) := \sum_{k=0}^{n-1} \tau(T^k(q))$ and $\theta$-Fourier coefficients $c_\xi(h)(q) := \int_0^1 h(q,\theta) \exp(-2 i \pi \theta \xi) d\theta$:
$$ \iint h_1 \circ T_\tau^n \ {h_2} \ d\omega \otimes d\theta = \iint h_1( T^n(q),\theta+S_n \tau(q)) h_2(q,\theta) \ d\theta d\omega(q) $$
$$ = \int_{\mathcal{S}} c_0(h_1)(T^n q) c_0(h_2)(q) d\omega(q) + \sum_{\xi \in \mathbb{Z} \setminus \{0\}} \int_{\mathcal{S}} c_\xi(h_1)(T^n q) c_\xi(h_2)(q) e^{2 i \pi \xi S_n \tau(q)} d\omega(q). $$
The first integral converge exponentially fast to $\iint h_1 \iint h_2$ as soon as $T$ is exponentially mixing, and the main difficulty is to control the sum. Let us denote $\mathcal{U}:L^2(\omega) \rightarrow L^2(\omega)$, $\mathcal{U} h := h \circ T$, the Koopman operator associated to $(T,\mathcal{S})$, and denote by $\mathcal{U}^*$ its adjoint. We can bound
$$ \Big| \int_{\mathcal{S}} c_\xi(h_1)(T^n q) c_\xi(h_2)(q) e^{2 i \pi \xi S_n \tau(q)} d\omega(q) \Big|^2
\lesssim \int_{\mathcal{S}} \Big| (\mathcal{U}^*)^{n}\Big( c_\xi(h_2) e^{2 i \pi \xi S_n \tau} \Big) \Big|^2 d\omega.$$
In the setting where $T$ is a hyperbolic dynamical system, these $L^2$-norms are bounded by Dolgopat's estimates, in a process that can be considerably simplified under the knowledge that $\omega$ has power Fourier decay in all local charts, see for example \cite{Li2018FourierDR,Le25}. Can these arguments be adapted in our setting ? \\

\noindent\textbf{Structure of the paper.}
{
The paper is organized as follows.
Section~\ref{sec:preliminaries} contains the basic notation and
preliminary results. In \S~\ref{sec:Gibbs}, we introduce the Gibbs
measures considered in the paper and state the tree theorem.
Section~\ref{sec:Rigidity} establishes the rigidity properties of the
unstable curves used later in the transversality arguments.}

{
In \S~\ref{sec:Reduction-Fourier-decay}, the Fourier decay problem is
reduced to the tree theorem, which is derived in
\S~\ref{sec:proof-tree-theorem} from the quantitative separation
property in Proposition \ref{pro:Separation}.
Section~\ref{Renormalization} contains the renormalization estimates and
the approximation by the functions $H_{\mathbf u}$ needed in the
nonlinear case. The proof of Proposition \ref{pro:Separation} is completed
in \S~\ref{sec:transversality-I} and
\S~\ref{sec:transversality-II}: the former treats
$\partial_y^2 g\not\equiv 0$ by higher-order transversality, while the
latter treats the affine case $\partial_y^2 g\equiv 0$ by a first-order
argument.
}

% ------------------------------------------------------------
\section{Preliminaries}\label{sec:preliminaries}

\subsection{Setup}
Without loss of generality, we may assume that $f$ is orientation preserving and that $f(0)=0$. Indeed, if $f$ is orientation-reversing, we replace $T$ by its second iterate $T^2$, {while the SRB and MME measures remains unchanged.} After this reduction, we conjugate the system by a rotation of $\mathbb{S}^1$ sending the fixed point of $f$ to $0$.  Since this conjugacy is a translation in the base variable, it changes the Fourier transform only by a unimodular factor and hence does not affect Fourier decay.

Let $b>1$ be the degree of $f$, and set
$
\mathcal{A}:=\{0,1,\ldots,b-1\}
$ and $\mathcal{A}^* = \bigcup_{n \geq 0} \mathcal{A}^n$, {where $\mathcal{A}^0:=\{\emptyset\}$ and $\emptyset$ denotes the empty word.}

{Let $\widetilde{f}:\mathbb{R}\to\mathbb{R}$ be the unique lift of $f$ satisfying $\widetilde{f}(0)=0$. Then}
$$
\widetilde{f}(x+1)=\widetilde{f}(x)+b
\qquad
\text{for all }x\in\mathbb{R}.
$$
For each $i\in\mathcal{A}$, define
\begin{equation}\nonumber%\label{def:f_i-Formal}
  f_i(x):=\widetilde{f}^{-1}(x+i)\qquad \text{for }x\in\R.  
\end{equation}
Since $f$ is a real analytic expanding map, the lift $\widetilde{f}$ is a real analytic diffeomorphism of $\mathbb{R}$. Consequently, each $f_i$ is a real analytic contraction on $\mathbb{R}$.

Throughout this paper, we identify $\mathbb{S}^1$ with $\mathbb{R}/\mathbb{Z}$ and use $[0,1)$ as its fundamental domain.
{We also regard $g$ as its $\Z$-periodic lift to $\mathbb{R}\times U$, still denoted by $g$. Thus $g$ is real analytic on $\mathbb{R}\times U$ and satisfies}
$
g(x+1,y)=g(x,y)
$
for all $(x,y)\in\mathbb{R}\times U$. 

Throughout the rest of the paper, we assume that
$$
0<\left|\partial_y g(x,y)\right|<1
\qquad
\text{for all }(x,y)\in\mathbb{R}\times U.
$$

\subsection{Unstable curves and finite approximations} \label{subsec:Unstable-curves}
{The solenoidal attractor $\mathcal{S}$ is described in terms of the unstable curves $\{(x,\theta_{\mathbf{u}}(x))\}_{x\in\mathbb{S}^1}$, with $\mathbf{u}\in\mathcal{A}^{\infty}$. In this subsection, we study the analytic properties of these curves and quantify their approximation by finite iterates, together with the corresponding convergence rates for derivatives of arbitrary order.}
{To obtain uniform estimates for higher-order derivatives, we use holomorphic extensions and Cauchy's integral formula. The estimates established in this subsection will be used repeatedly in the subsequent sections.}

Given $\mathbf{u}=u_1\ldots u_n\in\mathcal{A}^*$, define the function
\begin{equation}\label{def:u[x]}
     \mathbf{u}[x]:= f_{u_n}\circ\cdots\circ f_{u_1}(x)\qquad\text{for }x\in\R.
\end{equation}
(Notice the letter reversal.) Let $\theta_{\mathbf{u}}:\mathbb{S}^1\times U\to U$ be the function such that \begin{equation}{\label{def:theta}}(x,\theta_{\mathbf{u}}(x,y))=T^{|\mathbf{u}|}(\mathbf{u}[x],y)\qquad\text{for all }(x,y)\in \mathbb{S}^1\times U. \end{equation}

Using the chosen lifts of the inverse branches, we extend $\theta_{\mathbf{u}}$ to a real analytic function on  $\mathbb{R}\times U$, still denoted by $\theta_{\mathbf{u}}$.
More precisely, 
{for $\mathbf{u}=u_1\ldots u_n\in\mathcal{A}^n$, we define}
\begin{equation}\label{def:G_u-formal} 
 \theta_{\mathbf{u}}(x,y):
=
g\left( (\mathbf{u}|_1)[x], g\left( (\mathbf{u}|_2)[x], \ldots, g\left( (\mathbf{u}|_n)[x],y \right) \ldots \right) \right),
\end{equation}
where
$
\mathbf{u}|_k:=u_1\ldots u_k
$
for $1\leq k\leq n$.

Given $\mathbf{v}=v_1\cdots v_m\in\mathcal{A}^*$, set $\mathbf{u}\mathbf{v}:=u_1\ldots u_nv_1\cdots v_m\in\mathcal{A}^*$. 
{Then \eqref{def:theta} and \eqref{def:u[x]} imply that}
\begin{equation}\label{eq:theta-concatenation}
\theta_{\mathbf{u}\mathbf{v}}(x,y)
=
\theta_{\mathbf{u}}\bigl(x,\theta_{\mathbf{v}}(\mathbf{u}[x],y)\bigr).
\end{equation}
{Indeed, after the iterates corresponding to $\mathbf{u}$, the base point is $\mathbf{u}[x]$, and the remaining iterates are precisely those corresponding to $\mathbf{v}$.}

Since 
$
0<\left|\partial_y g(x,y)\right|<1
$
for all $(x,y)\in\mathbb{S}^1\times U$, we obtain
\begin{equation}\label{eq:r_min-r_max}
    r_{\min}:=
\min_{(x,y)\in\mathbb{S}^1\times U}
\left|\partial_y g(x,y)\right|
>0\qquad\text{and}\qquad
r_{\max}:=
\max_{(x,y)\in\mathbb{S}^1\times U}
\left|\partial_y g(x,y)\right|
<1.
\end{equation}

By \eqref{def:G_u-formal}, \eqref{eq:theta-concatenation}, and the chain's rule, we have
\begin{equation}
\label{eq:partial-G-formula-Formal}
\partial_y\theta_{\mathbf{u}}(x,y)
=
\prod_{k=1}^{|\mathbf{u}|}
\partial_y g\left(
(\mathbf{u}|_k)[x],
\theta_{\mathbf{u}|_{k+1}^{|\mathbf{u}|}}
\left((\mathbf{u}|_k)[x],y\right)
\right),
\end{equation}
where
$
\mathbf{u}|_{k+1}^{|\mathbf{u}|}
:=
u_{k+1}\cdots u_{|\mathbf{u}|}
$
for $1\leq k<|\mathbf{u}|$, and we adopt the convention
$
\theta_{\mathbf{u}|_{|\mathbf{u}|+1}^{|\mathbf{u}|}}(x,y):=y.
$

Therefore, given $\mathbf{w}\in\mathcal{A}^*$, for any $(x,y)\in\mathbb{R}\times U$, we have
\begin{equation}
\label{eq:common-prefix-difference}
\begin{split}
\theta_{\mathbf{u}\mathbf{v}}(x,y)-\theta_{\mathbf{u}\mathbf{w}}(x,y)
&=
\theta_{\mathbf{u}}\bigl(x,\theta_{\mathbf{v}}(\mathbf{u}[x],y)\bigr)
-
\theta_{\mathbf{u}}\bigl(x,\theta_{\mathbf{w}}(\mathbf{u}[x],y)\bigr)
\\
&=
\Bigl(
\theta_{\mathbf{v}}(\mathbf{u}[x],y)
-
\theta_{\mathbf{w}}(\mathbf{u}[x],y)
\Bigr)
\\
&\quad\times
\int_0^1
\partial_y\theta_{\mathbf{u}}
\left(
x,
s\theta_{\mathbf{v}}(\mathbf{u}[x],y)
+
(1-s)\theta_{\mathbf{w}}(\mathbf{u}[x],y)
\right)
\,ds.
\end{split}
\end{equation}

Combining this with \eqref{eq:partial-G-formula-Formal}, { the bounds
$|\theta_{\mathbf{v}}|\le 1$ and $|\theta_{\mathbf{w}}|\le 1$,} and 
$r_{\max}<1$, we obtain
\begin{equation}
\label{eq:theta-common-prefix-distance}
\left|\theta_{\mathbf{u}\mathbf{v}}(x,y)-\theta_{\mathbf{u}\mathbf{w}}(x,y)\right|
\leq 2(r_{\max})^{|\mathbf{u}|}
\qquad
\text{for all }(x,y)\in [0,1]\times U.
\end{equation}

Since $f:\mathbb{S}^1\to\mathbb{S}^1$ is an expanding map, its inverse branches are uniformly contracting. 
{Moreover, their derivatives are bounded away from zero. Hence, by \eqref{def:u[x]}, we have  }
\begin{equation}\label{eq:kappa_min-max}
    0<
\kappa_{\min}
:=
\min_{\substack{\mathbf{u}\in\mathcal{A}\\ x\in[0,1]}}
\left|\mathbf{u}[x]'\right|
\qquad\text{and}\qquad
\kappa_{\max}
:=
\max_{\substack{\mathbf{u}\in\mathcal{A}\\ x\in[0,1]}}
\left|\mathbf{u}[x]'\right|
<1.
\end{equation}

{
To study the holomorphic extensions of the maps introduced above, we define the following complex neighborhoods.} Given $\delta>0$, set
\begin{equation}\label{eq:def-D-delta}
\mathcal{D}_{\delta}
:=
\left\{
z\in\mathbb{C}:
\operatorname{dist}(z,[0,1])<\delta
\right\},
\end{equation}
and
\begin{equation}\label{eq:def-U-delta}
\mathcal{U}_{\delta}
:=
\left\{
(z_1,z_2)\in\mathbb{C}^2:
\operatorname{dist}(z_1,[0,1])<\delta
\text{ and }
\operatorname{dist}(z_2,U)<\delta
\right\}.
\end{equation}

To study the analyticity of the unstable curves, we need the following lemma and corollary.

\begin{lemma}
\label{lem:uniform-holomorphic-extension-theta-u}
For any $\epsilon>0$, there exists $\delta\in (0,\epsilon)$ such that, for every
$\mathbf{u}\in\mathcal{A}^*$, the functions
$\mathbf{u}[x]$ and
$\theta_{\mathbf{u}}(x,y)$ admit holomorphic
extensions to $\mathcal{D}_{\delta}$ and
$\mathcal{U}_{\delta}$, respectively. Moreover,
for each $(z_1,z_2)\in \mathcal{U}_{\delta}$, we have
$$
\operatorname{dist}(\mathbf{u}[z_1],[0,1])<\delta
\qquad\text{and}\qquad
\operatorname{dist}(\theta_{\mathbf{u}}(z_1,z_2),U)<\epsilon.$$ 
\end{lemma}

\begin{proof}
Let $\epsilon>0$.
Set
$
\kappa:=\frac{1+\kappa_{\max}}{2}<1
\text{ and }
r:=\frac{1+r_{\max}}{2}<1.
$
Since $g$ is real analytic on $[0,1]\times U$, and since
$
\left|\partial_y g(x,y)\right|\leq r_{\max}
$
for all $(x,y)\in[0,1]\times U$, while
$
\left|\mathbf{v}[x]'\right|\leq\kappa_{\max}
$
for all $x\in[0,1]$ and $\mathbf{v}\in\mathcal{A}$, there exists $\delta_1\in(0,\epsilon)$ such that $g$ admits a holomorphic extension to $\mathcal{U}_{\delta_1}$ and every inverse branch $\mathbf{v}\in\mathcal{A}$ admits a holomorphic extension to
$
\mathcal{D}_{\delta_1}.
$
Moreover, for every $(z_1,z_2)\in\mathcal{U}_{\delta_1}$ and every $\mathbf{v}\in\mathcal{A}$, we have
\begin{equation}\label{eq:distance-decrease-Formual}
    \operatorname{dist}\bigl(\mathbf{v}[z_1],[0,1]\bigr)<\epsilon,
\qquad
\left|\mathbf{v}[z_1]'\right|\leq\kappa\qquad \left|\partial_{z_2} g(z_1,z_2)\right|\leq r.
\end{equation}
Set
$$
M_{\delta_1}
:=
\sup_{(z_1,z_2)\in \mathcal{U}_{\delta_1/2}}
\left|\partial_{z_1}g(z_1,z_2)\right|\qquad\text{and}\qquad \delta:=\frac{(1-r)\delta_1}{4(1+M_{\delta_1})}.
$$

Let $z_1\in \mathcal{D}_{\delta}$ and
$\mathbf{v}\in\mathcal{A}$. Let $x_1\in [0,1]$ with $|z_1-x_1|<\delta$.
By \eqref{eq:distance-decrease-Formual} and
Leibniz's rule,  
$$\left|\mathbf{v}[z_1]-\mathbf{v}[x_1]\right|
=
\left|\int_{x_1}^{z_1}\mathbf{v}[s]'\,ds\right|
\leq
\kappa |z_1-x_1|
<
\delta.$$
Thus, by induction, for every $\mathbf{u}\in\mathcal{A}^*$ and every $1\leq k\leq|\mathbf{u}|$, we have
$(\mathbf{u}|_k)[z_1]\in\mathcal{D}_{\delta}$.
{
Using the same argument as above, together with the chain rule, \eqref{def:u[x]}, and \eqref{eq:distance-decrease-Formual},}  we obtain
\begin{equation}
\label{eq:complex-displacement-u}
\left|\mathbf{u}[z_1]-\mathbf{u}[x_1]\right|
< \delta.
\end{equation}
Thus  $\mathbf{u}[z_1]$ is holomorphic on
$\mathcal{D}_{\delta}$ and satisfies
$\operatorname{dist}(\mathbf{u}[z_1],[0,1])< \delta $ for every $z_1\in\mathcal{D}_{\delta}$.

Let $z_2\in\C$ with $ \operatorname{dist}\bigl(z_2,U\bigr)<\delta_1$.
  Thus there exists $x_2\in U$ with $|z_2-x_2|<\delta_1$.
 Let $\mathbf{w}\in\mathcal{A}^*$.
 Then we have
\begin{equation}\label{eq:triangle-inequality-g}
\begin{split}
     &| g(\mathbf{w}[z_1],z_2)-g(\mathbf{w}[x_1],x_2) |\\&\le | g(\mathbf{w}[z_1],z_2)-g(\mathbf{w}[x_1],z_2) |+| g(\mathbf{w}[x_1],z_2)-g(\mathbf{w}[x_1],x_2) |.
\end{split}
\end{equation}
{
Hence, by the definition of $M_{\delta_1}$, together with
\eqref{eq:complex-displacement-u} and the choice of $\delta$, we obtain
}
$$  | g(\mathbf{w}[z_1],z_2)-g(\mathbf{w}[x_1],z_2) | \le M_{\delta_1}|\mathbf{w}[z_1]-\mathbf{w}[x_1]|<M_{\delta_1}\delta<(1-r)\delta_1.$$
Similarly, by  \eqref{eq:distance-decrease-Formual}, we have
$$  | g(\mathbf{w}[x_1],z_2)-g(\mathbf{w}[x_1],x_2) |\le r|z_2-x_2|< r\delta_1. $$
Combining these two inequalities with \eqref{eq:triangle-inequality-g}, we obtain
$$| g(\mathbf{w}[z_1],z_2)-g(\mathbf{w}[x_1],x_2) |<\delta_1.$$
Note that $g(\mathbf{w}[x_1],x_2)\in U$.
{Therefore, proceeding inductively from the innermost composition to the outermost one,} we obtain that, for each $\mathbf{u}\in\mathcal{A}^*$, the function 
$$ \theta_{\mathbf{u}}(z_1,z_2)
=
g\left(
(\mathbf{u}|_1)[z_1],
g\left(
(\mathbf{u}|_2)[z_1],
\ldots,
g\left(
(\mathbf{u}|_n)[z_1],z_2
\right)
\ldots
\right)
\right) $$
 is holomorphic on   $\mathcal{U}_{\delta}$ satisfying $\operatorname{dist}(\theta_{\mathbf{u}}(z_1,z_2),U)<\delta_1<\epsilon$.
Thus our claim holds.
\end{proof}

\begin{corollary}\label{cor:uniform-holomorphic-extensions}
There exist $\eta>0$ and $r\in (0,1)$ such that for any $\mathbf{w},\mathbf{u},\mathbf{v}\in\mathcal{A}^*$, the functions $\theta_{\mathbf{w}\mathbf{u}}$  and $\theta_{\mathbf{w}\mathbf{v}}$ admit holomorphic
extensions to  $\mathcal{U}_{\eta}$ satisfying
$$  | \theta_{\mathbf{w}\mathbf{u}}(z_1,z_2)-\theta_{\mathbf{w}\mathbf{v}}(z_1,z_2) |\le 4r^{|\mathbf{w}|}\qquad\text{for all }(z_1,z_2)\in \mathcal{U}_{\eta}.  $$
\end{corollary}

\begin{proof}
Since 
$
\left|\partial_y g(x,y)\right|<1
$
for all $(x,y)\in[0,1]\times U$, there exists
$\delta_1,r\in (0,1)$ such that
\begin{equation}\label{eq:uniform-complex-contraction-g_y}
  \left|\partial_{z_2} g(z_1,z_2)\right|\leq r\qquad\text{for every }(z_1,z_2)\in\mathcal{U}_{\delta_1}.
\end{equation}
{
To accommodate two successive compositions, we apply
Lemma~\ref{lem:uniform-holomorphic-extension-theta-u} twice.
First, choose $\delta_2>0$ such that the conclusion of the lemma holds
with $\epsilon=\delta_1$. Then choose $\eta>0$ such that the conclusion
holds with $\delta=\eta$ and $\epsilon=\delta_2$.
}

Combining \eqref{eq:partial-G-formula-Formal}
with \eqref{eq:common-prefix-difference} and applying
Lemma \ref{lem:uniform-holomorphic-extension-theta-u}, we have
\begin{equation}\nonumber
\begin{split}
    &\left|
    \theta_{\mathbf{w}\mathbf{u}}(z_1,z_2)
    -
    \theta_{\mathbf{w}\mathbf{v}}(z_1,z_2)
    \right|
    \\
    &\le
    \left|
    \theta_{\mathbf{u}}(\mathbf{w}[z_1],z_2)
    -
    \theta_{\mathbf{v}}(\mathbf{w}[z_1],z_2)
    \right|
    \sup_{s\in[0,1]}
    \left|
    \partial_{z_2}\theta_{\mathbf{w}}
    \left(
        z_1,
        s\theta_{\mathbf{u}}(\mathbf{w}[z_1],z_2)
        +
        (1-s)\theta_{\mathbf{v}}(\mathbf{w}[z_1],z_2)
    \right)
    \right|
    \\
    &\le
    4
    \sup_{\substack{(z_1',z_2')\in\mathcal{U}_{\delta_2}}}
    \left|
    \prod_{k=1}^{|\mathbf{w}|}
    \partial_{z_2}g
    \left(
        (\mathbf{w}|_k)[z_1'],
        \theta_{\mathbf{w}|_{k+1}^{|\mathbf{w}|}}
        \left(
            (\mathbf{w}|_k)[z_1'],
            z_2'
        \right)
    \right)
    \right|
    \\
    &\le
    4r^{|\mathbf{w}|}.
\end{split}
\end{equation}
Here, in the second inequality, we used
$|\theta_{\mathbf{u}}(\mathbf{w}[z_1],z_2)|\le 2$ and
$|\theta_{\mathbf{v}}(\mathbf{w}[z_1],z_2)|\le 2$.
Thus the claim follows.
\end{proof}

For convenience, given $\mathbf{u}\in\mathcal{A}^*$, whenever no confusion can arise, we write
$$ \theta_{\mathbf{u}}(x):= \theta_{\mathbf{u}}(x,0)\qquad\text{for all }x\in \R. $$
Let $\mathbf{w}\in\mathcal{A}^{\infty}$.
By the Weierstrass convergence theorem and
Corollary \ref{cor:uniform-holomorphic-extensions}, the following function
$$ \theta_{\mathbf{w}}(x):=\lim_{n\to\infty}\theta_{\mathbf{w}|_n}(x) $$
is real analytic on $[0,1]$.

We shall also need the following lemma, which provides uniform estimates for derivatives of all orders of the approximations to the unstable curves.

\begin{lemma}\label{lem:uniform-bound-k-derivative}
  There exist  $r,\delta\in(0,1)$ such that, for every $k\in\mathbb{Z}_{>0}$, there exists  $C_k>1$ for which the following holds.
  
   {
For every $\mathbf{u}\in\mathcal{A}^*$ and
$\mathbf{a},\mathbf{b}\in\mathcal{A}^*\cup\mathcal{A}^{\infty}$, the functions
$\theta_{\mathbf{a}}$, $\mathbf{u}[\cdot]$,
$\theta_{\mathbf{u}\mathbf{a}}$, and
$\theta_{\mathbf{u}\mathbf{b}}$ admit holomorphic extensions to
$\mathcal{D}_{\delta}$. Moreover, for every $z\in\mathcal{D}_{\delta}$, we have}
   $$ | \theta_{\mathbf{a}}^{(k)}(z)  | \le C_k,\qquad | \mathbf{u}[z]^{(k)}  | \le C_k r^{|\mathbf{u}|},\qquad| (\theta_{\mathbf{u}\mathbf{a}} -\theta_{\mathbf{u}\mathbf{b}})^{(k)}(z)  | \le C_k r^{|\mathbf{u}|}.$$
\end{lemma}

\begin{proof}
By Lemma \ref{lem:uniform-holomorphic-extension-theta-u} and Corollary \ref{cor:uniform-holomorphic-extensions}, together with the Weierstrass convergence theorem if necessary, there exists $\delta>0$ and $r\in (0,1)$ such that such that, for every $\mathbf{u}\in\mathcal{A}^*$ and $\mathbf{a},\mathbf{b}\in\mathcal{A}^*\cup\mathcal{A}^{\infty}$, the functions $\mathbf{u}[z],\theta_{\mathbf{u}}(z), \theta_{\mathbf{u}\mathbf{a}}(z), \theta_{\mathbf{u}\mathbf{b}}(z)$ are  analytic on $\mathcal{D}_{4\delta}$ satisfying 
\begin{equation}\label{eq:uniform-bound-k-derivative-1}
    |\theta_{\mathbf{a}}(z)|\le2\qquad\text{and}\qquad 
| (\theta_{\mathbf{u}\mathbf{a}} -\theta_{\mathbf{u}\mathbf{b}})(z)  | \le 4 r^{|\mathbf{u}|}\qquad\text{for each }z\in\mathcal{D}_{4\delta}.
\end{equation}
By the chain rule, after decreasing $\delta$ and increasing $r\in(0,1)$ if necessary, we may assume that $|\mathbf{u}[z]'|\le r^{|\mathbf{u}|}$.

Let $k\in\Z_{>0}$ and $z\in \mathcal{D}_{\delta}$.
Let $\gamma$ be
a positively oriented curve whose image is
$\partial\mathcal{D}_{2\delta}$. Then, for every
$\mathbf{a}\in\mathcal{A}^*$, Cauchy's integral formula gives
$$
\theta_{\mathbf{a}}^{(k)}(z)
=
\frac{k!}{2\pi\mathrm{i}}
\oint_{\gamma}
\frac{ \theta_{\mathbf{a}}(\xi) }{(\xi-z)^{k+1}}\,d\xi.
$$
{Since $|\xi-z|\ge\delta$ for every $\xi\in\gamma$, by \eqref{eq:uniform-bound-k-derivative-1}, we obtain}
 $$ |\theta_{\mathbf{a}}^{(k)}(z)|\leq
\frac{k!}{2\pi}
\cdot
\frac{2}{\delta^{k+1}}
|\gamma|, $$  
where $|\gamma|$ denotes the length of the curve $\gamma$.  

{Using \eqref{eq:uniform-bound-k-derivative-1} and $|\mathbf{u}[z]'|\le r^{|\mathbf{u}|}$, and applying the same argument as above,}  we obtain 
$$
 |\mathbf{u}[z]^{(k)}|\leq
\frac{(k-1)!}{2\pi}
\cdot
\frac{|\gamma|r^{|\mathbf{u}|}}{\delta^{k}} 
\qquad\text{and}\qquad
| (\theta_{\mathbf{u}\mathbf{a}} -\theta_{\mathbf{u}\mathbf{b}})^{(k)}(z)  | \le \frac{k!}{2\pi}
\cdot
\frac{4|\gamma|r^{|\mathbf{u}|}}{\delta^{k+1}}.$$

Thus, the lemma follows with $C_k:=\frac{k!}{2\pi}
\cdot
\frac{4}{\delta^{k+1}}
|\gamma|$.
\end{proof}

\begin{corollary}\label{cor:S}
The Solenoid $\mathcal{S} := \bigcap_{n \geq 0} T^n (\mathbb{S}^1 \times U)$ is a union of analytic graphs:
$$ \mathcal{S} = \bigcup_{\mathbf{u} \in \mathcal{A}^\infty} \{ (x,\theta_\mathbf{u}(x)) , \ x \in \mathbb{S}^1 \}. $$
\end{corollary}

\begin{proof}
Denote $\mathcal{G} := \bigcup_{\mathbf{u} \in \mathcal{A}^\infty} \{ (x,\theta_\mathbf{u}(x)) , \ x \in \mathbb{S}^1 \}$ the union of the graphs, and let us show that $\mathcal{S} = \mathcal{G}$.
First of all, $\mathcal{G}$ is compact and $T$-invariant. Indeed, fixing $x=b[y]$ and $\mathbf{a} \in \mathcal{A}^\infty$, we find $T(x,\theta_\mathbf{a}(x)) = T(b[y],\theta_\mathbf{a}(b[y])) = (y,\theta_{b \mathbf{a}}(y)) \in \mathcal{G}$, using \eqref{def:theta} and \eqref{eq:theta-concatenation}. It follows that $ \mathcal{G} \subset \bigcap_{n \geq 0} T^{n}(\mathbb{S}^1 \times {U}) = \mathcal{S}$. Reciprocally, we have $T^n(\mathbb{S}^1 \times U) \subset \mathcal{G}_{\mathcal{O}(r^n)}$,
for some $r \in (0,1)$, where $\mathcal{G}_{\varepsilon} := \{q \in \mathbb{S}^1 \times U, d(q,\mathcal{S})< \varepsilon \}$. Indeed, fixing $x=\mathbf{a}[z] \in \mathbb{S}^1$, $\mathbf{a} \in \mathcal{A}^n$ and $y \in U$ , we have $ T^n(x,y) = (z,\theta_\mathbf{a}(z,y)) $. Extending randomly $\mathbf{a} \in \mathcal{A}^n$ into an infinite word $\mathbf{aw} \in \mathcal{A}^\infty$, Corollary \ref{cor:uniform-holomorphic-extensions} yields $\theta_\mathbf{a}(x,y) = \theta_\mathbf{aw}(x) + \mathcal{O}(r^n) $, which indeed gives $T^n(x,y) \in \mathcal{G}_{\mathcal{O}(r^n)}$. It follows that $\mathcal{S} = \bigcap_n T^n(\mathbb{S}^1 \times U) \subset \mathcal{G}$, and that $\mathcal{G}=\mathcal{S}$.\end{proof}

\begin{remark}
The solenoid $\mathcal{S}$ is a single graph if and only if $\theta_\mathbf{u}=\theta_\mathbf{v}$ for all $\mathbf{u},\mathbf{v} \in \mathcal{A}^\infty$.
\end{remark}

\subsection{Symbolic Addition}
\label{subsec:symbolic-addition}
{
In this subsection, we recall an addition operation on the symbolic space
$\mathcal{A}^{\infty}$ and establish an identity relating
$\theta_{\mathbf{u}}(x+k)$ to $\theta_{\mathbf{u}+k}(x)$. This identity will be used repeatedly in our rigidity and transversality analysis of unstable curves later in the paper. The use of symbolic addition in our transversality analysis is inspired by Gao and Shen \cite{gao2024low}.
}

 For $i\in\mathcal{A}$, recall that
\begin{equation}\label{def:f_i-formal}
  f_i(x):=\widetilde{f}^{-1}(x+i)\qquad \text{for }x\in\R.  
\end{equation}
{Since  $\widetilde{f}$ is a lift of the degree $b$ map $f$, we have $$ \widetilde{f}(x+m)=\widetilde{f}(x)+bm \qquad \text{for all }x\in\mathbb{R}\text{ and }m\in\mathbb{Z}. $$ 
Consequently,}
\begin{equation}\label{eq:f_i-+m-formal}
    \widetilde{f}^{-1}(x+bm)=\widetilde{f}^{-1}(x)+m\qquad\text{ for all }x\in\R\text{ and }m\in\Z.
\end{equation}

We define an addition on $\mathcal{A}^{\infty}$  as follows.
Given $n\in\Z$ and $\mathbf{u}=u_1u_2\ldots\in\mathcal{A}^{\infty}$, let $\mathbf{u}+n$ be the unique element $\mathbf{v}=v_1v_2\ldots\in\mathcal{A}^{\infty}$ such that 
\begin{equation}\label{eq:symbol-space-addition-n-formal}
     n+\sum_{k=1}^m  u_kb^{k-1}=\sum_{k=1}^m  v_kb^{k-1}\pmod{ b^m}\qquad\text{for each }m\in\Z_{>0}.
\end{equation}

Given $m\in\mathbb{Z}_{>0}$ let $q_m\in\mathbb{Z}$ be defined by \begin{equation}\nonumber
n+\sum_{k=1}^m u_kb^{k-1} = \sum_{k=1}^m v_kb^{k-1} + q_mb^m. \end{equation} 
Setting $q_0:=n$, we have \begin{equation}\nonumber
q_{m-1}+u_m=v_m+bq_m \qquad\text{for every }m\in\mathbb{Z}_{>0}. \end{equation}
Hence, by \eqref{def:f_i-formal}, \eqref{def:u[x]}, and \eqref{eq:f_i-+m-formal}, an induction on $m$ gives
$$(\mathbf{u}|_m)[x+n] = ((\mathbf{u}+n)|_m)[x]+q_m.$$
In particular,
 \begin{equation}\label{eq:addition-f_u-formal}
     (\mathbf{u}|_m)[x+n]=((\mathbf{u}+n)|_m)[x]\pmod{1}.
 \end{equation}

Thus we have the following lemma.

\begin{lemma}\label{lem:addition-theta_u-formal}
    For $\mathbf{u}\in\mathcal{A}^{\infty}$, $n\in\Z$, and $m\in\Z_{>0}$, the following holds:
    $$  \theta_{\mathbf{u}|_m}(x+n)\equiv \theta_{(\mathbf{u}+n)|_m}(x)\qquad\text{and}\qquad \theta_{\mathbf{u}}(x+n)\equiv \theta_{\mathbf{u}+n}(x).$$
\end{lemma}

\begin{proof}
    Let $\mathbf{u}\in\mathcal{A}^{\infty}$. Since $\theta_{\mathbf{u}}(x)=\lim_{m\to\infty}\theta_{\mathbf{u}|_m}(x)$, we only need to show that
    \begin{equation}\label{eq:addition-G_u-formal}
      \theta_{\mathbf{u}|_m}(x+n)\equiv \theta_{(\mathbf{u}+n)|_m}(x) \qquad\text{for all }m\in\Z_{>0}.
    \end{equation}

  By \eqref{eq:addition-f_u-formal}, for each $k\in\{1,\cdots,m\}$, we have
$$(\mathbf{u}|_k)[x+n]=((\mathbf{u}+n)|_k)[x]\mod 1.$$
Thus \eqref{def:G_u-formal} implies that \eqref{eq:addition-G_u-formal} holds by induction. Then our claim holds.
\end{proof}

Since $\theta_{\mathbf{u}}(x)$ is real analytic on $[0,1]$ for each $\mathbf{u}\in\mathcal{A}^{\infty}$,
 it follows from Lemma \ref{lem:addition-theta_u-formal} that 
$\theta_{\mathbf{u}}(x)$ is real analytic on $\R$.

\section{Gibbs Measure on $\mathcal{S}$ and tree theorem}
\label{sec:Gibbs}
In this section, we give a precise characterization of Gibbs measures on
$\mathcal{S}$, together with a disintegration formula and an explicit
finite level expansion of the corresponding conditional measures.
Then, we establish regularity estimates for the weights appearing in this
finite level expansion.

\subsection{Gibbs Measure on $\mathcal{S}$}
Let $\psi:\mathbb{S}^1\to\mathbb{R}$ be a H\"older continuous
function. By the Ruelle--Perron--Frobenius
theorem, there exists a unique Gibbs measure
$\mu_\psi$ on $\mathbb{S}^1$ for the potential $\psi$ with respect to $f$.

Let
$$
\mathcal{L}_{\psi}g(x)
:=
\sum_{f(y)=x}e^{\psi(y)}g(y),\qquad\qquad g\in C(\mathbb{S}^1)
$$
be the Ruelle transfer operator associated with $\psi$.
Applying \cite[Theorem~3.6(ii)]{RuelleExpandingMaps} to the  H\"older continuous function $e^{\psi}$, there exist
$\lambda_\psi\in\mathbb{R}$ and a H\"older continuous function
$h_\psi:\mathbb{S}^1\to\mathbb{R}$ such that
$$
\mathcal{L}_{\psi}e^{h_\psi}
=
e^{\lambda_\psi+h_\psi}.
$$
Identifying $\mathbb{S}^1$ with $\mathbb{R}/\mathbb{Z}$, we also regard
$h_\psi$ as a  H\"older continuous function on $[0,1]$.

Thus the function
\begin{equation}\label{eq:widetilde-psi}
\widetilde{\psi}
:=
\psi+h_\psi-h_\psi\circ f-\lambda_\psi
\end{equation}
is a  H\"older continuous normalized potential, that is,
$\mathcal{L}_{\widetilde{\psi}}\mathbf{1}=\mathbf{1}.$
Therefore we have
\begin{equation}\label{eq:sum-operator}
    \sum_{\mathbf{u}\in\mathcal{A}^n }
    \exp\left( \sum_{k=1}^n\widetilde{\psi}((\mathbf{u}|_k)[x])\right)=1\qquad\text{for all }n\in\Z_{>0}\text{ and }x\in [0,1].
\end{equation}

Moreover, by
\cite[Theorem~3.6(ii) and Corollary~5.2]{RuelleExpandingMaps},
$\mu_\psi$ is the unique Borel probability measure such that
\begin{equation}\label{eq:Ruelle-operator-int}
\int_{\mathbb S^1}
\mathcal L_{\widetilde{\psi}}g\,d\mu_\psi
=
\int_{\mathbb S^1}g\,d\mu_\psi
\qquad
\text{for every }g\in C(\mathbb S^1).
\end{equation}

{
\medskip
\noindent\textbf{Normalization and symbolic coding.}
Define the coding map
$\bar{\pi}:\mathcal{A}^{\Z_{\le 0}}\to[0,1]$ by
$$
\bar{\pi}(\mathbf{u})
:=
\lim_{n\to\infty}(\mathbf{u}|_{-n}^0)[0].
$$
Let $q:[0,1]\to\mathbb{S}^1$ be the canonical quotient map, and define
$$
\pi:=q\circ\bar{\pi}:
\mathcal{A}^{\Z_{\le 0}}\to\mathbb{S}^1.
$$
}

Let $\tau$ be the right shift on $\mathcal{A}^{\Z_{\le 0}}$.
{For a continuous function $\Phi$ on
$\mathcal A^{\mathbb Z_{\leq 0}}$, define
$$
\mathcal L_\Phi G(\mathbf a)
:=
\sum_{i\in\mathcal A}
e^{\Phi(\mathbf a i)}G(\mathbf a i)\qquad\text{for }G\in C(\mathcal{A}^{\mathbb Z_{\leq 0}}).
$$}

Since
$f\circ\pi=\pi\circ\tau,$
it follows from \eqref{eq:widetilde-psi} that
\begin{equation}\label{eq:widetilde-psi-1}
    \widetilde{\psi}\circ\pi
=
\psi\circ\pi+h_{\psi}\circ\pi
-(h_{\psi}\circ\pi)\circ \tau
-\lambda_{\psi}.
\end{equation}
In particular, $\psi\circ\pi$ and $\widetilde{\psi}\circ\pi$ are
cohomologous up to an additive constant, and hence have the same Gibbs
measure on $\mathcal{A}^{\Z_{\le 0}}$.

Let $\mu_{\psi\circ\pi}$ be the unique Gibbs measure on $\mathcal{A}^{\Z_{\le 0}}$ for
the potential $\psi\circ\pi$ with respect to $\tau$.  Notice that
$\widetilde{\psi}\circ\pi$ is normalized.
Therefore, for all $G\in C(\mathcal{A}^{\mathbb Z_{\leq 0}})$,
$$
\int_{\mathcal{A}^{\mathbb Z_{\leq 0}}}
\mathcal L_{\widetilde{\psi}\circ\pi}G\,
d\mu_{\psi\circ\pi}
=
\int_{\mathcal{A}^{\mathbb Z_{\leq 0}}}
G\,d\mu_{\psi\circ\pi}.
$$

For every $g\in C(\mathbb S^1)$, the corresponding Ruelle operators satisfy
$$
\mathcal{L}_{\widetilde{\psi}\circ\pi}(g\circ\pi)
=
(\mathcal{L}_{\widetilde{\psi}}g)\circ\pi.
$$
 For each $g\in C(\mathbb S^1)$, let $G=g\circ\pi$. Thus we have
$$
\int_{\mathbb S^1}
\mathcal L_{\widetilde{\psi}}g\,
d\bigl(\pi\mu_{\psi\circ\pi}\bigr)
=
\int_{\mathbb S^1}
g\,d\bigl(\pi\mu_{\psi\circ\pi}\bigr).
$$
By the uniqueness in \eqref{eq:Ruelle-operator-int}, we conclude that
\begin{equation}\label{eq:nu-equal}
  \mu_{\psi}=  \pi \mu_{\psi\circ\pi}.
\end{equation}

Let $\nu_{\psi}$ be the the unique $T$-invariant ergodic Borel probability measure on
$\mathcal{S}$ whose projection onto $\mathbb{S}^1$ is $\mu_{\psi}$.
{The uniqueness follows from the uniqueness of the natural extension of $\mu_\psi$. Indeed, once a full backward itinerary of the base point is fixed, the uniform contraction in the fibers implies that the diameters of the corresponding iterated fiber images tend to zero, and hence uniquely determines the fiber coordinate.}  

{
\medskip
\noindent\textbf{The invariant lift to $\mathcal S$.}
}
For convenience, we also denote by $\tau$ the right shift on $\mathcal{A}^{\mathbb{Z}}$.
Let $\nu_{\psi\circ\pi}$ be the $\tau$-invariant ergodic Borel probability
measure on $\mathcal{A}^{\mathbb{Z}}$ whose projection onto
$\mathcal{A}^{\mathbb{Z}_{\leq 0}}$ is $\mu_{\psi\circ\pi}$.

Given $\mathbf{u}=(u_k)_{k\in\Z}$, let
$\mathbf{u}_-:=\cdots u_{-1}u_0\in\mathcal{A}^{\Z_{\le 0}}$ and 
$\mathbf{u}_+:= u_{1}u_2\cdots\in\mathcal{A}^{\infty}$.

Define the map $\Pi$ from $\mathcal{A}^{\Z} $  to $\mathbb{S}^1\times U$ by
\begin{equation}\label{def:Pi}
  \Pi(\mathbf{u}):=\left( \pi(\mathbf{u}_-),\theta_{\mathbf{u}_+}\left( \pi(\mathbf{u}_-) \right)  \right).  
\end{equation}
{
Therefore  by \eqref{eq:theta-concatenation} we have}
$$ T\circ\Pi=\Pi\circ \tau,$$
which implies that $\Pi(\nu_{\psi\circ\pi})$  is a $T$-invariant ergodic Borel
probability measure on $\mathcal{S}$.
Since the projection of $\Pi(\nu_{\psi\circ\pi})$ onto $\mathbb{S}^1$
is $\mu_{\psi}$. by the uniqueness of $\nu_{\psi}$, it follows that
\begin{equation}\label{eq:nu-equivalent-formular}
    \nu_{\psi}=\Pi(\nu_{\psi\circ\pi}).
\end{equation}

{
\medskip
\noindent\textbf{Conditional measures on symbolic fibers.}
Given
$\mathbf{a}=\cdots a_{-1}a_0\in\mathcal{A}^{\Z_{\le 0}}$,
let $\nu_{\mathbf{a}}$ be the probability measure on
$\mathcal{A}^{\infty}$ characterized by the following cylinder
probabilities. For any $m\in\Z_{>0}$ and
$\mathbf{z}=z_1\cdots z_m\in\mathcal{A}^m$, let
$$
[\mathbf{z}]
:=
\left\{
\mathbf{u}=u_1u_2\cdots\in\mathcal{A}^{\infty}:
u_j=z_j\ \text{for all }1\le j\le m
\right\}.
$$
Moreover, let $\mathbf a\mathbf z$ denote the left-infinite sequence
obtained by appending $z_1,\ldots,z_m$ to $\mathbf a$ and then
reindexing such that $z_m$ occupies the coordinate $0$.  That is,
$$
\mathbf a\mathbf z
=
\cdots a_{-1}a_0z_1\cdots z_m\in\mathcal{A}^{\Z_{\le 0}}.
$$
We then define
\begin{equation}
\label{eq:conditional-measure-cylinder-nu-formal}
\nu_{\mathbf{a}}([\mathbf{z}])
:=
\exp\left(
\sum_{k=0}^{|\mathbf{z}|-1}
\widetilde{\psi}\circ\pi\circ\tau^k(\mathbf{a}\mathbf{z})
\right).
\end{equation}
}

{
The cylinder probabilities are consistent. Indeed, by \eqref{eq:conditional-measure-cylinder-nu-formal}, for every
$\mathbf z\in\mathcal A^m$,
$$
\sum_{i\in\mathcal A}
\nu_{\mathbf a}([\mathbf z i])
=
\nu_{\mathbf a}([\mathbf z])
\sum_{i\in\mathcal A}
e^{\widetilde\psi(\pi(\mathbf a\mathbf z i))}
=
\nu_{\mathbf a}([\mathbf z]),
$$
where the last equality follows from
$\mathcal L_{\widetilde\psi}\mathbf 1=\mathbf 1$.
Therefore, by the Kolmogorov extension theorem, these cylinder
probabilities determine a unique Borel probability measure
$\nu_{\mathbf a}$ on $\mathcal A^\infty$.
}

{After reversing the order of the coordinates, our right shift on
$\mathcal A^{\mathbb Z_{\leq0}}$ becomes the usual one sided left
shift. Thus the conditional measure construction in
\cite[Lemmas~2.1 and~2.2]{GramaQuintXiao} applies directly in our
notation. Consequently, }
\begin{equation}\label{eq:decom-nu-formal} \nu_{\psi\circ\pi}=\int_{\mathcal{A}^{\Z_{\le0}}}\delta_{\mathbf{a}}\otimes \nu_{\mathbf{a}}\,d \mu_{\psi\circ\pi}(\mathbf{a}).
\end{equation}

{Since we identify $\mathbb S^1$ with $[0,1)$ through the
bijection
$$
q|_{[0,1)}:[0,1)\to\mathbb S^1,
$$
and accordingly regard $\mu_\psi$ as a Borel probability measure on
$[0,1)$.}

{
\medskip
\noindent\textbf{Fiber disintegration and finite-level expansion.}
}
Given $x\in[0,1)$, there exists a unique sequence
$A(x):=\cdots a_{-1}a_0\in\mathcal{A}^{\Z_{\le 0}}$
such that
$$
\widetilde{f}^{\,m}(x)\bmod 1
\in f_{a_{-m}}([0,1))
\qquad\text{for every }m\in\N.
$$
Let $A(1):=\cdots\,(b-1)\,(b-1)\in\mathcal{A}^{\Z_{\le0}}$.
In this way, we obtain a coding map
$A:[0,1]\to\mathcal{A}^{\Z_{\le 0}},$
satisfying
$\bar{\pi}\circ A=\operatorname{id}_{[0,1]}.$ 
{The only points with more than one symbolic representation are the
iterated preimages of the endpoints of the Markov partition.
Since the set of such points is countable,}
there exists a countable set
$\mathcal{E}\subset\mathcal{A}^{\Z_{\le 0}}$ such that
$
A:[0,1)\to\mathcal{A}^{\Z_{\le 0}}\setminus\mathcal{E}
$
is a bijection. 

{Since  $\mu_{\psi\circ\pi}$ is non-atomic, we have
$\mu_{\psi\circ\pi}(\mathcal E)=0$.
Moreover,
$$
A\circ\bar\pi=\operatorname{id}
$$
on $\mathcal A^{\mathbb Z_{\leq0}}\setminus\mathcal E.$
Hence, under the above identification of $\mathbb S^1$ with $[0,1)$, by \eqref{eq:nu-equal}, we have
$$
A\mu_\psi
=
A\bar\pi\mu_{\psi\circ\pi}=
(A\circ\bar\pi)\mu_{\psi\circ\pi}=
\mu_{\psi\circ\pi}.
$$
}

By \eqref{eq:decom-nu-formal} and $\nu_{\psi}=\Pi(\nu_{\psi\circ\pi})$, for any $G\in C(\mathbb{S}^1\times U)$, we have
\begin{equation}\label{eq:decom-nu-Gibbs-formal}   \nu_{\psi}(G)=\int_{\mathcal{A}^{\Z_{\le0}}}\int_{\mathcal{A}^{\infty}}G\left( \pi(\mathbf{a}),\theta_{\mathbf{u}}(\pi(\mathbf{a}))\right)\,d \nu_{\mathbf{a}}(\mathbf{u})\,d \mu_{\psi\circ\pi}(\mathbf{a}).
\end{equation}

For $x\in[0,1]$, define a probability measure $m_x$ on $U$ by
\begin{equation}\label{def:m_x-formal}
    m_x:=\Pi_x(\nu_{A(x)}),
\end{equation}
where the map $\Pi_x:\mathcal{A}^{\infty}\to U$ is defined by
$$\Pi_x(\mathbf{u}):=\theta_{\mathbf{u}}(x).$$

{Using \eqref{eq:decom-nu-Gibbs-formal} and
$A\mu_\psi=\mu_{\psi\circ\pi}$, we obtain
\begin{equation}\label{eq:decom-Gibbs-nu-formal}
    \begin{aligned}
\nu_\psi(G)
&=
\int_{\mathbb S^1}
\int_{\mathcal A^\infty}
G\left(x,\theta_{\mathbf u}(x)\right)
\,d\nu_{A(x)}(\mathbf u)\,d\mu_\psi(x)\\
&=
\int_{\mathbb S^1}
\int_U G(x,y)\,dm_x(y)\,d\mu_\psi(x).
\end{aligned}
\end{equation}
}

Fix $\mathbf{a}\in\mathcal{A}^{\Z_{\le 0}}$. By
\eqref{eq:conditional-measure-cylinder-nu-formal}, for any
$\mathbf{u},\mathbf{v}\in\mathcal{A}^*$, we obtain
\begin{equation}\label{eq:conditional-measure-concatenation}
\nu_{\mathbf{a}}([\mathbf{u}\mathbf{v}])
=
\nu_{\mathbf{a}}([\mathbf{u}])\,
\nu_{\mathbf{a}\mathbf{u}}([\mathbf{v}]).
\end{equation}

Consequently, for any $n\in\Z_{>0}$ and
$F\in C(\mathcal{A}^{\infty})$,
\begin{equation}\label{eq:conditional-measure-finite-decomposition}
\nu_{\mathbf{a}}(F)
=
\sum_{\mathbf{u}\in\mathcal{A}^n}
\nu_{\mathbf{a}}([\mathbf{u}])
\int_{\mathcal{A}^{\infty}}
F(\mathbf{u}\mathbf{v})\,
d\nu_{\mathbf{a}\mathbf{u}}(\mathbf{v}).
\end{equation}

For $x\in[0,1]$ and $\mathbf{u}\in\mathcal{A}^*$, define
\begin{equation}\label{def:p-u-x}
p_{\mathbf{u}}(x)
:=
\nu_{A(x)}([\mathbf{u}]).
\end{equation}
Thus by \eqref{eq:conditional-measure-cylinder-nu-formal}, we have
\begin{equation}\label{eq:p_u}
p_{\mathbf{u}}(x)
=\exp\left(
\sum_{k=1}^{|\mathbf{u}|} \widetilde{\psi}( (\mathbf{u}|_k)[x] )
\right).
\end{equation}
Here, we note that by \eqref{eq:sum-operator}, we have
\begin{equation}\label{eq:sum=1}  \sum_{\mathbf{u}\in\mathcal{A}^m}p_{\mathbf{u}}(x)=1\qquad\text{for each }x\in [0,1]\text{ and }m\in\Z_{>0}.
\end{equation}

{Moreover, by \eqref{eq:p_u}, for any $\mathbf u,\mathbf v\in\mathcal A^*$,
\begin{equation}\label{eq:p-formula}
p_{\mathbf u\mathbf v}(x)
=
p_{\mathbf u}(x)
p_{\mathbf v}(\mathbf u[x]).
\end{equation}}

{
For all $\mathbf u,\mathbf v\in\mathcal A^*$, by
\eqref{eq:conditional-measure-concatenation} and
\eqref{eq:p-formula},
$$
\begin{aligned}
\nu_{A(x)\mathbf u}([\mathbf v])
=
\frac{\nu_{A(x)}([\mathbf u\mathbf v])}
{\nu_{A(x)}([\mathbf u])}=
\frac{p_{\mathbf u\mathbf v}(x)}
{p_{\mathbf u}(x)}=
p_{\mathbf v}(\mathbf u[x])=
\nu_{A(\mathbf u[x])}([\mathbf v]).
\end{aligned}
$$
Thus we obtain
\begin{equation}\label{eq:conditional-shift-identity}
\nu_{A(x)\mathbf u}
=
\nu_{A(\mathbf u[x])}.
\end{equation}}

{Using \eqref{eq:conditional-measure-finite-decomposition},
\eqref{eq:conditional-shift-identity}, and 
$\theta_{\mathbf u\mathbf v}(x)
=
\theta_{\mathbf u}
\left(x,\theta_{\mathbf v}(\mathbf u[x])\right),$
we obtain, for every $H\in C(U)$,
$$
\begin{aligned}
\int_U H(y)\,dm_x(y)
&=
\int_{\mathcal A^\infty}
H(\theta_{\mathbf w}(x))
\,d\nu_{A(x)}(\mathbf w)\\
&=
\sum_{\mathbf u\in\mathcal A^n}
p_{\mathbf u}(x)
\int_{\mathcal A^\infty}
H(\theta_{\mathbf u\mathbf v}(x))
\,d\nu_{A(x)\mathbf u}(\mathbf v)\\
&=
\sum_{\mathbf u\in\mathcal A^n}
p_{\mathbf u}(x)
\int_{\mathcal A^\infty}
H\left(
\theta_{\mathbf u}
\left(x,\theta_{\mathbf v}(\mathbf u[x])\right)
\right)
\,d\nu_{A(\mathbf u[x])}(\mathbf v)\\
&=
\sum_{\mathbf u\in\mathcal A^n}
p_{\mathbf u}(x)
\int_U
H(\theta_{\mathbf u}(x,y))
\,dm_{\mathbf u[x]}(y).
\end{aligned}
$$}
Thus
\begin{equation}\label{eq:fiber-measure-finite-decomposition}
\int_U H(y)\,dm_x(y)
=
\sum_{\mathbf{u}\in\mathcal{A}^n}
p_{\mathbf{u}}(x)
\int_U
H\bigl(\theta_{\mathbf{u}}(x,y)\bigr)\,
dm_{\mathbf{u}[x]}(y).
\end{equation}

{Combining \eqref{eq:decom-Gibbs-nu-formal} with
\eqref{eq:fiber-measure-finite-decomposition}, we obtain, for every
$n\in\mathbb Z_{>0}$ and
$G\in C(\mathbb S^1\times U)$,}
\begin{equation}\label{eq:int-omega}
    \nu_{\psi}(G)=\int_{\mathbb{S}^1 }
    \left(\sum_{\mathbf{u}\in\mathcal{A}^n} p_{\mathbf{u}}(x)\int_{U} G(x,\theta_{\mathbf{u}}(x,y))\,d m_{\mathbf{u}[x]}(y)     \right)\,d\mu_{\psi}(x).
\end{equation}

%\subsection{SRB measure}\label{subsec:SRB}
%Let $\mu$ be the unique absolutely continuous invariant probability measure
%of $f$ on $\mathbb{S}^1$. Its density $\rho$ is strictly positive and real analytic
% on $[0,1]$ by \cite[Theorem]{BandtlowJenkinson}. That is 
%\[
%    d\mu(x)=\rho(x)\,dx.
%\]
%Recall that the Perron--Frobenius operator associated with $f$ is given by
%$$
%\mathcal{L}g(x)
%:=
%\sum_{f(y)=x}
%\frac{g(y)}{|f'(y)|}\qquad\text{for %}g\in C([0,1]).
%$$
%Here we take $\psi(x)=\log |f'(x)|$.

%By \cite[Proposition~5.4, Proposition~7.1, and Lemma~10.2]
%{BandtlowJenkinson}, $\rho$ is the normalized strictly positive fixed point of
%$\mathcal{L}$. In particular,
%\begin{equation}\label{eq:invariant-%density}
%\sum_{f(y)=x}
%\frac{\rho(y)}{|f'(y)|}
%=
%\rho(x).
%\end{equation}

%Thus, by the definition of $\widetilde{\psi}$ and $h_{\psi}=\log \rho$ and $\lambda_{\psi}=0$, we obtain
%$$
%\widetilde{\psi}
%=
%\psi+\log \rho\circ f-\log \rho.
%$$

%Hence, combining this and $\psi(x)=\log|f'(x)|$ with
% \eqref{eq:p_u} together with  the chain rule,  we obtain
%\begin{equation}\label{eq:SRB-nu-mass}
%   p_{\mathbf{u}}(x)
%=
%\frac{\rho(\mathbf{u}[x])}{\rho(x)}
%\cdot \mathbf{u}[x]'.
%\end{equation}
%for every $x\in [0,1]$ and  %$\mathbf{u}\in\mathcal{A}^*$.

\subsection{Uniform Regularity of $p_{\mathbf{u}}$}

\begin{lemma}\label{lem:sum-p_u-Holder}
Let $\psi:\mathbb{S}^1\to\mathbb{R}$ be a H\"older continuous
function. Then there exist $C>1$ and $\alpha\in(0,1]$ such that the following holds for any $n\in\Z_{>0}$ and $\mathbf{u}\in\mathcal{A}^*$, each $x,y\in [0,1]$,
$$ \left| \frac{ p_{ \mathbf{u} }(x) }{p_{ \mathbf{u} }(0)}- \frac{ p_{ \mathbf{u} }(y) }{p_{ \mathbf{u} }(0)} \right|  \le C|x-y|^{\alpha}. $$
\end{lemma}

{
Here $p_{\mathbf{u}}(x)$ is defined as in \eqref{eq:p_u} with respect to the normalized potential function $\widetilde{\psi}$ associated with the given potential function $\psi$.
}

\begin{proof}
  Let $\widetilde{\psi}$ be the H\"older continuous function defined by
\eqref{eq:widetilde-psi}. Fix $M>1$ and $\alpha\in(0,1]$ such that
$$
|\widetilde{\psi}(s)-\widetilde{\psi}(t)|
\leq M|s-t|^\alpha
\qquad\text{for all }s,t\in[0,1].
$$

Let $n\in\Z_{>0}$ and $\mathbf{u}
\in\mathcal{A}^n$, let $x,y\in [0,1]$. Given, we have
\begin{equation}\nonumber
\begin{split}
    \left| \sum_{k=1}^{n} \widetilde{\psi}( (\mathbf{u}|_k)[x] )-\widetilde{\psi}( (\mathbf{u}|_k)[y] )\right|&\le\sum_{k=1}^{n} M\left| (\mathbf{u}|_k)[x] -(\mathbf{u}|_k)[y] \right|^{\alpha}
\\&\le\sum_{k=1}^{n} M (\kappa_{\max})^{k\alpha} \left| x-y \right|^{\alpha}\\&\le
\frac{M}{1-(\kappa_{\max})^{\alpha}}| x-y |^{\alpha}.
\end{split}
\end{equation}
{
where $\kappa_{\max}$ is as defined in \eqref{eq:kappa_min-max}.
}

Let $C_1:= \frac{M}{1-(\kappa_{\max})^{\alpha}}\exp(\frac{M}{1-(\kappa_{\max})^{\alpha}})$. 
{Therefore, using the inequality $|e^t-1|\le |t|e^{|t|}$,}
\begin{equation}\label{eq:divide-bound}
    \left|\exp\left(\sum_{k=1}^{n} \widetilde{\psi}( (\mathbf{u}|_k)[x] )-\widetilde{\psi}( (\mathbf{u}|_k)[y] ) \right) -1  \right|\le C_1 | x-y |^{\alpha}.
\end{equation}
In particular, applying this and \eqref{eq:p_u}, we have
$$ \frac{p_{\mathbf{u}}(y)}{ p_{\mathbf{u}}(0)}   \le  \left| \frac{p_{\mathbf{u}}(y)}{ p_{\mathbf{u}}(0)} -1 \right|+1\le C_1+1. $$
Thus, by \eqref{eq:divide-bound} and \eqref{eq:p_u}, we have
\begin{equation}\label{eq:bound-distosion-p_u}
    \begin{split}
         \left| \frac{ p_{ \mathbf{u} }(x) }{p_{ \mathbf{u} }(0)}- \frac{ p_{ \mathbf{u} }(y) }{p_{ \mathbf{u} }(0)} \right| &=\frac{p_{\mathbf{u}}(y)}{ p_{\mathbf{u}}(0)}  \left| \frac{ p_{ \mathbf{u} }(x) }{p_{ \mathbf{u} }(y)}- 1\right|\\&\le
         (C_1+1)\left|\exp\left(
         \sum_{k=1}^{n} \widetilde{\psi}( (\mathbf{u}|_k)[x] )-\widetilde{\psi}( (\mathbf{u}|_k)[y] ) \right) -1  \right|\\&\le
         C_1(C_1+1)| x-y |^{\alpha}.
    \end{split}
\end{equation}
Therefore our claim holds for $C:=C_1(C_1+1)$.
\end{proof}

\subsection{Tree Theorem}
Given $k\in\Z_{>0}$ and $\mathbf{a}\in\mathcal{A}^*$, define the vector valued function
\begin{equation}\label{eq:Theta_k}
    \Theta_{\mathbf{a},k}(x):=\left( \theta_{\mathbf{a}}^{(1)}(x),\ldots,\theta_{\mathbf{a}}^{(k)}(x)  \right)\qquad\text{for }x\in [0,1].
\end{equation}

The following tree theorem is the main tool of this paper.

%$\max_{t\in \mathbb{S}^1}|f'(t)|^{-1}$

\begin{theorem}
\label{thm:tree-theorem-general-case}
Assume that 
$0 <|\partial_yg(x,y) |<1$ for all $(x,y)\in   [0,1]\times U$. 
Suppose further that $\mathcal{S}$ does not consist of a single graph.

Let $\psi:\mathbb{S}^1\to\R$ be  a H\"older continuous function.
Then there exist $K\in\Z_{>0}$ and constants $C_0,\gamma,\alpha>0$ such that, for every $n\in\Z_{>0}$ and $\sigma>0$, the following holds.

For $x\in [0,1]$ and $t\in\R^K$, we have
$$\sum_{\mathbf{a}\in\mathcal{G}} 
p_{\mathbf{a}}(x)\le C_0(\sigma^{\gamma}+e^{-\alpha n}), $$
where the set $\mathcal{G}$ is defined by 
$$ \mathcal{G}:=\left\{ \mathbf{a} \in \mathcal{A}^n\::\: \Theta_{\mathbf{a},K}(x) \in \mathbf{B}(t,\sigma) \right\}. $$
\end{theorem}
Here $\mathbf{B}(t,\sigma)$ is the closed ball in $\R^K$ with center at $t$ and radius $\sigma$.
The positive probability vector $\{p_{\mathbf{a}}(x)\}_{\mathbf{a}\in\mathcal{A}^n } $ is defined by \eqref{eq:p_u}. This will be established in the last section.

\section{Reduction of Fourier decay to the tree theorem}\label{sec:Reduction-Fourier-decay}

In this section, we reduce Fourier decay to the tree Theorem \ref{thm:tree-theorem-general-case}. Recall that $\omega := \omega_\psi$ is a Gibbs measure on $\mathcal{S}$. In general, $\omega$ may be the measure of maximal entropy or the SRB measure, but if $f$ is not smoothly conjugated to a linear map, we can assume that $\omega$ is any Gibbs state for some Hölder potential. First of all, recall that the Gibbs measure $\omega$ on $\mathcal{S}$ is of the form $\omega = \int_{\mathbb{S}^1} \delta_x \otimes m_x \ d\mu(x) $ \eqref{eq:decom-Gibbs-nu-formal} where $d\mu(x)$  is the Gibbs state $\mu_\psi$ for $f$, and where $(m_x)_{x \in \mathbb{S}^1}$ is a family of probability measures, each supported on $\mathcal{S} \cap (\{x\} \times U)$. In particular, if $f(x,y)=g(x)$, we have $$\int_{\mathcal{S}} fd\omega = \int_{\mathbb{S}^1} g d\mu .$$
Furthermore, we have the following autosimilarity relation (\ref{eq:int-omega}), for any continuous $f \in C^0(\mathcal{S},\mathbb{C})$, and for any $n \geq 0$:
\begin{equation}\label{eq:autosim}\int_{\mathcal{S}} f d\omega =  \sum_{\mathbf{u} \in \mathcal{A}^n} \int_{\mathbb{S}^1}  p_\mathbf{u}(x) \int_{U} f(x,\theta_\mathbf{u}(x,y)) d m_x(y) d\mu(x) .\end{equation}
where the familly $(p_\mathbf{u}/p_\mathbf{u}(0))_{\mathbf{u} \in \mathcal{A}^*}$ is uniformly Hölder by Lemma \ref{lem:sum-p_u-Holder}. \\
 
First of all, let us establish a non-concentration bound for smooth functions, that will be helpful to reformulate the Van Der Corput Lemma in a convenient way.

\begin{lemma}\label{lem:deri}
Let $k \geq 0$ and $\varphi \in C^{k+1}(\mathbb{S}^1,\mathbb{R})$. \ Recall that $\mu$ is upper regular: there exists $\mathfrak{s} \in (0,1)$ such that for any small enough interval $I$, we have $\mu(I) \leq |I|^\mathfrak{s}$. The following holds, for every $k \geq 0$, $\sigma \in (0,1)$: $$\mu(x \in \mathbb{S}^1 \ | \ |\varphi(x)| \leq \sigma) \lesssim_k \sigma^{A^{-k}} + \mu(x \in \mathbb{S}^1 \ | \ \forall 0 \leq j \leq k, \ |\varphi^{(j)}(x)| \leq \sigma^{A^{-k}})$$
where $A := 1+2\mathfrak{s}^{-1}$.
\end{lemma}

\begin{proof}
The proof goes by induction on $k$. If $k=0$, the bound is trivial. Suppose then that the bound holds for some $k$ and let us prove it for $k+1$. Fix then $\varphi \in C^{k+2}(\mathbb{S}^1,\mathbb{R})$. Set $\eta := \sigma^{A^{-(k+1)}}$, $\eta^A = \sigma^{A^{-k}}$ so that:
$$ \mu(x \in \mathbb{S}^1 \ | \ \max_{0 \leq j \leq k} |\varphi^{(j)}(x)| \leq \sigma^{A^{-k}}  ) = $$ $$ \mu(x \in \mathbb{S}^1 \ | \ \max_{0 \leq j \leq k} |\varphi^{(j)}(x)| \leq \eta^A, |\varphi^{(k+1)}(x)| \leq \eta  ) + \mu(x \in \mathbb{S}^1 \ | \ \max_{0 \leq j \leq k} |\varphi^{(j)}(x)| \leq \eta^A, |\varphi^{(k+1)}(x)| \geq \eta  )  $$
$$ \leq \mu(x \in \mathbb{S}^1 \ | \ \max_{0 \leq j \leq k+1} |\varphi^{(j)}(x))| \leq \eta) + \mu(x \in \mathbb{S}^1 \ | \ |\varphi^{(k)}(x)| \leq \eta^A, \ |\varphi^{(k+1)}(x)| \geq \eta ). $$
To conclude, we bound the second probability. Let us introduce a partition $(S_i(\eta))_{i \in \mathcal{I}(\eta)}$ of $\mathbb{S}^1$ of diameter $|S_i(\eta)| \leq \eta \cdot (2+2\|\varphi\|_{C^{k+2}})^{-1}$, and with cardinality $\# \mathcal{I}(\eta) \lesssim \eta^{-1}$. {
(Notice that the number of elements in the partition depends on
$\|\varphi\|_{C^{k+2}}$.)
} Set $\tilde{\varphi} := \varphi^{(k)}$. We have:
$$ \mu(x \in \mathbb{S}^1 \ | \ |\tilde{\varphi}(x)| \leq \eta^A, \ |\tilde{\varphi}'(x)| \geq \eta ) $$
$$ = \sum_{i \in \mathcal{I}(\eta)} \mu(x \in S_i(\eta) \ | \ |\tilde{\varphi}(x)| \leq \eta^A, \ |\tilde{\varphi}'(x)| \geq \eta   ).$$
Now, either the set $(x \in S_i(\eta) \ | \ |\tilde{\varphi}(x)| \leq \eta^A, \ |\tilde{\varphi}'(x)| \geq \eta   )$ is empty, or because the diameter of each $S_i(\eta)$ is small enough, we have
$$(x \in S_i(\eta) \ | \ |\tilde{\varphi}(x)| \leq \eta^{A}, \ |\tilde{\varphi}'(x)| \geq \eta   ) \subset (x \in S_i(\eta) \ | \ |\tilde{\varphi}(x)| \leq \eta^A, \ \inf_{S_i(\eta)} |\tilde{\varphi}'| \geq \eta/2   ).$$ In any case, $(x \in S_i(\eta) \ | \ |\tilde{\varphi}(x)| \leq \eta^A, \ \inf_{S_i(\eta)} |\tilde{\varphi}'| \geq \eta/2   )$ is included in an interval $J_i(\eta) \subset S_i(\eta)$ of diameter at most $|J_i(\eta)| \leq \eta^{A}/(\eta/2) \leq 2\eta^{A-1}. $ Hence:
$$ \mu(x \in \mathbb{S}^1 \ | \ |\tilde{\varphi}(x)| \leq \eta^A, 
\ |\tilde{\varphi}'(x)| \geq \eta ) $$
$$ = \sum_{i \in \mathcal{I}(\eta)} \mu(x \in S_i(\eta) \ | \ |\tilde{\varphi}(x)| \leq \eta^{A}, \ |\tilde{\varphi}'(x)| \geq \eta   )$$
$$ \leq \sum_{i \in \mathcal{I}(\eta)} \mu(J_i(\eta)) \lesssim \# \mathcal{I}(\eta) \cdot \eta^{(A-1) \mathfrak{s}} \lesssim \eta^{(A-1) \mathfrak{s}-1} \lesssim \eta . $$
This provides the desired bound:
$$ \mu(x \in \mathbb{S}^1 \ | \ \max_{0 \leq j \leq k} |\varphi^{(j)}(x)| \leq \sigma^{A^{-k}}  ) \lesssim \sigma^{A^{-(k+1)}} + \leq \mu(x \in \mathbb{S}^1 \ | \ \max_{0 \leq j \leq k+1} |\varphi^{(j)}(x)| \leq \sigma^{A^{-(k+1)}}).  $$
\end{proof}

\begin{lemma}\label{lem:vdc-gibbs}
Let $K \geq 1$ and let $\alpha >0$ be small enough. Then there exists $\rho(K,\alpha,\mu) >0$ such that the following holds. Let $\phi \in C^{K+1}([0,1],\mathbb{R})$ and $\chi \in C^\alpha(\mathbb{S}^1,\mathbb{C})$. Denote $\Phi(x) := (\phi'(x),\dots, \phi^{(K)}(x)) \in \mathbb{R}^K$. Then, there exists $C \geq 1$ that only depends on $K$, $\alpha$, $\|\Phi\|_{C^{K+1}}$, and $\|\chi\|_{C^\alpha}$ such that
$$\forall t \geq 1, \quad   \Big| \int_0^1 e^{i t \phi(x)} \chi(x) d\mu(x) \Big| \leq C \Big( t^{-\rho} + \mu\big( x \in \mathbb{S}^1, \ \|\Phi(x)\| \leq t^{-\rho}  \big)\Big).$$
\end{lemma}

Notice that, $K$ being fixed and finite, the bound on $\|F(x)\|$ is independent on the choice of the norm used in $\mathbb{R}^K$, up to a multiplicative constant.

\begin{proof}
In the setting where $f$ is smoothly conjugated to a linear map, then by hypothesis, $\mu$ is either the SRB measure or the measure of maximal entropy, and these are actually the same measure. In this case there exists an analytic density $\rho$ such that $d\mu(x)=\rho(x)dx$ \cite{AST_1977__50__205_0}, %\href{https://www.numdam.org/item/AST_1977__50__205_0.pdf}{(this link)}
and the desired bound is a direct consequence of Van Der Corput Lemma (and our non-concentration Lemma \ref{lem:deri}). Now, in the setting where $f$ is not smoothly conjugated to a linear map, $\mu$ can be any equilibrium state. \\
In this setting, Fourier decay is known to hold \cite{SS20} (see also \cite{Le22,AHW23,BS23}). We will use the formulation of Lemma E.5 in \cite{Le25}, which gives a bound
$$ \Big| \int_0^1 e^{i t \phi(x)} \chi(x) d\mu(x) \Big| \leq C \Big( t^{-{\rho_1}} + \mu\big( x \in \mathbb{S}^1, \ \|\varphi'(x)\| \leq t^{-\rho_2} \big) \Big). $$ 
for some exponents $\rho_1,\rho_2 \in (0,1)$ that only depends on $\mu$, $\alpha$, and for $C$ that depends only on $K,\alpha,\|\phi\|_{C^2},\|\chi\|_{C^\alpha}$. The desired statement follows directly by Lemma \ref{lem:deri}, remembering that any norm in $\mathbb{R}^k$ is equivalent to the max norm $\|(x_1,\dots,x_k)\| := \max_j |x_j|$.
\end{proof}

We are ready to reduce Fourier decay to a non-concentration statement. Notice first the following easy observation.

\begin{proposition}
Suppose that $\mathcal{S}$ is a single graph: $\mathcal{S} = \{(x,\theta(x)), \ x \in \mathbb{S}^1 \}$ for some analytic and 1-periodic $\theta \in C^\omega(\mathbb{S}^1,\mathbb{R})$. Then there exists a constant $C>1$ and $\gamma \in (0,1)$ such that
$$ |\widehat{\omega}(k,\xi)| \leq C (1+|k|+|\xi|)^{-\gamma} \quad \text{ for all } (k,\xi) \in \mathbb{Z}\times \mathbb{R}$$
if and only if $\theta$ is not constant.
\end{proposition}

\begin{proof}
Suppose that $\mathcal{S}$ is a single graph, of an analytic function $\theta \in C^\omega(\mathbb{S}^1,\mathbb{R})$. Recall then that there exists probability measures $m_x$ supported on $\mathcal{S} \cap (\{x\} \times U) $ such that $\omega = \int_{\mathbb{S}^1} \delta_x \otimes m_x d\mu(x)$. Since $\mathcal{S} \cap (\{x\} \times U) = \{(x,\theta(x))\}$, we have $m_x = \delta_{(x,\theta(x))}$. It follows that $\omega$ is the pushforward of $\mu$ to the graph of $\theta$ under the map $x \mapsto (x,\theta(x))$. In particular:
$$ \widehat{\omega}(k,\xi) = \int_{\mathcal{S}} e^{2i \pi (x k + \xi y) } d\omega(x,y) = \int_{\mathbb{S}^1} e^{2i \pi (x k + \xi \theta(x)) } d\mu(x). $$
Now, if $\theta=\theta_0$ is constant, then $\widehat{\omega}(0,\xi) = e^{2 i \pi \xi \theta_0}$ doesn't decay to zero. Reciprocally, suppose that $\theta$ is not constant. We will conclude by applying the Van Der Corput Lemma in different "regimes": when $k$ dominates, and when $\xi$ dominates. Write $(k,\xi)=t(\alpha,\beta)$ with $\alpha^2+\beta^2=1$. If $|\beta|\leq \beta_0$ with $\beta_0$ small enough, we have for all $x \in \mathbb{S}^1$,
$$ |\partial_x(\alpha x + \beta \theta(x))| \geq 1/2 .$$
In this case, Van Der Corput Lemma \ref{lem:vdc-gibbs} directly yields
$$ |\widehat{\omega}(k,\xi)| = \Big| \int_{\mathbb{S}^1} e^{2 i \pi t (\alpha x + \beta \theta(x))} d\mu(x) \Big| \lesssim t^{-\rho} + \mu(x, |\alpha+\beta \theta'(x)| \leq t^{-\rho}  ) =   C t^{-\rho}$$
for $t$ large enough. In the opposite case where $\beta \geq \beta_0 >0$, applying the Van Der Corput Lemma \ref{lem:vdc-gibbs} again, we find
$$ |\widehat{\omega}(k,\xi)| \leq C \Big( t^{-\rho} + \mu(x \in \mathbb{S}^1, \ \max_{1 \leq j \leq K} |\partial_x^j(\alpha x + \beta \theta)(x)| \leq t^{-\rho} ) \Big) $$ $$ \leq C \Big( t^{-\rho} + \mu(x \in \mathbb{S}^1, \ \max_{2 \leq j \leq K} |\theta^{(j)}(x)| \leq \beta_0^{-1} t^{-\rho} ) \Big).$$
Now notice that for each $x \in \mathbb{S}^1$, there exists $k_x \geq 2$ such that $|\theta^{(k_x)}(x)| >0$ (Indeed, in the opposite case, $\theta$ would be affine and 1-periodic, so constant). By compactness of $\mathbb{S}^1$, it follows that there exists $\sigma_0>0$ and $K \geq 1$ such that $\inf_{x \in \mathbb{S}^1} \max_{2 \leq j \leq K} |\theta^{(j)}(x)| > \sigma_0$. It follows that, for $t$ large enough, we have
$|\widehat{\omega}(k,\xi)| \lesssim C t^{-\rho}$, and the proof is done.
\end{proof}

In the case where $\mathcal{S}$ is not a single graph, it is a superposition of graphs of multiple maps $\theta_\mathbf{a}$. Fourier decay is then established under a weak transversality condition. Before heading into the proof, we need a preliminary lemma.

\begin{lemma}{\label{lem:zoom}}
Let $\phi \in C^\omega(\mathbb{S}^1,\mathbb{R})$ be an analytic map. Let $u \in C^\infty([0,1],\mathbb{R})$ be a contraction satisfying $u' \in (0,1)$. Denote $\phi_u := \phi \circ u$. Let $k \geq 1$, and set $\Phi_k = (\phi^{(j)})_{1 \leq j \leq k}$, and $\Phi_{u,k} := (\phi_ u^{(j)})_{1 \leq j \leq k}$ There exists $C_k$, that depends only on $k$ and $\|u\|_{C^k}$, such that for all $x \in \mathbb{S}^1$:
$$ \| \Phi_k(u(x)) \| \leq \frac{C_k}{ |u'(x)|^{\frac{k(k+1)}{2}}} \|\Phi_{u,k}(x)\|. $$
\end{lemma}

\begin{proof}
The proof is done by induction on $k$, and correspond to bounding the norm of the inverse of an explicit triangular matrix. If $k=1$, then $\phi_{u}'=\phi'(u(x)) u'(x)$ and the bound holds choosing $C_1 := 1$. Now suppose the bound holds {for all $j < k$}, and let us establish it for $k$. Faa-Di-Bruno formula gives
$$ (\phi \circ u)^{(k)}(x) = \phi^{(k)}(u(x)) u'(x)^k + \sum_{j=1}^{k-1} \phi^{(j)}(u(x)) P_{k,j}(u'(x),u^{(2)}(x),\dots,u^{(k-j+1)}(x)) $$
where $P_{k,j}(x_1,\dots,x_{k-j+1})$ are universal polynomials. Denoting $$U_j := \|u\|_{C^j} \qquad\text{and}\qquad M_k := 1+\max_{1 \leq j < k}\|P_{k,j}\|_{L^\infty([-U_{k-j+1},U_{k-j+1}]^{k-j+1})} \geq 1,$$ we find
$$ |(\phi \circ u)^{(k)}(x) - \phi^{(k)}(u(x)) u'(x)^k| \leq \sum_{j=1}^{k-1} \| \Phi_{k-1}(u(x))\| \sup_{x} |P_{k,j}(u'(x),u^{(2)}(x),\dots,u^{(k-j+1)}(x))| $$
$$ \leq k M_k \| \Phi_{k-1}(u(x)) \|. $$
Hence, using the induction hypothesis and by definition of $\Phi_{u,k}$:
$$ |\phi^{(k)}(u(x))| \leq |u'(x)|^{-k} \Big( |(\phi \circ u)^{(k)}(x)| + k M_k \| \Phi_{k-1}(u(x)) \| \Big) $$
$$ \leq |u'(x)|^{-k}\Big(1+ \frac{kM_k C_{k-1}}{|u'(x)|^{k(k-1)/2}} \Big) \|\Phi_{u,k}(x) \|  $$
$$ \leq  \frac{{\|u\|_{C^1}^{\frac{k(k-1)}{2}}+k M_k C_{k-1}}}{|u'(x)|^{k(k+1)/2} } \cdot \| \Phi_{u,k}(x) \|.$$
The induction hypothesis then gives
$$ \| \Phi_{k}(u(x)) \| \leq C_k |u'(x)|^{-k(k+1)/2} \| \Phi_{u,k}(x) \| $$
where $C_{k} := \|u\|_{C^1}^{k(k-1)/2}+ k M_k C_{k-1} \geq C_{k-1}$. The sequence of constants $C_k$ only depends on $\|u\|_{C^k}$, which concludes the proof.  
\end{proof}

\begin{proposition}
Define for each $u \in \mathcal{A}^*$ and $x \in [0,1]$,  $\Theta_{\mathbf{u},K}(x) := (\theta_\mathbf{u}'(x),\dots,\theta_\mathbf{u}^{(K)}(x)) \in \mathbb{R}^K$. Suppose that there exists $K \geq 1$ such that the following holds: $$ \exists C_0 \geq 1, \forall \sigma >0, \forall n \geq 0, \  \sup_{x \in [0,1], t \in \mathbb{R}^K} \sum_{\mathbf{u} \in \mathcal{A}^n} p_\mathbf{u}(x) \mathbf{1}_{ \mathbf{B}(t,\sigma) }(\Theta_{\mathbf{u},K}(x)) \leq C_0(\sigma^{\gamma} + e^{-\alpha n}) .$$ Then there exists a constant $C>1$ and $\gamma \in (0,1)$ such that
$$ |\widehat{\omega}(k,\xi)| \leq C (1+|k|+|\xi|)^{-\gamma} \quad \text{ for all } (k,\xi) \in \mathbb{Z}\times \mathbb{R}.$$
\end{proposition}

\begin{proof}
Let us compute the Fourier transform of $\omega$. Applying Lemma \ref{eq:autosim} to $f(x,y) := e^{2i \pi (k x+\xi y)}$ gives
$$ \widehat{\omega}(k,\xi) = \sum_{\mathbf{u} \in \mathcal{A}^n} \int_{\mathbb{S}^1}   \int_{U} p_\mathbf{u}(x)  e^{2 i \pi (k x + \xi \theta_\mathbf{u}(x,y))} d m_x(y) d\mu(x) .$$

Now fix $(\alpha,\beta) \in [0,1]^2$ with $\alpha^2+\beta^2=1$ and write $(k,\xi) = t(\alpha,\beta)$ for some $t >0$. Fourier decay needs to be proved again in "two direction": when $\xi$ dominates, and when $k$ dominates. \\

Suppose first that $\beta \leq \varepsilon$ and $\alpha \geq (1-\varepsilon^2)^{1/2}$, for some $\varepsilon$ small enough to be fixed soon, and let us establish Fourier decay in this regime. Recall that $|\theta_\mathbf{u}(x,y) - \theta_{\mathbf{u}}(x,0)| \leq r_{max}^n$, so that
$$ \widehat{\omega}(k,\xi) = \sum_{\mathbf{u} \in \mathcal{A}^n} \int_{\mathbb{S}^1}   \int_{U} p_\mathbf{u}(x)  e^{2 i \pi (k x + \xi \theta_\mathbf{u}(x,0))} d m_x(y) d\mu(x) + \mathcal{O}(t \ r_{min}^n) $$
$$ =  \sum_{\mathbf{u} \in \mathcal{A}^n} p_\mathbf{u}(0) \int_{0}^1 e^{2 i \pi t \big(\alpha x + \beta \theta_{\mathbf{u}}(x,0)\big)} \Big(\frac{p_\mathbf{u}(x)}{p_\mathbf{u}(0)}\Big)d\mu(x) + \mathcal{O}(t \ r_{min}^n). $$
Now, the phases $\varphi_{\mathbf{u}}(x) := \alpha x + \beta \theta_{\mathbf{u}}(x,0)$ are analytic, and satisfy for all $\mathbf{u} \in \mathcal{A}^*$
$$ |\varphi_{\mathbf{u}}'(x)| = |\alpha + \beta \ \partial_x \theta_{\mathbf{u}}(x,0)| \geq 1/2 $$ if $\varepsilon>0$ is chosen small enough, by Lemma \ref{lem:uniform-bound-k-derivative}. The constant $\varepsilon$ is then fixed until the end of the proof so that the previous bound holds. Furthermore, the maps $p_\mathbf{u}/p_\mathbf{u}(0)$ are uniformly Hölder. The Van Der Corput Lemma \ref{lem:vdc-gibbs} then yields, uniformly in $\mathbf{u} \in \mathcal{A}^n$:
$$ \Big| \int_{\mathbb{S}^1} e^{2 i \pi t \varphi_\mathbf{u}(x)} \frac{p_\mathbf{u}(x)}{p_\mathbf{u}(0)} d\mu(x) \Big| \leq C t^{-\rho}. $$
Averaging over $\mathbf{u} \in \mathcal{A}^n$ with weights $p_\mathbf{u}(0)$, and choosing $n=n(t)$ so that $r_{min}^n \simeq t^2$, yields $\widehat{\omega}(k,\xi) \leq C t^{-\rho}$ in the regime where $\beta \leq \varepsilon$. \\

Let us now turn into the regime $\beta \geq \varepsilon$. We still write $(k,\xi) = t(\alpha,\beta)$. We start by canceling the linear term, using Cauchy-Schwarz as follow:
$$ | \widehat{\omega}(k,\xi) |^2 = \Big| \int_{\mathbb{S}^1} e^{2 i \pi k x}  \int_{U} \sum_{\mathbf{u} \in \mathcal{A}^n} p_\mathbf{u}(x)  e^{2 i \pi \xi \theta_\mathbf{u}(x,y)} d m_x(y) d\mu(x)\Big|^2  $$
$$ = \Big| \int_{\mathbb{S}^1} e^{2 i \pi k x}  \sum_{\mathbf{u} \in \mathcal{A}^n} p_\mathbf{u}(x)  e^{2 i \pi \xi \theta_\mathbf{u}(x,0)} d\mu(x)\Big|^2 + \mathcal{O}\big((t \ r_{max}^n)^2\big) $$
$$ \leq \int_{\mathbb{S}^1} \Big| \sum_{\mathbf{u} \in \mathcal{A}^n} p_\mathbf{u}(x) e^{2 i \pi \xi \theta_{\mathbf{u}}(x,0)} \Big|^2 d\mu(x) + \mathcal{O}\big((t \ r_{max}^n)^2\big)$$
$$ =  \sum_{\mathbf{u},\mathbf{v} \in \mathcal{A}^n} p_\mathbf{u}(0) p_\mathbf{v}(0) \int_{0}^1  e^{2 i \pi \beta t \Delta_{\mathbf{u},\mathbf{v}}(x) } \chi_{\mathbf{u},\mathbf{v}}(x) d\mu(x) + \mathcal{O}\big((t \ r_{max}^n)^2\big)$$
where $\chi_{\mathbf{u},\mathbf{v}}(x) = \frac{p_\mathbf{u}(x) p_\mathbf{v}(x)}{p_\mathbf{u}(0) p_\mathbf{v}(0)}$ (are uniformly Hölder maps) and $\Delta_{\mathbf{u},\mathbf{v}}(x) := \theta_\mathbf{u}(x,0)-\theta_\mathbf{v}(x,0)$ (are uniformly analytic maps).  \\

Recall that the measure $\mu$ satisfies the following autosimilarity formula, for any continuous function $h \in C^0(\mathbb{S}^1,\mathbb{R})$:
$$ \int h d\mu = \sum_{\mathbf{a} \in \mathcal{A}^m} \int_{\mathbb{S}^1} h(\mathbf{a}[x]) p_\mathbf{a}(x) d\mu(x).$$

In particular, fixing $\mathbf{u},\mathbf{v}$, we find, for all $m \geq 1$:
$$ \int_{\mathbb{S}^1} e^{2 i \pi \beta t \Delta_{\mathbf{u,v}}(x)} \chi_{\mathbf{u},\mathbf{v}}(x) d\mu(x) = \sum_{\mathbf{w} \in \mathcal{A}^m} p_\mathbf{w}(0) \int_{\mathbf{S}^1} e^{2 i \pi \beta t \overline{\Delta}_{\mathbf{u,v,w}}(x)} \chi_{\mathbf{u,v,w}}(x) dx,$$
where $\overline{\Delta}_{\mathbf{u,v,w}}(x) := \Delta_{\mathbf{u},\mathbf{v}}(\mathbf{w}[x])$ (uniformly analytic in $\mathbf{u,v,w}$), and $\chi_{\mathbf{u,v,w}}(x) := \chi_{\mathbf{u,v}}(\mathbf{w}[x]) \frac{p_\mathbf{w}(x)}{p_\mathbf{w}(0)}$ (uniformly Hölder in \textbf{u,v,w}). Van Der Corput Lemma \ref{lem:vdc-gibbs} then yields, uniformly in $\mathbf{u,v,w}$:
$$ \Big| \int_{\mathbf{S}^1} e^{2 i \pi \beta t \overline{\Delta}_{\mathbf{u,v,w}}(x)} \chi_{\mathbf{u,v,w}}(x) d\mu(x) \Big| \leq C\Big( t^{-\rho} + \mu(x \in \mathbb{S}^1\mid \ \max_{1 \leq j \leq K}| \overline{\Delta}_{\mathbf{u,v,w}}^{(j)}(x) | \leq t^{-\rho}) \Big). $$ 
Now, Lemma \ref{lem:zoom} ensure that if $ \max_{1 \leq j \leq K}| \overline{\Delta}_{\mathbf{u,v,w}}^{(j)}(x) | \leq t^{-\rho}$, then 
$ \max_{1 \leq j \leq K}| \Delta_{\mathbf{u,v}}^{(j)}(\mathbf{w}[x]) | \leq C_K r_{\text{min}}^{-m k^2} t^{-\rho}$.  Notice that the constant $C_K$ is uniform in $\mathbf{w}$ thanks to uniform bounds of the $C^k$ norms of $\mathbf{w}[\cdot]$ given by Lemma \ref{lem:uniform-bound-k-derivative}. We will now fix $m(t)$ such that $r_{min}^{m k^2} \simeq t^{-\rho/2}$. We then find:
$$ \Big| \int_{\mathbb{S}^1} e^{2 i \pi \beta t \overline{\Delta}_{\mathbf{u,v,w}}(x)} \chi_{\mathbf{u,v,w}}(x) d\mu(x) \Big| \leq C\Big( t^{-\rho} + \mu(x \in \mathbb{S}^1, \ \max_{1 \leq j \leq K}| \overline{\Delta}_{\mathbf{u,v,w}}^{(j)}(x) | \leq t^{-\rho}) \Big)$$
$$\leq C\Big( t^{-\rho} + \mu(x \in \mathbb{S}^1, \ \max_{1 \leq j \leq K}| {\Delta}_{\mathbf{u,v}}^{(j)}(\mathbf{w}[x]) | \leq C_K t^{-\rho/2}) \Big)  $$
$$ \leq C\Big( t^{-\rho} + \mu(x \in \mathbb{S}^1, \ \|\Theta_{\mathbf{u},K}(\mathbf{w}[x]) - \Theta_{\mathbf{v},K}(\mathbf{w}[x]) \| \leq C_K t^{-\rho/2}) \Big) $$

We average all of these bounds in $\mathbf{u,v,w}$ to find, using the Tree bound:

$$| \widehat{\omega}(k,\xi) |^2 \leq  \Big|\underset{\mathbf{w} \in \mathcal{A}^m}{\sum_{\mathbf{u},\mathbf{v} \in \mathcal{A}^n}} p_\mathbf{u}(0) p_\mathbf{v}(0) p_\mathbf{w}(0) \int_{\mathbb{S}^1}  e^{2 i \pi \beta t \Delta_{\mathbf{u},\mathbf{v},\mathbf{w}}(x) } \chi_{\mathbf{u},\mathbf{v},\mathbf{w}}(x) d\mu(x) \Big| + \mathcal{O}\big((t \ r_{max}^n)^2\big)$$
$$ \lesssim \underset{\mathbf{w} \in \mathcal{A}^m}{\sum_{\mathbf{u},\mathbf{v} \in \mathcal{A}^n}} p_\mathbf{u}(0) p_\mathbf{v}(0) p_\mathbf{w}(0) \mu(x \in \mathbb{S}^1, \ \|\Theta_{\mathbf{u},K}(\mathbf{w}[x]) - \Theta_{\mathbf{v},K}(\mathbf{w}[x]) \| \leq C_K t^{-\rho/2}) + \mathcal{O}\big((t \ r_{max}^n)^2 + t^{-\rho}\big) $$
$$ = \underset{\mathbf{w} \in \mathcal{A}^m}{\sum_{\mathbf{v} \in \mathcal{A}^n}} p_\mathbf{v}(0) p_\mathbf{w}(0) \int_{\mathbb{S}^1} \Big(\sum_{\mathbf{u} \in \mathcal{A}^n} p_\mathbf{u}(0) \mathbf{1}_{\mathbf{B}(\Theta_{\mathbf{v},K}(\mathbf{w}[x]),C_K t^{-\rho/2})}( \Theta_{\mathbf{u},K}(\mathbf{w}[x]) ) \Big) d\mu(x) + \mathcal{O}\big((t \ r_{max}^n)^2 + t^{-\rho}\big) $$
$$ \leq \mathcal{O}(t^{-\rho \gamma/2} + e^{-\alpha n} + t r_{max}^n + t^{-\rho}). $$
Letting $n \rightarrow \infty$ yields the polynomial bound $\mathcal{O}(t^{-\rho \gamma/2})$, which concludes the proof.
\end{proof}

The rest of the paper is devoted to the proof of the non-concentration estimate we need on $\Theta_{\mathbf{a},K}$. We will prove it under the condition that $\mathcal{S}$ is not a single graph.

\section{Rigidity of unstable curves}
\label{sec:Rigidity}
In this section, we establish several preliminary rigidity properties of the unstable
curves $\{(x,\theta_{\mathbf{a}}(x))\}_{x\in [0,1]}$.

\begin{lemma}\label{lem:rigidity-derivatives}
$\mathcal{S}$ is a single graph if and only if for all $\mathbf{a},\mathbf{b} \in \mathcal{A}^\infty$, $\theta_\mathbf{a}'\equiv\theta_{\mathbf{b}}'$ .
\end{lemma}

\begin{proof}
We only prove the sufficiency, since the necessity is immediate. Suppose that $\theta_\mathbf{a}'\equiv\theta_{\mathbf{b}}'$ for all $\mathbf{a},\mathbf{b}\in \mathcal{A}^\infty$. 

Let $\mathbf{0}:=00\cdots\in\mathcal{A}^{\infty}$.
Let $\Delta:= \theta_{\mathbf{0}+1}(0)-\theta_{\mathbf{0}}(0)$.
By Lemma~\ref{lem:addition-theta_u-formal} and $\theta_{\mathbf{0}+1}'\equiv\theta_{\mathbf{0}}'$, we have
$$ \theta_{\mathbf{0}+2}(0)-\theta_{\mathbf{0}+1}(0)=\theta_{\mathbf{0}+1}(1)-\theta_{\mathbf{0}}(1) =\theta_{\mathbf{0}+1}(0)-\theta_{\mathbf{0}}(0)=\Delta. $$
By induction, for every $n\in\Z_{>0}$, we have
$$  \theta_{\mathbf{0}+n+1}(0)-\theta_{\mathbf{0}+n}(0)=\theta_{\mathbf{0}+n}(1)-\theta_{\mathbf{0}+n-1}(1)=\theta_{\mathbf{0}+n}(0)-\theta_{\mathbf{0}+n-1}(0) =\Delta.  $$
Summing these identities over $0,\ldots,n-1$, we obtain
\begin{equation}\label{eq:unstable_curve-degeneration-x=0}
    \theta_{\mathbf{0}+n}(0)-\theta_{\mathbf{0}}(0)=n\Delta. 
\end{equation}
Since $\theta_{\mathbf{0}+n}'\equiv\theta_{\mathbf{0}}'$ , we obtain 
\begin{equation}\label{eq:unstable_curve-degeneration}
    \theta_{\mathbf{0}+n}(x)\equiv  \theta_{\mathbf{0}}(x)+n\Delta. 
\end{equation}
 Since 
$|\theta_{\mathbf{0}+n}(0)|\le 1$ for all $n\in\Z$, by \eqref{eq:unstable_curve-degeneration-x=0}, we have
$\Delta=0$.

Therefore \eqref{eq:unstable_curve-degeneration} implies that
\begin{equation}\label{eq:rigidity-derivatives-1}
     \theta_{\mathbf{0}+n}(x)\equiv  \theta_{\mathbf{0}}(x) \qquad\text{for all }n\in\N.
\end{equation}
 Since the set $\{\mathbf{0}+n\}_{n\in\N}$ is a dense subset of $\mathcal{A}^{\infty}$, then  the set
$$\{(x,\theta_\mathbf{\mathbf{0}+n}(x)): x\in [0,1),n\in\mathbb{N}\}$$
is a dense subset of $\mathcal{S}$. Therefore
by \eqref{eq:rigidity-derivatives-1}, $\mathcal{S}$ is a single graph.
\end{proof}

Given $\mathbf{a},\mathbf{b}\in\mathcal{A}^{*}\cup\mathcal{A}^{\infty}$, define the function 
$$  \Delta_{\mathbf{a},\mathbf{b}}(x):=\left(\theta_{\mathbf{a}}-\theta_{\mathbf{b}}\right)(x)\qquad\text{for }x\in\R.$$

\begin{lemma}\label{lem:Dleta-nonzero-formal}
Assume that $\mathcal{S}$ is not a single graph. Then
   for all $x\in[0,1]$ and $\mathbf{w}\in\mathcal{A}^*$, there exists  $\mathbf{a},\mathbf{b}\in\mathcal{A}^{\infty}$ such that
 $ \Delta_{\mathbf{w}\mathbf{a},\mathbf{w}\mathbf{b}}(x)\neq 0. $
\end{lemma}

\begin{proof}
Assume that our claim fails. Then there exists $x_0\in [0,1]$ and $\mathbf{w}\in\mathcal{A}^*$ such that
$$\Delta_{\mathbf{w}\mathbf{a},\mathbf{w}\mathbf{b}}(x_0)=0\qquad\text{for all }\mathbf{a},\mathbf{b}\in\mathcal{A}^{\infty}.$$

Fix $\mathbf{a},\mathbf{b}\in\mathcal{A}^{\infty}$, and let
$\mathbf{u}\in\mathcal{A}^*$. By
\eqref{eq:theta-concatenation} and
the mean value theorem, we have
\begin{equation}\label{eq:Dleta-nonzero-1}
\begin{split}   |\Delta_{\mathbf{w}\mathbf{u}\mathbf{a},\mathbf{w}\mathbf{u}\mathbf{b}}(x_0)|&= |\theta_{\mathbf{w}\mathbf{u}\mathbf{a}}(x_0)-\theta_{\mathbf{w}\mathbf{u}\mathbf{b}}(x_0)|\\&=
\left|\theta_{\mathbf{w}\mathbf{u}}(x_0,\theta_{\mathbf{a}}(\mathbf{w}\mathbf{u}[x_0])  )- \theta_{\mathbf{w}\mathbf{u}}(x_0,\theta_{\mathbf{b}}(\mathbf{w}\mathbf{u}[x_0])  )    \right|\\&=
|\partial_y\theta_{\mathbf{w}\mathbf{u}}(x_0,s)|\cdot |  \theta_{\mathbf{a}}(\mathbf{w}\mathbf{u}[x_0])-\theta_{\mathbf{b}}(\mathbf{w}\mathbf{u}[x_0])  |
\end{split} 
\end{equation}
for some $s\in U$. By 
\eqref{eq:partial-G-formula-Formal}, we have
$|\partial_y\theta_{\mathbf{w}\mathbf{u}}(x_0,s)|>0$, Thus we have
      $$\Delta_{\mathbf{a},\mathbf{b}}(\mathbf{w}\mathbf{u}[x_0])=0\qquad\text{for all }\mathbf{a},\mathbf{b}\in\mathcal{A}^{\infty}\text{ and  }\mathbf{u}\in\mathcal{A}^*.$$
    By $\mathbf{w}\mathbf{u}[x_0]=\mathbf{u}[\mathbf{w}[x_0]] $, the set $\{\mathbf{w}\mathbf{u}[x_0]\}_{\mathbf{u}\in\mathcal{A}^*}$ is a dense set on $[0,1]$. Thus we have
      $$\Delta_{\mathbf{a},\mathbf{b}}(x)\equiv 0\qquad\text{for all }\mathbf{a},\mathbf{b}\in\mathcal{A}^{\infty},$$
      which contradicts with the condition that $\mathcal{S}$ does not consist of a single graph. Thus our assumption fails and the lemma holds.
\end{proof}

{
The following lemma shows that, if  $\mathcal{S}$ is not a single graph, there are infinitely many unstable curves.
}

\begin{lemma}\label{lem:curve-splitting-formal}
{
Suppose that $\mathcal{S}$ is not a single graph. Then there exist
$\mathbf{c}_1,\mathbf{c}_2\in\mathcal{A}^*$ such that, for every distinct
$\mathbf{u}=u_1u_2\cdots,\mathbf{v}=v_1v_2\cdots\in\{1,2\}^{\infty}$,
}
 we have $$\theta_{w(\mathbf{u})}(x)\not\equiv\theta_{w(\mathbf{v})}(x),$$
   where $ w(\mathbf{u}):=\mathbf{c}_{u_1}\mathbf{c}_{u_2}\cdots, w(\mathbf{v}):=\mathbf{c}_{v_1}\mathbf{c}_{v_2}\cdots \in\mathcal{A}^{\infty}$.
\end{lemma}

\begin{proof}
Since $\mathcal{S}$ does not consist of a single graph, there exist $\mathbf{a},\mathbf{b}\in\mathcal{A}^{\infty}$ such that $\theta_{\mathbf{a}}(x)\not\equiv\theta_{\mathbf{b}}(x)$.
{Thus applying Corollary \ref{cor:uniform-holomorphic-extensions},} there exist $n\in\Z_{>0}$ such that
\begin{equation}\label{eq:start-defferience-formal}
    \theta_{(\mathbf{a}|_n)\mathbf{w}}(x)\not\equiv\theta_{(\mathbf{b}|_n)\mathbf{t}}(x)\qquad\text{ for all }\mathbf{w},\mathbf{t}\in\mathcal{A}^{\infty}.
\end{equation}
Set $\mathbf{c}_1:=\mathbf{a}|_n$ and $\mathbf{c}_2:=\mathbf{b}|_n$. 

Assume that there exists two different symbols $\mathbf{u},\mathbf{v}\in\{1,2\}^{\infty}$ such that $\theta_{w(\mathbf{u})}(x)\equiv\theta_{w(\mathbf{v})}(x)$.
Without loss generality, suppose  that
$$ w(\mathbf{u})= \mathbf{z}\mathbf{c}_1\mathbf{e}\qquad
\text{and}\qquad w(\mathbf{v})= \mathbf{z}\mathbf{c}_2\mathbf{r}$$
for some  $\mathbf{z}\in \mathcal{A}^*$
and $\mathbf{e},\mathbf{r}\in \mathcal{A}^{\infty}$.
Then $\theta_{w(\mathbf{u})}(x)\equiv\theta_{w(\mathbf{v})}(x)$ implies that
$$ \theta_{\mathbf{z}}(x,\theta_{\mathbf{c}_1\mathbf{e}}(\mathbf{z}[x]) )\equiv \theta_{\mathbf{z}}(x,\theta_{\mathbf{c}_2\mathbf{r}}(\mathbf{z}[x]) ).  $$
Combining this with $|\partial_y\theta_{\mathbf{z}}|>0$ and the mean value theorem, we have
$$ \theta_{\mathbf{c}_1\mathbf{e}}(\mathbf{z}[x]) \equiv \theta_{\mathbf{c}_2\mathbf{r}}(\mathbf{z}[x]).$$
Since the functions $ \theta_{\mathbf{c}_1\mathbf{e}}(x),\theta_{\mathbf{c}_2\mathbf{r}}(x)$ are real analytic, we have
$$ \theta_{\mathbf{c}_1\mathbf{e}}(x) \equiv \theta_{\mathbf{c}_2\mathbf{r}}(x).$$
This contradicts with \eqref{eq:start-defferience-formal}. Thus the assumption fails and this lemma holds.
\end{proof}

In the remainder of this section, we establish several quantitative rigidity results.

\begin{lemma}\label{lem:lower-bound-Delta-formal}
     Assume that $\mathcal{S}$ is not a single graph.  Then, for every $N\in\mathbb{Z}_{>0}$, there exist
     $\delta_N\in(0,1)$ and $M_N\in\mathbb{Z}_{>0}$ such that the following holds.
     
    For every $\mathbf{w}\in\mathcal{A}^N$,  $n\geq M_N$,
    $\mathbf{u}\in\mathcal{A}^*$ with $|\mathbf{u}|\geq M_N$, there exist
    $\mathbf{a},\mathbf{b}\in\mathcal{A}^n$ such that
     $$ |\theta_{\mathbf{w}\mathbf{a}\mathbf{c}}(\mathbf{u}[x])-\theta_{\mathbf{w}\mathbf{b}\mathbf{d}}(\mathbf{u}[x])|\ge\delta_N \qquad\text{for all }\mathbf{c},\mathbf{d}\in\mathcal{A}^*\text{ and }x\in [0,1].$$
\end{lemma}

\begin{proof}
Fix $N\in\mathbb{Z}_{>0}$ and $\mathbf{w}\in\mathcal{A}^N$.
For each $x\in [0,1]$, by Lemma \ref{lem:Dleta-nonzero-formal}, there exists
$\mathbf{a}_x,\mathbf{b}_x\in\mathcal{A}^{\infty}$ such that $$\delta_x:=|\theta_{\mathbf{w}\mathbf{a}_x}(x)- \theta_{\mathbf{w}\mathbf{b}_x}(x)|>0.$$
Thus there exists $r_x>0$ such that for each $x'\in \mathbf{B}(x,r_x)$, we have
\begin{equation}\label{eq:lower-bound-Delta-formal-1}
    |\theta_{\mathbf{w}\mathbf{a}_x}(x')- \theta_{\mathbf{w}\mathbf{b}_x}(x')|\ge \frac{\delta_x}2.
\end{equation}
By the compactness of $[0,1]$, there exist finitely many open balls
$
\mathbf{B}(x_1,r_{x_1}),\ldots,\mathbf{B}(x_k,r_{x_k})
$
that cover $[0,1]$.
Consequently, there exists $M_{\mathbf{w}}\in\mathbb{N}$ such that, for every
$\mathbf{u}\in\mathcal{A}^*$ with $|\mathbf{u}|\geq M_{\mathbf{w}}$, there exists $t=t(\mathbf{w},\mathbf{u})\in\{1,\ldots,k\}$ such that
$
\mathbf{u}[x]\in \mathbf{B}(x_t,r_{x_t})
$
for all
$x\in[0,1]$.

Set $\delta_{\mathbf{w}}:=\frac14\min_{1\leq t\leq k}\delta_{x_t}. $
Increasing $M_{\mathbf{w}}$ if necessary, we may assume that $4(r_{\max})^{M_{\mathbf{w}}} \leq \frac{\delta_{\mathbf{w}}}{4}.$
Let $\mathbf{u}\in\mathcal{A}^*$ with $|\mathbf{u}|\geq M_{\mathbf{w}}$, let $t=t(\mathbf{w},\mathbf{u})$, and set $\mathbf{a}:=(\mathbf{a}_{x_t})|_{M_{\mathbf{w}}}, \mathbf{b}:=(\mathbf{b}_{x_t})|_{M_{\mathbf{w}}}. $ 
Then, by \eqref{eq:theta-common-prefix-distance} and \eqref{eq:lower-bound-Delta-formal-1}, for all $\mathbf{c},\mathbf{d}\in\mathcal{A}^*$ and $x\in[0,1]$,
\begin{equation}\nonumber
\begin{split}
&\left|
\theta_{\mathbf{w}\mathbf{a}\mathbf{c}}(\mathbf{u}[x])
-
\theta_{\mathbf{w}\mathbf{b}\mathbf{d}}(\mathbf{u}[x])
\right|
\\
&\geq
\left|
\theta_{\mathbf{w}\mathbf{a}_{x_t}}(\mathbf{u}[x])
-
\theta_{\mathbf{w}\mathbf{b}_{x_t}}(\mathbf{u}[x])
\right|
\\
&\quad
-
\left|
\theta_{\mathbf{w}\mathbf{a}\mathbf{c}}(\mathbf{u}[x])
-
\theta_{\mathbf{w}\mathbf{a}_{x_t}}(\mathbf{u}[x])
\right|
-
\left|
\theta_{\mathbf{w}\mathbf{b}\mathbf{d}}(\mathbf{u}[x])
-
\theta_{\mathbf{w}\mathbf{b}_{x_t}}(\mathbf{u}[x])
\right|
\\
&\geq
\frac{\delta_{x_t}}{2}
-
4(r_{\max})^{M_{\mathbf{w}}}
\\
&\geq
\delta_{\mathbf{w}}.
\end{split}
\end{equation}

Thus, the lemma holds with $M_N:=\max_{\mathbf{w}\in\mathcal{A}^N }M_{\mathbf{w}} \text{ and }\delta_N:=\min_{\mathbf{w}\in\mathcal{A}^N }\delta_{\mathbf{w}}.$
\end{proof}

\begin{lemma}\label{lem:G_a^(k)-G_b^(k)-lower-bound-formal}
  Assume that $\mathcal{S}$ is not a single graph. 
  There exist $\delta>0$ and $N,K\in\Z_{>0}$ such that the following holds. For each $n\ge N$  and $x\in [0,1]$, there exists $\mathbf{a},\mathbf{b}\in\mathcal{A}^n$
  and $k\in\{1,2,\ldots,K\}$
  such that 
  $$ \left| (\theta_{\mathbf{a}\mathbf{c}}-\theta_{\mathbf{b}\mathbf{d}})^{(k)}(x)  \right|\ge\delta\qquad\text{for all } \mathbf{c},\mathbf{d}\in\mathcal{A}^*. $$    
\end{lemma}

\begin{proof}
 Since $\mathcal{S}$ does not consist of a single graph, by Lemma \ref{lem:rigidity-derivatives},  there exists $\mathbf{u},\mathbf{v}\in\mathcal{A}^{\infty}$ with $\theta'_{\mathbf{u}}(x)\not\equiv\theta'_{\mathbf{v}}(x).$
  It suffices to prove the following claim. 
  
  \begin{claim}\label{claim:lower-bound-k-derivative}
      There exist $\delta>0$ and $K\in\Z_{>0}$ such that the following holds. For each $x\in[0,1]$, there exists $k\in\{1,\ldots,K\}$ such that
\begin{equation}\label{eq:theta^(k)_u-lower-bound-formal}
       \left| (\theta_{\mathbf{u}}-\theta_{\mathbf{v}})^{(k)}(x)  \right|\ge2\delta.
   \end{equation}
  \end{claim}

Indeed, assume that the claim holds. By Lemma \ref{lem:uniform-bound-k-derivative}, we may choose $N>1$ sufficiently large such that, for every $n\geq N$, with setting $\mathbf{a}:=\mathbf{u}|_n$ and
$\mathbf{b}:=\mathbf{v}|_n$, the following holds: for every $x\in[0,1]$, every $k\in\{1,\ldots,K\}$, and all $\mathbf{c},\mathbf{d}\in\mathcal{A}^*$,
   $$\left|( \theta_{\mathbf{a}\mathbf{c}}- \theta_{\mathbf{u}})^{(k)}(x)\right|  \le\frac{\delta}2\qquad\text{and}\qquad    \left|( \theta_{\mathbf{b}\mathbf{d}}- \theta_{\mathbf{u}})^{(k)}(x)\right|  \le\frac{\delta}2.$$
   Combining this with \eqref{eq:theta^(k)_u-lower-bound-formal}, we obtain the desired conclusion.

   We now prove Claim \ref{claim:lower-bound-k-derivative}. Suppose that the claim fails. Then, for each $n\in\mathbb{Z}_{>0}$, there exists $x_n\in[0,1]$ such that, for each $k\in\{1,\ldots,n\}$, we have
    $$  \left| (\theta_{\mathbf{u}}-\theta_{\mathbf{v}})^{(k)}(x_n)  \right|\le \frac1n. $$
  Without loss of generality, by passing to a subsequence, we may assume that $x_n\to x_\infty$. Therefore, we get
  $$  (\theta_{\mathbf{u}}-\theta_{\mathbf{v}})^{(k)}(x_\infty)=0\qquad\text{for all }k\ge1.  $$
  Since $ (\theta_{\mathbf{u}}-\theta_{\mathbf{v}})(x)$ is real analytic function on $[0,1]$, we have $(\theta_{\mathbf{u}}-\theta_{\mathbf{v}})'(x)\equiv0$, a contradiction. 
  
  Therefore, our assumption is false, and Claim \ref{claim:lower-bound-k-derivative} follows.
\end{proof}

\section{Renormalization}
\label{Renormalization}

To prove the Tree Theorem, we will  establish a multiscale version of the separation Lemma \ref{lem:G_a^(k)-G_b^(k)-lower-bound-formal} . To do so, we introduce a normalization of $\frac{\Delta_{\mathbf{u}\mathbf{a},\mathbf{u}\mathbf{b}}(x)}{\Delta_{\mathbf{a},\mathbf{b}}(\mathbf{u}[x])}$ that allows us to control the effect of the common prefix $\mathbf{u}$, see Proposition \ref{pro:approximate-H-formal} which is  the main result of this section.

\subsection{The normalization}\label{sub:def-normalization}
 Let
$\mathbf{u},\mathbf{a},\mathbf{b}\in\mathcal{A}^*$. 
We aim to establish a connection between $ \Delta_{\mathbf{u}\mathbf{a},\mathbf{u}\mathbf{b}}$  and $\Delta_{\mathbf{a},\mathbf{b}}$. To this end, if $\Delta_{\mathbf{a},\mathbf{b}}(x)\not\equiv 0$,
we define
$$ \Lambda_{\mathbf{u};\mathbf{a},\mathbf{b}}(x):= 
\frac{\Delta_{\mathbf{u}\mathbf{a},\mathbf{u}\mathbf{b}}(x)}{\Delta_{\mathbf{a},\mathbf{b}}(\mathbf{u}[x])}\qquad\text{for }x\in\R. $$
This is well defined since the function
$$ \frac{\Delta_{\mathbf{u}\mathbf{a},\mathbf{u}\mathbf{b}}(x)}{\Delta_{\mathbf{a},\mathbf{b}}(\mathbf{u}[x])}= \frac{\theta_{\mathbf{u}\mathbf{a}}(x)- \theta_{\mathbf{u}\mathbf{a}}(x) }{\theta_{\mathbf{a}}(\mathbf{u}[x])- \theta_{\mathbf{b}}(\mathbf{u}[x])}=\frac{\theta_{\mathbf{u}}(x,\theta_{\mathbf{a}}( \mathbf{u}[x]))- \theta_{\mathbf{u}}(x,\theta_{\mathbf{b}}( \mathbf{u}[x])) }{\theta_{\mathbf{a}}(\mathbf{u}[x])- \theta_{\mathbf{b}}(\mathbf{u}[x])} $$
admits a holomorphic extension to $\R$. Furthermore, by the uniqueness of analytic continuation, 
\begin{equation}\label{def:Lambda-u-a,b-formal}
\Lambda_{\mathbf{u};\mathbf{a},\mathbf{b}}(x)\equiv\int_0^1\partial_y\theta_{\mathbf{u}}\left(x,  s\theta_{\mathbf{a}}(\mathbf{u}[x])+(1-s) \theta_{\mathbf{b}}(\mathbf{u}[x])\right)\,ds.
\end{equation}
Therefore, if $\Delta_{\mathbf{a},\mathbf{b}}(x)\equiv 0$ holds,
we define
$$  \Lambda_{\mathbf{u};\mathbf{a},\mathbf{b}}(x):=\partial_y\theta_{\mathbf{u}}\left(x,  \theta_{\mathbf{a}}(\mathbf{u}[x])\right)=\partial_y\theta_{\mathbf{u}}\left(x,  \theta_{\mathbf{b}}(\mathbf{u}[x])\right).$$
In all cases, we have
\begin{equation} \label{def:Delta-u-a,b-formal}
\Delta_{\mathbf{u}\mathbf{a},\mathbf{u}\mathbf{b}}(x)\equiv\Delta_{\mathbf{a},\mathbf{b}}(\mathbf{u}[x])\cdot\Lambda_{\mathbf{u};\mathbf{a},\mathbf{b}}(x).
\end{equation}

If $\Delta_{\mathbf{a},\mathbf{b}}(x)\not\equiv 0$, then we have
\begin{equation}\label{eq:Lmabda-decomposition-formal}
\begin{split}
\Lambda_{\mathbf{u};\mathbf{a},\mathbf{b}}(x)&\equiv\frac{\left( \theta_{\mathbf{u}\mathbf{a}}-\theta_{\mathbf{u}\mathbf{b}}\right)(x) }{\left( \theta_{\mathbf{a}}-\theta_{\mathbf{b}}\right)(\mathbf{u}[x]) }\\&\equiv
\prod_{k=1}^{|\mathbf{u}|}  \frac{\left( \theta_{(\mathbf{u}|_{k}^{|\mathbf{u}|})\mathbf{a}}-\theta_{(\mathbf{u}|_{k}^{|\mathbf{u}|})\mathbf{b}}\right)(( \mathbf{u}|_{k-1})[x])}{\left( \theta_{(\mathbf{u}|_{k+1}^{|\mathbf{u}|})\mathbf{a}}-\theta_{(\mathbf{u}|_{k+1}^{|\mathbf{u}|})\mathbf{b}}\right)(( \mathbf{u}|_{k})[x]) }
\\&\equiv\prod_{k=1}^{|\mathbf{u}|}\frac{
g\left( ( \mathbf{u}|_{k})[x],    \theta_{(\mathbf{u}|_{k+1}^{| \mathbf{u}|})\mathbf{a}}(( \mathbf{u}|_{k})[x])  \right)-g\left( ( \mathbf{u}|_{k})[x],    \theta_{(\mathbf{u}|_{k+1}^{| \mathbf{u}|})\mathbf{b}}(( \mathbf{u}|_{k})[x])  \right)
}{\theta_{(\mathbf{u}|_{k+1}^{| \mathbf{u}|})\mathbf{a}}(( \mathbf{u}|_{k})[x]) -\theta_{(\mathbf{u}|_{k+1}^{| \mathbf{u}|})\mathbf{b}}(( \mathbf{u}|_{k})[x])  }\\&
\equiv\prod_{k=1}^{|\mathbf{u}|}\int_0^1
\partial_yg\left(( \mathbf{u}|_{k})[x],
s\cdot\theta_{(\mathbf{u}|_{k+1}^{| \mathbf{u}|})\mathbf{a}}(( \mathbf{u}|_{k})[x])+(1-s)\cdot\theta_{(\mathbf{u}|_{k+1}^{| \mathbf{u}|})\mathbf{b}}(( \mathbf{u}|_{k})[x])
\right)\,ds
\end{split}
\end{equation}
where, in the third equality, we used the following identity:
\begin{equation}\nonumber
    \begin{split}        \theta_{(\mathbf{u}|_{k}^{|\mathbf{u}|})\mathbf{a}}(( \mathbf{u}|_{k-1})[x])&=\theta_{u_k(\mathbf{u}|_{k+1}^{|\mathbf{u}|})\mathbf{a}}(( \mathbf{u}|_{k-1})[x])\\&=
    \theta_{u_k}\left( ( \mathbf{u}|_{k-1})[x],  \theta_{(\mathbf{u}|_{k+1}^{|\mathbf{u}|})\mathbf{a}}(( \mathbf{u}|_{k})[x])  \right)\\&=
     g\left( ( \mathbf{u}|_{k})[x],  \theta_{(\mathbf{u}|_{k+1}^{|\mathbf{u}|})\mathbf{a}}(( \mathbf{u}|_{k})[x])  \right).
    \end{split}
\end{equation}

Given $\mathbf{w}\in\mathcal{A}^{\infty}$ and $n,m\in\Z_{>0}$, we have:
\begin{equation}\label{eq:Lmabda-decomposition-true}
\begin{split}
    &\Lambda_{\mathbf{w}|_n;(\mathbf{w}|_{n+1}^{n+m})\mathbf{a},(\mathbf{w}|_{n+1}^{n+m})\mathbf{b}}(x)\\&\equiv
    \prod_{k=1}^{n}\int_0^1
\partial_yg\left(( \mathbf{w}|_{k})[x],
s\cdot\theta_{(\mathbf{w}|_{k+1}^{m+n})\mathbf{a}}(( \mathbf{w}|_{k})[x])+(1-s)\cdot\theta_{(\mathbf{w}|_{k+1}^{m+n})\mathbf{b}}(( \mathbf{w}|_{k})[x])
\right)\,ds.
\end{split}
\end{equation}

After normalizing by the value at $x=0$ and letting both $n$ and $m$ tend to infinity, we  naturally  consider the following function:

\begin{equation}\label{def:H_v-formal}
    H_{\mathbf{w}}(x):=\prod_{k=1}^{\infty} \frac{\partial_yg\left((\mathbf{w}|_{k})[x],  \theta_{\mathbf{w}|_{k+1}^{\infty}}(( \mathbf{w}|_{k})[x])\right)}
    {\partial_yg\left((\mathbf{w}|_{k})[0],  \theta_{\mathbf{w}|_{k+1}^{\infty}}(( \mathbf{w}|_{k})[0])\right)}.
\end{equation}

The main purpose of this section is to prove the following proposition.

\begin{proposition}\label{pro:approximate-H-formal}
 Assume that $\mathcal{S}$ is not a single graph. Then $H_\mathbf{u}$ is real analytic on $[0,1]$ for every $\mathbf{u} \in \mathcal{A}^\infty$.
 Furthermore, there exists $r\in(0,1)$ such that the following holds.

Let $k\in\mathbb{N}$. Then there exists a constant $C_k>1$ such that, for every $n,m\in\mathbb{Z}_{>0}$, $\mathbf{u}\in\mathcal{A}^{\infty}$, $\mathbf{a},\mathbf{b}\in\mathcal{A}^*$ and $x\in[0,1]$, setting $\mathbf{u}':=\mathbf{u}|_{n+1}^{n+m}$ we have
\begin{equation}\label{eq:H-approximation}
\left|\left(H_{\mathbf{u}}\right)^{(k)}(x)-\left(\frac{\Lambda_{\mathbf{u}|_n;\mathbf{u}'\mathbf{a},\mathbf{u}'\mathbf{b}}}{\Lambda_{\mathbf{u}|_n;\mathbf{u}'\mathbf{a},\mathbf{u}'\mathbf{b}}(0)} \right)^{(k)}(x)  \right|\le C_k(r^n+r^m).
\end{equation}
\end{proposition}

\subsection{Proof of Proposition \ref{pro:approximate-H-formal}}
We first introduce the following auxiliary lemma.

\begin{lemma}\label{lem:complex-prod-estimate}
 Let $L>1$ and $\rho\in(0,1)$. Then there exists  $M>1$ so
that the following holds.
Let $m,N\in\mathbb{Z}_{>0}$ and $z_1,\ldots,z_N\in\mathbb{C}$ satisfy $|z_j|\le L\rho^{m+j}$ for each $j\in\{1,\cdots,N\}$. Then
$$
     \left| \prod_{j=1}^{N}(1+z_j)-1 \right|\le M \rho^m.
$$
\end{lemma}

\begin{proof}
Set $M:= \frac{L e^{L/{(1-\rho)}}}{1-\rho}$. Therefore we have
\begin{equation}\nonumber
\begin{split}
   \left| \prod_{j=1}^{N}(1+z_j)-1 \right|&\le    \prod_{j=1}^{N}(1+|z_j|)-1 \\&\le \prod_{j=1}^{N}(1+L\rho^{m+j})-1 \\&\le
   e^{\sum_{j=1}^{N}L\rho^{m+j}}-1\\&\le
   e^{\frac{L}{1-\rho}\rho^m}-1\\&\le
   M \rho^m,
\end{split}
\end{equation}
where the last inequality is from the fact that $e^t-1\le e^tt$ for each $t\in \R_+$. 
\end{proof}

Given $\mathbf{w},\mathbf{a},\mathbf{b}\in\mathcal{A}^*$, define the function on $\R$ by
$$\lambda_{\mathbf{w};\mathbf{a},\mathbf{b}}(x):=  \int_0^1
\partial_yg\left(\mathbf{w}[x],
s\cdot\theta_{\mathbf{a}}(\mathbf{w}[x])+(1-s)\cdot \theta_{\mathbf{b}}(\mathbf{w}[x])
\right)\,ds.$$
For $\mathbf{v}\in\mathcal{A}^{\infty}$ and $m\in\mathbb{Z}_{>0}$, define the function on $\R$ by
$$ h_{\mathbf{v},m}(x):= \partial_yg\left((\mathbf{v}|_{m})[x],  \theta_{\mathbf{v}|_{m+1}^{\infty}}(( \mathbf{v}|_{m})[x])\right).$$

\begin{lemma}\label{lem:approximation-each-term}
 There exist constants $r,\delta\in(0,1)$ and $C>1$ such that the following holds.

Let $\mathbf{w},\mathbf{a},\mathbf{b}\in\mathcal{A}^*$, $\mathbf{v}\in\mathcal{A}^{\infty}$, and $m,n\in\mathbb{Z}_{>0}$.  Set $\mathbf{v}':=\mathbf{v}|_{n+1}^{n+m}$. Then the functions
$\lambda_{\mathbf{w};\mathbf{a},\mathbf{b}},h_{\mathbf{v},m}, \lambda_{\mathbf{v}|_n;\mathbf{v}'\mathbf{a}, \mathbf{v}'\mathbf{b}}$ admit holomorphic extensions to $\mathcal{D}_{\delta}$ such that the following holds.

Then, for all $z_1,z_2\in\mathcal{D}_{\delta}$, we have
$$ \left|\frac{\lambda_{\mathbf{w};\mathbf{a},\mathbf{b}}(z_1)} {\lambda_{\mathbf{w};\mathbf{a},\mathbf{b}}(z_2)}  -1\right| \le Cr^{|\mathbf{w}|}\qquad\text{and}\qquad \left| \frac{h_{\mathbf{v},m}(z_1)}{h_{\mathbf{v},m}(z_2)} -1\right|\le Cr^m.$$
    Furthermore, we obtain:
$$ \left|\frac{\lambda_{\mathbf{v}|_n;\mathbf{v}'\mathbf{a}, \mathbf{v}'\mathbf{b}}(z_1)} {\lambda_{\mathbf{v}|_n;\mathbf{v}'\mathbf{a},\mathbf{v}'\mathbf{b}}(z_2)}\cdot  \frac{h_{\mathbf{v},m}(z_2)}{h_{\mathbf{v},m}(z_1)}-1   \right| \le C r^m.   $$ 
\end{lemma}

\begin{proof}
Since $ r_{\min}\le |\partial_yg(x,y)|\le r_{\max}$ for all $(x,y)\in [0,1]\times U$, thus there exists $\epsilon>0$ such that
$\partial_yg(x,y)$ admit holomorphic extensions to $\mathcal{U}_{\epsilon}$ satisfying $$ \frac{r_{\min}}2\le |\partial_{z_2}g(z_1,z_2)|\le \frac{1+r_{\max}}2\qquad \text{for all}(z_1,z_2)\in \mathcal{U}_{\epsilon}.$$
{
Here  $\mathcal{D}_{\delta}$ and $\mathcal{U}_{\epsilon}$ are defined in
\eqref{eq:def-D-delta} and \eqref{eq:def-U-delta}, respectively.
}
Let $\delta\in(0,1)$ be given by Lemma \ref{lem:uniform-holomorphic-extension-theta-u} for this choice of $\epsilon/2$.
Therefore the functions
$\lambda_{\mathbf{w};\mathbf{a},\mathbf{b}},h_{\mathbf{v},m}, \lambda_{\mathbf{v}|_n;\mathbf{v}'\mathbf{a}, \mathbf{v}'\mathbf{b}}$ admit holomorphic extensions to $\mathcal{D}_{\delta}$.

By decreasing $\delta$ if necessary, there are $r\in(0,1)$ and $C_1>1$ such that the conclusion of Lemma \ref{lem:uniform-bound-k-derivative} holds for $k=1$ with  $\delta$, $r$, and $C_1$. Define $$
M:=\sup_{(z_1,z_2)\in\mathcal{U}_{\epsilon/2}}
\left|\partial^2_{z_2} g(z_1,z_2)\right|+\sup_{(z_1,z_2)\in\mathcal{U}_{\epsilon/2}}
\left|\partial_{z_1}\partial_{z_2} g(z_1,z_2)\right|.
$$

Let $z_1,z_2\in\mathcal{D}_{\delta}$. By Leibniz's rule and Lemma \ref{lem:uniform-bound-k-derivative}, we have
\begin{equation}\nonumber
    \begin{split}
    |h_{\mathbf{v},m}(z_1)-h_{\mathbf{v},m}(z_2)|&\le M\left( \left| (\mathbf{v}|_{m})[z_1]- (\mathbf{v}|_{m})[z_2] \right|+\left|\theta_{\mathbf{v}|_{m+1}^{\infty}}(( \mathbf{v}|_{m})[z_2])-\theta_{\mathbf{v}|_{m+1}^{\infty}}(( \mathbf{v}|_{m})[z_1])  \right| \right)  \\&\le     M(1+C_1)|( \mathbf{v}|_{m})[z_2]-( \mathbf{v}|_{m})[z_1]|\\&\le
    2MC_1^2 r^m |z_2-z_1|\\&\le
    8MC_1^2 r^m.
    \end{split}
\end{equation}
Since $|h_{\mathbf{v},m}(z_2)|\ge\frac{r_{\min}}2$, setting $C:=\frac{16MC_1^2}{r_{\min}}$, we have
$$ \left| \frac{h_{\mathbf{v},m}(z_1)}{h_{\mathbf{v},m}(z_2)}-1  \right| \le Cr^m. $$
By enlarging $C$ if necessary, the remaining two inequalities follow from the similar argument. This completes the proof of the lemma.
\end{proof}

\begin{proof}[Proof of Proposition \ref{pro:approximate-H-formal}]
Let $r,\delta\in(0,1)$ and $C>1$ be the constants provided by Lemma \ref{lem:approximation-each-term}.
By applying Cauchy's integral formula as in the proof of Lemma \ref{lem:uniform-bound-k-derivative}, it suffices to establish the following claim.

\begin{claim}
      There exists  $C'_1>1$ such that, given $n,m\in\mathbb{Z}_{>0}$, $\mathbf{u}\in\mathcal{A}^{\infty}$, $\mathbf{a},\mathbf{b}\in\mathcal{A}^*$ and $z\in\mathcal{D}_{\delta/2}$, setting $\mathbf{u}':=\mathbf{u}|_{n+1}^{n+m}$, we get
    $$ \left|H_{\mathbf{u}}(z)-\frac{\Lambda_{\mathbf{u}|_n;\mathbf{u}'\mathbf{a},\mathbf{u}'\mathbf{b}}(z)}{\Lambda_{\mathbf{u}|_n;\mathbf{u}'\mathbf{a},\mathbf{u}'\mathbf{b}}(0)}   \right|\le C'_1(r^n+r^m).$$
\end{claim}

We now prove the claim. Let $M>1$ be the constant given by Lemma \ref{lem:complex-prod-estimate} with $L=C$ and $\rho=r$. Set
$$ H(z):=\prod_{j=1}^n\frac{h_{\mathbf{u},j}(z)}{h_{\mathbf{u},j}(0)}\qquad\text{and}\qquad \Lambda(z):=\prod_{j=1}^n\frac{\lambda_{\mathbf{u}|_n;\mathbf{u}'\mathbf{a}, \mathbf{u}'\mathbf{b}}(z)} {\lambda_{\mathbf{u}|_n;\mathbf{u}'\mathbf{a},\mathbf{u}'\mathbf{b}}(0)}.$$
Thus we have
\begin{equation}\label{eq:triangle-inequality-2}
  \left|H_{\mathbf{u}}(z)-\frac{\Lambda_{\mathbf{u}|_n;\mathbf{u}'\mathbf{a},\mathbf{u}'\mathbf{b}}(z)}{\Lambda_{\mathbf{u}|_n;\mathbf{u}'\mathbf{a},\mathbf{u}'\mathbf{b}}(0)}   \right|\le |H_{\mathbf{u}}(z)-H(z)|+|H(z)-\Lambda(z)|.
\end{equation}

By the definitions of $H_{\mathbf{u}}(z),H(z)$, we have
\begin{equation}\label{eq:H-minus}
    \begin{split}
      |H_{\mathbf{u}}(z)-H(z)|&\le |H(z)|\cdot \left| \frac{H_{\mathbf{u}}(z)}{H(z)} -1 \right|\\&\le
      \left(\left| \prod_{j=1}^n\frac{h_{\mathbf{u},j}(z)}{h_{\mathbf{u},j}(0)}-1\right|+1\right) \cdot \left| \prod_{j=n+1}^{\infty}\frac{h_{\mathbf{u},j}(z)}{h_{\mathbf{u},j}(0)} -1 \right|.
    \end{split}
\end{equation}
By Lemma \ref{lem:approximation-each-term}, we have
$$  \left| \frac{h_{\mathbf{u},j}(z)}{h_{\mathbf{u},j}(0)} -1 \right|\le C r^j\qquad\text{for each }j\in\Z_{>0}.   $$
Combining this estimate and \eqref{eq:H-minus} with Lemma \ref{lem:complex-prod-estimate}, we obtain
\begin{equation}\label{eq:H-minus-upper-bound}
     |H_{\mathbf{u}}(z)-H(z)|\le (M+1)M r^n.
\end{equation}

By the similar argument, applying Lemmas \ref{lem:approximation-each-term} and \ref{lem:complex-prod-estimate}  again, we obtain
$$ |H(z)-\Lambda(z)|\le (M+1)^2 M r^m. $$
Combining this inequality with \eqref{eq:H-minus-upper-bound} and \eqref{eq:H-minus}, we obtain
$$|H_{\mathbf{u}}(z)-\Lambda(z)|\le (M+1)^2M (r^n+r^m).$$
Thus,  \eqref{eq:H-approximation} holds with $C_1:=(M+1)^2M$.

Finally, applying \eqref{eq:H-minus-upper-bound} and the definition of $H$, and since $n>0$ is arbitrary, the Weierstrass convergence theorem implies that $H_{\mathbf{u}}$ is real analytic on $[0,1]$. Thus our claim holds.
\end{proof}

\section{Transversality I: The case $\partial_y^2 g\not\equiv 0$}
\label{sec:transversality-I}

In this section, we establish a multiscale higher-order transversality bound, under the condition that for all $\mathbf{u}$, $H_\mathbf{u}' \neq 0$. We will see that, since $\mathcal{S}$ is not a single graph, this condition is satisfied as soon as $\partial_y^2 g(x,y) \neq 0$. The case $\partial_y^2 g(x,y)  = 0$ will be dealt with in the next section.

\subsection{Preliminary remarks}

The next lemma explains why the $k$-th derivative of $\Lambda$ (or $H$) drives the separation of derivatives that we search for. In the next lemma, think of $\tau$ as being small.

\begin{lemma}\label{lem:higher-order-product-composition-formal}
    Let an integer $k>0$. Then, for every $1\leq j\leq k$ and  $1\leq i\leq k+1-j$, 
    {
    there exists a 
    homogeneous
    polynomial $P_{i,j,k}\in\mathbb{R}[x_1,\ldots,x_{k+2-j-i}] $ of  degree  $j$  and the following  holds.
    }
    Let $\Lambda,\Delta,\tau$ be  $C^k$ function on $\R$, then 
    \begin{equation}\nonumber
    \begin{split}
        \left(\Lambda \cdot \Delta\circ\tau \right)^{(k)}(x)&=
        \Lambda^{(k)}(x)\cdot \Delta(\tau(x)) \\&+\sum_{j=1}^{k}\Delta^{(j)}(\tau(x))\sum_{i=1}^{k+1-j}\Lambda^{(i-1)}(x)\cdot P_{i,j,k}\left(\tau^{(1)}(x), \ldots,  \tau^{(k+2-j-i)}(x)\right).
    \end{split}
    \end{equation}
\end{lemma}

\begin{proof}
  The lemma follows by induction on \(k\), using the Leibniz rule and the chain rule. We omit the details.
\end{proof}

A natural approach is then to first establish transversality when $\Lambda^{(k)} \neq 0$. More formally, the relevant setting here is $H_\mathbf{u}' \neq 0$ for all $\mathbf{u}$.

\begin{lemma}\label{lem:H_u-high-order-derivative-formal}
    Assume that $ H_{\mathbf u}'(x)\not\equiv 0\text{ for each }\mathbf u\in\mathcal{A}^{\infty}. $  Then there exists
     $\delta>0$ and $K\in\Z_{>0}$ such that the following holds.
     For all $\mathbf{u}\in\mathcal{A}^{\infty}$ and $x\in [0,1]$, there exists $k\in\{1,2,\ldots,K\}$ satisfying
     $ | (H_{\mathbf u})^{(k)}(x) |\ge\delta. $
\end{lemma}
    
\begin{proof}
    Assume that our claim fails. Then for each $n\in\Z_{>0}$, there exists
    $\mathbf u_n\in\mathcal{A}^{\infty}$
    and $x_n\in [0,1]$ such that 
    $$| (H_{\mathbf u_n})^{(k)}(x_n)|<\frac1n\qquad\text{for all }k\in\{1,2,\ldots,n\}.$$
    By passing to a subsequence, we  assume that
    $\mathbf u_n\to \mathbf u_0\text{ and } x_n\to x_0. $
    Thus for every $k\in\Z_{>0}$, 
    $$(H_{\mathbf u_0})^{(k)}(x_0)=\lim_{n\to\infty}(H_{\mathbf u_n})^{(k)}(x_n)=0.$$
    Since $H_{\mathbf u_0}(x)$ is a real analytic function satisfying
    $$  (H_{\mathbf u_0})^{(k)}(x_0)=0\qquad\forall\,k\in\Z_{>0},$$
    we have $H'_{\mathbf u_0}(x)\equiv 0$. This contradicts with our assumption. Thus the lemma holds.
\end{proof}

\subsection{Transversality when $H_\mathbf{u}' \not \equiv 0$}
The following is the main result of this subsection.

\begin{proposition}\label{pro:Separation-nozero-case}
    Assume that for all $\mathbf{u} \in \mathcal{A}^\infty$, $H_\mathbf{u}'\not\equiv 0$. Then there exists $K,N\in\Z_{>0}$ and $c\in(0,1)$ such that the following holds. For any $x\in [0,1]$ and $\mathbf{u}\in\mathcal{A}^*$ satisfying $|\mathbf{u}|\ge N$, there exists 
    $k\in\{1,\dots,K   \}$ and
    $\mathbf{a},\mathbf{b}\in\mathcal{A}^N$
    such that
    $$ |(\Delta_{\mathbf{u}\mathbf{a}\mathbf{c} ,\mathbf{u}\mathbf{b}\mathbf{d}})^{(k)}(x)|\ge c (r_{\min})^{|\mathbf{u}|}\qquad\text{for all }\mathbf{c},\mathbf{d}\in\mathcal{A}^* .$$   
\end{proposition}

{
\noindent\textbf{Idea of the proof of Proposition \ref{pro:Separation-nozero-case}.}
We split the common prefix as $\mathbf{u}=\mathbf{u}_1\mathbf{u}_2$ and use
\begin{equation}\nonumber
\Delta_{\mathbf{u}\mathbf{a}\mathbf{c},
\mathbf{u}\mathbf{b}\mathbf{d}}(x)
=
\Lambda_{\mathbf{u}_1;\mathbf{u}_2\mathbf{a}\mathbf{c},
\mathbf{u}_2\mathbf{b}\mathbf{d}}(x)
\cdot
\Delta_{\mathbf{u}_2\mathbf{a}\mathbf{c},
\mathbf{u}_2\mathbf{b}\mathbf{d}}(\mathbf{u}_1[x]).
\end{equation}
Lemma \ref{lem:higher-order-product-composition-formal} is then applied to the higher-order derivatives of the right-hand side.
When $|\mathbf{u}_1|$ is large, the derivatives of $\mathbf{u}_1[x]$ are small, so the main term is
\begin{equation}\nonumber
\Lambda_{\mathbf{u}_1;\mathbf{u}_2\mathbf{a}\mathbf{c},
\mathbf{u}_2\mathbf{b}\mathbf{d}}^{(k)}(x)
\cdot
\Delta_{\mathbf{u}_2\mathbf{a}\mathbf{c},
\mathbf{u}_2\mathbf{b}\mathbf{d}}(\mathbf{u}_1[x]).
\end{equation}
The second part $\mathbf{u}_2$ of the prefix is also needed to apply Proposition
\ref{pro:approximate-H-formal}. After normalizing by the value at $0$, that proposition gives, when $|\mathbf{u}_2|$ is large, an approximation of
\begin{equation}\nonumber
\left(
\frac{
\Lambda_{\mathbf{u}_1;\mathbf{u}_2\mathbf{a}\mathbf{c},
\mathbf{u}_2\mathbf{b}\mathbf{d}}(x)}
{
\Lambda_{\mathbf{u}_1;\mathbf{u}_2\mathbf{a}\mathbf{c},
\mathbf{u}_2\mathbf{b}\mathbf{d}}(0)}
\right)^{(k)}
\end{equation}
by $H_{\mathbf{u}_1\mathbf{u}_2 00\cdots}^{(k)}(x)$.
Lemma \ref{lem:H_u-high-order-derivative-formal} gives a nontrivial lower bound for
$|H_{\mathbf{u}_1\mathbf{u}_2 00\cdots}^{(k)}(x)|$
for some $k$, while Lemma \ref{lem:lower-bound-Delta-formal} gives the required lower bound for
$\Delta_{\mathbf{u}_2\mathbf{a}\mathbf{c},
\mathbf{u}_2\mathbf{b}\mathbf{d}}(\mathbf{u}_1[x])$.
These estimates give the desired derivative lower bound.
}

\begin{proof}
\noindent\textbf{Step 1: Choice of the parameters $K$, $N$, and $c$.}
Since $\mathcal{S}$ does not consist of a single graph, by Proposition \ref{pro:non-degenecay-H_u-formal} and Lemma \ref{lem:H_u-high-order-derivative-formal}, there exists $\delta_1\in(0,1)$ and $K\in\Z_{>0}$ such that the conclusion of Lemma \ref{lem:H_u-high-order-derivative-formal} holds with $\delta=\delta_1$ and  $K$.  

{
Given $k\in\{1,\ldots,K\}$, we apply Lemma \ref{lem:higher-order-product-composition-formal} to obtain
 homogeneous polynomials
$$
P_{i,j,k}\in\mathbb{R}[x_1,\ldots,x_{k+2-j-i}],
\qquad
1\leq j\leq k,
\qquad
1\leq i\leq k+1-j,
$$
of  degree $j$, such that the conclusion of Lemma \ref{lem:higher-order-product-composition-formal} holds for these polynomials. In particular, $P_{i,j,k}(0,\cdots,0)=0$.
}
Fix  $M_1>1$  such that $M_1$ is greater than both the absolute value of every coefficient and the number of monomials of each of these polynomials.

Choose $r,\delta_2\in(0,1)$ and  $C_k>1$, for $k\in\{1,\ldots,K\}$, such that the conclusion of Lemma \ref{lem:uniform-bound-k-derivative} holds with these choices of $r$, $\delta=\delta_2$, and $C_k$ for every $k\in\{1,\ldots,K\}$.

Let $r'\in(0,1)$ and $C_k'>1$, for $k\in\{1,\ldots,K\}$, be given by Proposition \ref{pro:approximate-H-formal}  so that its conclusion holds with $r=r'$ and $C_k=C_k'$ for any $k\in\{1,\ldots,K\}$.

Since $\mathcal{A}^{\infty}$ is a compact set, and the map $\mathbf{u}\to (H_{\mathbf{u}})^{(k)}$ is continuous for each $k\in\N$, thus
$$ M_2:=\sup_{ \substack{1\le k \le K\\ \mathbf{u}\in\mathcal{A}^{\infty},x\in[0,1] }  } \left|  (H_{\mathbf{u}})^{(k)}(x)  \right|<\infty. $$

Define
$C':=\max_{1\leq k\leq K} C_k'.$
Fix $N_1\in\Z_{>0}$ sufficiently large so that
$4C'(r')^{N_1}\leq \delta_1.$
Let $\delta_{N_1}\in (0,1)$ and $M_{N_1}\in\Z_{>0}$ such that the conclusion of Lemma \ref{lem:lower-bound-Delta-formal} holds for $N=N_1$.

Set
$C:=\max_{1\leq k\leq K} C_k.$
Choose $N_2\in\Z_{>0}$ sufficiently large that
\begin{equation}\label{eq:choose-N_2}
    K^2M_1^2(M_2+2C')C^2r^{N_2}\le \frac{\delta_1\delta_{N_1}}4.
\end{equation}
Set $N= 2N_1+N_2+M_{N_1}$ and $c:=\frac{\delta_1\delta_{N_1}}4$. Write
$\mathbf{0}:=00\cdots\mathcal{A}^{\infty}$.

\noindent\textbf{Step 2: Proof of Proposition \ref{pro:Separation-nozero-case}.}
Let $x\in [0,1]$ and $\mathbf{u}\in\mathcal{A}^*$ satisfying $|\mathbf{u}|\ge N$. Let $\mathbf{u}_1,\mathbf{u}_2\in\mathcal{A}^*$
be such that $$\mathbf{u}=\mathbf{u}_1\mathbf{u}_2\qquad\text{and}\qquad  |\mathbf{u}_2|=N_1.$$
{
Applying Lemma \ref{lem:lower-bound-Delta-formal} with 
$\mathbf{w}=\mathbf{u}_2$, $n=N$ and $\mathbf{u}=\mathbf{u}_1$, we obtain 
$\mathbf{a},\mathbf{b}\in\mathcal{A}^N$
such that the conclusion of  Lemma \ref{lem:lower-bound-Delta-formal} holds. 
}
That is, for all $\mathbf{c},\mathbf{d}\in\mathcal{A}^*$, setting
$$\Delta(x):=\Delta_{\mathbf{u}_2\mathbf{a}\mathbf{c}, \mathbf{u}_2\mathbf{b}\mathbf{d}}(x),$$
we obtain
\begin{equation}\label{eq:lower-bound-Delta}
|\Delta(\mathbf{u}_1[x])|\ge \delta_{N_1}.
\end{equation}

Applying Lemma \ref{lem:H_u-high-order-derivative-formal} with $\mathbf{w}=\mathbf{u}_1\mathbf{u}_2\mathbf{0}$ and $x$, there exists $k=k(\mathbf{w},x)\in\{1,\cdots,K\}$ satisfying
\begin{equation}\label{eq:lower-bound-k-order-H}
|H^{(k)}_{\mathbf{u}_1\mathbf{u}_2\mathbf{0}}(x)|\ge\delta_1.
\end{equation}
Write
$$  \Lambda(x):= \frac{\Lambda_{\mathbf{u}_1;\mathbf{u}_2\mathbf{a}\mathbf{c},\mathbf{u}_2\mathbf{b}\mathbf{d}}(x)}{\Lambda_{\mathbf{u}_1;\mathbf{u}_2\mathbf{a}\mathbf{c},\mathbf{u}_2\mathbf{b}\mathbf{d}}(0)}.$$
Applying Proposition \ref{pro:approximate-H-formal} with $n=|\mathbf{u}_1|$ and $m=|\mathbf{u}_2|$, we have
$$
\left|\left(H_{\mathbf{u}_1\mathbf{u}_2\mathbf{0}}\right)^{(k)}(x)- \Lambda^{(k)}(x)  \right|\le C'_k((r')^{|\mathbf{u}_1|}+(r')^{N_1}).
$$
Thus by $|\mathbf{u}_1|\ge N_1$ and $2C_k'(r')^{N_1}\le\frac{\delta_1}2$, we have
$$  \left|\left(H_{\mathbf{u}_1\mathbf{u}_2\mathbf{0}}\right)^{(k)}(x)- \Lambda^{(k)}(x)  \right|\le  \frac{\delta_1}2.$$
Combining this with \eqref{eq:lower-bound-k-order-H}, we obtain
\begin{equation}\label{eq:lower-bound-k-order-Lambda}
  |\Lambda^{(k)}(x)|\ge  \left|\left(H_{\mathbf{u}_1\mathbf{u}_2\mathbf{0}}\right)^{(k)}(x)  \right| - \left|\left(H_{\mathbf{u}_1\mathbf{u}_2\mathbf{0}}\right)^{(k)}(x)- \Lambda^{(k)}(x)  \right|\ge \frac{\delta_1}2.
\end{equation}

On the other hand, Lemma \ref{lem:higher-order-product-composition-formal} implies that
 \begin{equation}\label{eq:long-k-order-derivative}
    \begin{split}
        \left(\Lambda(x) \cdot \Delta (\mathbf{u}_1[x]) \right)^{(k)}&=
        \Lambda^{(k)}(x)\cdot \Delta(\mathbf{u}_1[x]) \\&+\sum_{j=1}^{k}\Delta^{(j)}(\mathbf{u}_1[x])\sum_{i=1}^{k+1-j}\Lambda^{(i-1)}(x)\cdot P_{i,j,k}\left(\mathbf{u}_1[x]^{(1)}, \ldots,  \mathbf{u}_1[x]^{(k+2-j-i)}\right).
    \end{split}
    \end{equation}

We shall use the following inequality to finish the proof.
\begin{equation}\label{eq:error-upper-bound}
    \left| \sum_{j=1}^{k}\Delta^{(j)}(\mathbf{u}_1[x])\sum_{i=1}^{k+1-j}\Lambda^{(i-1)}(x)\cdot P_{i,j,k}\left(\mathbf{u}_1[x]^{(1)}, \ldots,  \mathbf{u}_1[x]^{(k+2-j-i)}\right) \right|\le\frac{\delta_1\delta_{N_1}}4.
\end{equation}
Assume this inequality holds. Combining  this with \eqref{eq:lower-bound-Delta}, \eqref{eq:lower-bound-k-order-Lambda} and \eqref{eq:long-k-order-derivative}, we obtain
\begin{equation}\label{eq:key-lower-bound}
    \begin{split}
        \left| \left(\Lambda(x) \cdot \Delta (\mathbf{u}_1[x]) \right)^{(k)}  \right|&\ge\left|\Lambda^{(k)}(x)\cdot \Delta(\mathbf{u}_1[x]) \right|-\frac{\delta_1\delta_{N_1}}4\\&\ge\frac{\delta_1}{2}\cdot\delta_{N_1}-\frac{\delta_1\delta_{N_1}}4\\&\ge\frac{\delta_1\delta_{N_1}}4.
    \end{split}
\end{equation}

By the definitions of $\Lambda,\Delta$ and \eqref{def:Delta-u-a,b-formal} with $\mathbf{u}=\mathbf{u}_1\mathbf{u}_2$, we have
\begin{equation}\nonumber
\begin{split}
    \Lambda(x) \cdot \Delta (\mathbf{u}_1[x])&=\frac{\Lambda_{\mathbf{u}_1;\mathbf{u}_2\mathbf{a}\mathbf{c},\mathbf{u}_2\mathbf{b}\mathbf{d}}(x)\cdot\Delta_{\mathbf{u}_2\mathbf{a}\mathbf{c}, \mathbf{u}_2\mathbf{b}\mathbf{d}}(\mathbf{u}_1[x])}{\Lambda_{\mathbf{u}_1;\mathbf{u}_2\mathbf{a}\mathbf{c},\mathbf{u}_2\mathbf{b}\mathbf{d}}(0)}\\&=  \frac{\Delta_{\mathbf{u}\mathbf{a}\mathbf{c}, \mathbf{u}\mathbf{b}\mathbf{d}}(x)}{\Lambda_{\mathbf{u}_1;\mathbf{u}_2\mathbf{a}\mathbf{c},\mathbf{u}_2\mathbf{b}\mathbf{d}}(0)}.
\end{split}
\end{equation}
Thus \eqref{eq:key-lower-bound} implies that
\begin{equation}\label{eq:Delta-lover-bound-2}
\left|(\Delta_{\mathbf{u}\mathbf{a}\mathbf{c}, \mathbf{u}\mathbf{b}\mathbf{d}})^{(k)}(x)\right|\ge \frac{\delta_1\delta_{N_1}}4 \cdot \left|  \Lambda_{\mathbf{u}_1;\mathbf{u}_2\mathbf{a}\mathbf{c},\mathbf{u}_2\mathbf{b}\mathbf{d}}(0)  \right|.
\end{equation}
Combining \eqref{def:Lambda-u-a,b-formal} with \eqref{eq:partial-G-formula-Formal} and  $|\partial_yg(x,y)|\ge r_{\min}$ for all $(x,y)\in [0,1]\times U$, we have
$$ \left|  \Lambda_{\mathbf{u}_1;\mathbf{u}_2\mathbf{a}\mathbf{c},\mathbf{u}_2\mathbf{b}\mathbf{d}}(0)  \right|\ge (r_{\min})^{|\mathbf{u}_1|}.$$
Thus by \eqref{eq:Delta-lover-bound-2} and $c=\frac{\delta_1\delta_{N_1}}4 $, we have
$$ \left|(\Delta_{\mathbf{u}\mathbf{a}\mathbf{c}, \mathbf{u}\mathbf{b}\mathbf{d}})^{(k)}(x)\right|\ge c  (r_{\min})^{|\mathbf{u}|}.$$
Therefore the proposition holds.

\noindent\textbf{Step 3: Proof of \eqref{eq:error-upper-bound}.}
By Lemma \ref{lem:uniform-bound-k-derivative} and our choices of the parameters $r$, $C_k$, and $C$, for every $j\in\{1,\ldots,K\}$, we have
\begin{equation}\label{eq:Delta-upper-bound}
    \left| \Delta^{(j)}(s) \right|\le C_j r^{|\mathbf{u}_2|}\le C\qquad\text{for each }s\in[0,1].
\end{equation}

Applying Proposition \ref{pro:approximate-H-formal} again, for every $i\in\{0,1,\cdots, K \}$, we obtain
$$
\left|\left(H_{\mathbf{u}_1\mathbf{u}_2\mathbf{0}}\right)^{(i)}(x)- \Lambda^{(i)}(x)  \right|\le C'_i((r')^{|\mathbf{u}_1|}+(r')^{|\mathbf{u}_2|})\le 2C'.
$$
Thus by the definition of $M_2$, we obtain
\begin{equation}\label{eq:Lambda-upper-bound}
   \left|\Lambda^{(i)}(x)  \right|\le  \left|\left(H_{\mathbf{u}_1\mathbf{u}_2\mathbf{0}}\right)^{(i)}(x) \right|+\left|\left(H_{\mathbf{u}_1\mathbf{u}_2\mathbf{0}}\right)^{(i)}(x)- \Lambda^{(i)}(x)  \right|\le M_2+2C'.
\end{equation}

Given  $k\in\{1,\ldots,K\}$ and $1\leq j\leq k,
1\leq i\leq k+1-j$, 
By
$P_{i,j,k}(0,\cdots,0)=0$ and the choice of $M_1$,
we obtain
$$\left| P_{i,j,k}\left(\mathbf{u}_1[x]^{(1)}, \ldots,  \mathbf{u}_1[x]^{(k+2-j-i)}\right)    \right| \le M_1^2\cdot\max_{1\le t \le K}\left|\mathbf{u}_1[x]^{(t)}  \right|.$$ 
Combining Lemma \ref{lem:uniform-bound-k-derivative} with $|\mathbf{u}_1|\ge N_2$ and the definition of $C$, we have
$$ \left|\mathbf{u}_1[x]^{(t)}  \right|\le C_t r^{|\mathbf{u}_1|}\le C r^{N_2}.$$
Then we get
$$ \left| P_{i,j,k}\left(\mathbf{u}_1[x]^{(1)}, \ldots,  \mathbf{u}_1[x]^{(k+2-j-i)}\right)    \right| \le M_1^2  C r^{N_2}.$$
Combining this and \eqref{eq:Delta-upper-bound}, \eqref{eq:Lambda-upper-bound} and \eqref{eq:choose-N_2}, we obtain
\begin{equation}\nonumber
\begin{split}
    &\left| \sum_{j=1}^{k}\Delta^{(j)}(\mathbf{u}_1[x])\sum_{i=1}^{k+1-j}\Lambda^{(i-1)}(x)\cdot P_{i,j,k}\left(\mathbf{u}_1[x]^{(1)}, \ldots,  \mathbf{u}_1[x]^{(k+2-j-i)}\right)  \right|\\&\le
    \sum_{j=1}^{k}C\sum_{i=1}^{k+1-j}(M_2+2C') M_1^2Cr^{N_2}\\&\le
    KCK(M_2+2C') M_1^2Cr^{N_2}\\&\le \frac{\delta_1\delta_{N_1}}4.
\end{split}
\end{equation}
   Thus  \eqref{eq:error-upper-bound} holds.
\end{proof}

\subsection{$H_{\mathbf{u}}' \not \equiv 0 $ holds generically.}\label{subsec:H}

The goal of this subsection is to show that our hypothesis on $H_\mathbf{u}$ holds as soon as $\partial_y^2 g(x,y) \not \equiv 0$.

\begin{proposition}\label{pro:non-degenecay-H_u-formal}
     Suppose that $\partial_{y}^2g(x,y)\not\equiv 0$ and recall that $\mathcal{S}$ is not a single graph. Then  $H_{\mathbf{u}}'(x)\not\equiv 0$ for each $\mathbf{u}\in\mathcal{A}^{\infty}$.
\end{proposition}

The remainder of this section is devoted to proving this proposition. Accordingly, throughout this section, we assume that $\partial_{y}^2g(x,y)\not\equiv 0$. The strategy is the following: we first show that if
$H_{\mathbf{u}}'(x)\equiv 0$
for some $\mathbf{u}\in\mathcal{A}^{\infty}$, then
\begin{equation}\label{eq:partial-g}
\partial_y g\bigl(x,\theta_{\mathbf{v}}(x)\bigr)
\equiv c
\qquad
\text{for all }
\mathbf{v}\in\mathcal{A}^{\infty};
\end{equation}
see Lemma \ref{lem:partial_yg-constant-formal}. The hypothesis 
 $\partial_y^2 g(x,y)\not\equiv 0$ implies that $\mathcal{S}$ is composed of only finitely many unstable curves, see Corollary \ref{cor:finite-curves-formal}.
But this is impossible by Lemma \ref{lem:curve-splitting-formal}. 

\subsubsection{An equation for unstable curves}

We now establish relation \eqref{eq:partial-g} when $H_\mathbf{u}' \equiv 0$. We start by some remarks on the additive structure of $H_\mathbf{u}$ and some consequences.

\begin{lemma}\label{lem:addtion-H_u-formal}
    For $\mathbf{u}\in\mathcal{A}^{\infty}$  and $n\in\Z$, the following holds:
    $$ H_{\mathbf{u}+n}(x)\equiv\frac{H_{\mathbf{u}}(x+n)}{H_{\mathbf{u}}(n)}. $$
\end{lemma}

\begin{proof}
{Write $\mathbf{v}:=\mathbf{u}+n$.}
 Let $k\in\Z_{>0}$ and $x\in \mathbb{R}$. By \eqref{eq:addition-f_u-formal},  we have
$$ (\mathbf{u}|_{k+j})[x+n]=(\mathbf{v}|_{k+j})[x]\bmod1\qquad\text{for all }j\in\N.$$
{
 Thus, by \eqref{def:u[x]} and the $\Z$-periodicity of $g$ in the first variable, the corresponding finite iterates in the definitions of $\theta_{\mathbf{u}|_{k+1}^{\infty}}$ and $\theta_{\mathbf{v}|_{k+1}^{\infty}}$ agree at every level. We have
 } 
$$ \theta_{\mathbf{u}|_{k+1}^{\infty}}(( \mathbf{u}|_{k})[x+n])=\theta_{\mathbf{v}|_{k+1}^{\infty}}(( \mathbf{v}|_{k})[x]).$$
It follows that
$$ \partial_yg\left((\mathbf{u}|_{k})[x+n],  \theta_{\mathbf{u}|_{k+1}^{\infty}}(( \mathbf{u}|_{k})[x+n])\right)=\partial_yg\left((\mathbf{v}|_{k})[x],  \theta_{\mathbf{v}|_{k+1}^{\infty}}(( \mathbf{v}|_{k})[x])\right).   $$
{
Combining this with the definitions of $H_{\mathbf{u}}$ and
$H_{\mathbf{v}}$, we obtain
\begin{equation}\label{eq:H-addition-relation}
\begin{split}
    H_{\mathbf{v}}(x)
    &=
    \prod_{k=1}^{\infty}
    \frac{
    \partial_y g\left(
        (\mathbf{u}|_k)[x+n],
        \theta_{\mathbf{u}|_{k+1}^{\infty}}
        \bigl((\mathbf{u}|_k)[x+n]\bigr)
    \right)}
    {
    \partial_y g\left(
        (\mathbf{u}|_k)[n],
        \theta_{\mathbf{u}|_{k+1}^{\infty}}
        \bigl((\mathbf{u}|_k)[n]\bigr)
    \right)}
    \\
    &=
    \frac{H_{\mathbf{u}}(x+n)}{H_{\mathbf{u}}(n)}.
\end{split}
\end{equation}
}
Therefore this lemma holds.
\end{proof}

\begin{corollary}
\label{cor:H'_u-uniform-degeneracy-formal}
    Suppose that there exists $\mathbf{u}\in\mathcal{A}^{\infty}$ with $H_{\mathbf{u}}'(x)\equiv 0$. Then we have 
    $$ H_{\mathbf{v}}(x)\equiv 1\qquad\text{for all }\mathbf{v}\in\mathcal{A}^{\infty}. $$
\end{corollary}

\begin{proof}
    {
    Since $H_{\mathbf{v}}(0)=1$ for each $\mathbf{v}\in\mathcal{A}^{\infty}$, it suffices to show that
    $H_{\mathbf{v}}'(x)\equiv 0$.}
    By Lemma \ref{lem:addtion-H_u-formal}, for every $n\in\Z$, we obtain
    $$ H_{\mathbf{u}+n}(x)\equiv\frac{H_{\mathbf{u}}(x+n)}{H_{\mathbf{u}}(n)}. $$
    Thus $H_{\mathbf{u}}'(x)\equiv 0$ implies that
    \begin{equation}\label{eq:H_n'-zero-orbit}
        H_{\mathbf{u}+n}'(x)\equiv 0\qquad\text{for all }n\in\Z.
    \end{equation}
    On the other hand, the following set
    $$\{\mathbf{u}+n:n\in\Z   \}$$
    is dense in the space $\mathcal{A}^{\infty}$. Combining this with \eqref{eq:H_n'-zero-orbit} and the fact that
    $$ \mathbf{u}_n\to \mathbf{u}_0\in\mathcal{A}^{\infty}\qquad \implies\quad  H'_{\mathbf{u}_n}(x)\to H'_{\mathbf{u}_0}(x),$$
    thus  $H'_{\mathbf{v}}(x)\equiv0$ for any $\mathbf{v}\in\mathcal{A}^{\infty}$. Our corollary holds.
\end{proof}

\begin{lemma}
 \label{lem:partial_yg-constant-formal}
 Suppose that there exists $\mathbf{u}\in\mathcal{A}^{\infty}$ with $H_{\mathbf{u}}'(x)\equiv 0$.
 Then there exists  $c\in\R$ satisfying
 $$ \partial_yg(x, \theta_{\mathbf{v}}(x))\equiv c\qquad\text{for all }\mathbf{v}\in\mathcal{A}^{\infty}.$$ 
\end{lemma}

\begin{proof}[Proof of Lemma \ref{lem:partial_yg-constant-formal}]
Let $\mathbf{i}\in\mathcal{A}$. Then \eqref{def:H_v-formal} implies that
$$ H_{\mathbf{i}\mathbf{u}}(x)\equiv 
\frac{\partial_yg(\mathbf{i}[x],\theta_{\mathbf{u}}(\mathbf{i}[x])) }{\partial_yg(\mathbf{i}[0],\theta_{\mathbf{u}}(\mathbf{i}[0])) }\cdot \frac{H_{\mathbf{u}}(\mathbf{i}[x])}{H_{\mathbf{u}}(\mathbf{i}[0])}. $$
By Corollary \ref{cor:H'_u-uniform-degeneracy-formal}, we have
$ H_{\mathbf{i}\mathbf{u}}(x)\equiv1 $
and $H_{\mathbf{u}}(\mathbf{i}[x])\equiv1$. Thus we have
$$\partial_yg(\mathbf{i}[x],\theta_{\mathbf{u}}(\mathbf{i}[x]))\equiv c\quad\text{ for }\quad c:= \partial_yg(\mathbf{i}[0],\theta_{\mathbf{u}}(\mathbf{i}[0])). $$
Since $\partial_yg(\mathbf{i}[x],\theta_{\mathbf{u}}(\mathbf{i}[x]))$ is a real analytic function on $\R$, we have
$ \partial_yg(x,\theta_{\mathbf{u}}(x))\equiv c$. 
{
Combining this with $\theta_{\mathbf{u}+n}(x)\equiv\theta_{\mathbf{u}}(x+n)$, we obtain}
$$ \partial_yg(x,\theta_{\mathbf{u}+n}(x))\equiv c\qquad\text{for all }n\in\Z. $$
Since the  set
    $\{\mathbf{u}+n:n\in\Z \}$
    is dense in the space $\mathcal{A}^{\infty}$, our claim holds.
\end{proof}

\subsubsection{The Number of Unstable Curves}
{
We now complete the proof of Proposition \ref{pro:non-degenecay-H_u-formal} by bounding the number of unstable curves $\{(x,\theta_{\mathbf{v}}(x))\}_{x\in[0,1]}$ for $\mathbf{v}\in\mathcal{A}^{\infty}$.
}

{
We begin with the following lemma, which gives a locally uniform upper bound on the number of zeros of a nonzero real analytic function.
}

\begin{lemma}\label{lem:zero-point-formal}
    Let $h(x,y)$ be a real analytic function on $[0,1]\times U$ with $h(x,y)\not\equiv0$. Then there exist $N>1$ and closed interval $J\subset [0,1]$  such that $J$ has nonempty interior and  the following holds.
    For any $x_0\in J$, the number of zero points of the function $h_{x_0}(y):=h(x_0,y)$ on $U$ is no more than $N$.
\end{lemma}

\begin{proof}
For $x\in [0,1]$, let $N_x\in\N\cup\{\infty\}$ be the number of zero points of the function $h_{x}(y)$ on $U$. Since $h_x$ is real analytic function on  $U$, then $N_x=\infty$ implies that $h_x(y)\equiv 0$.
Then the set
$$ \{ x\in [0,1]:N_x=\infty  \} $$
is a finite set. Otherwise for every $y\in[-1,1]$, the number of zero points
of the real analytic function $h_y(x):=h(x,y)$ on $[0,1]$ is infinite, so
$h_y(x)\equiv0$. Since \(y\) was arbitrary, then $h(x,y)\equiv0$. Contradictions!

    Fix $x_0$ with $N_{x_0}<\infty$. Since $U$ is compact, the multiplicities of the zeros of $h_{x_0}$ in $U$ are uniformly bounded. Since, otherwise $h_{x_0}$ would have a zero of infinite multiplicity and  would vanish identically, contradicting $N_{x_0}<\infty$. Thus we  use Weierstrass preparation theorem and the open covering theorem to find a interval $J$ such that $x_0\in J$ and
    $J$ satisfies the conclusion the lemma for some uniform constant $N<\infty$.
\end{proof}

The following corollary shows that there are only finitely many unstable curves in the degenerate case.

\begin{corollary}\label{cor:finite-curves-formal}
    Let $w(x,y)$ be a real analytic function on  $[0,1]\times U$ such that $\partial_yw(x,y)\not\equiv0$. 
    Assume that 
    $w(x, \theta_{\mathbf{v}}(x))\equiv c$ for some constant $c\in\R$ and all $\mathbf{v}\in\mathcal{A}^{\infty}$.

    { Then
    there exist finitely many  real analytic functions $q_1,q_2,\ldots,q_M$ on $\R$  such that the following holds:
    }
    For any $\mathbf{u}\in\mathcal{A}^{\infty}$,  there exists $t\in\{1,2,\ldots,M\}$ satisfying
    $\theta_{\mathbf{u}}(x)\equiv q_t(x)$.
\end{corollary}

\begin{proof}
By Lemma \ref{lem:zero-point-formal}, there exists a closed interval $J\subset [0,1]$ and $N\in\Z_{>0}$  such that the conclusion of Lemma \ref{lem:zero-point-formal}  holds for $h(x,y)=\partial_yw(x,y)$. 
{
If the claim is false for $M:=N+1$, then there  exist
}
$\mathbf{u}_1,\ldots,\mathbf{u}_{N+2}\in\mathcal{A}^{\infty}$ such that 
$$ \theta_{\mathbf{u}_i}(x)\not\equiv\theta_{\mathbf{u}_j}(x)\qquad\text{for any }i\neq j.  $$
{
Since the functions
$\theta_{\mathbf{u}_1},\ldots,\theta_{\mathbf{u}_{N+2}}$
are real analytic, the set of points $x\in[0,1]$ for which
$\theta_{\mathbf{u}_i}(x)=\theta_{\mathbf{u}_j}(x)$
for some $i\neq j$ is finite.
}
Then we can find a nonempty interval
$J_1\subset J$ such that
$$ \theta_{\mathbf{u}_i}(x)\neq\theta_{\mathbf{u}_j}(x)\qquad\text{for any }i\neq j\text{ and }x\in J_1.  $$
Let $x\in J_1$.
Since
 $w(x, \theta_{\mathbf{u}_j}(x))\equiv c$
 for all $j\in\{1,\cdots,N+2\}$, applying the mean value theorem,   the number of the zero points
 of the function $p_x(y):=\partial_yw(x,y)$ is at least
 $N+1$. Since $J_1\subset J$, this contradicts the bound on the number of zeros given by Lemma \ref{lem:zero-point-formal}. Thus the assumption fails and the corollary holds.
\end{proof}

\begin{proof}[Proof of Proposition \ref{pro:non-degenecay-H_u-formal}]
 Assume that there exists $\mathbf{u}\in\mathcal{A}^{\infty}$ such that $H_{\mathbf{u}}'(x)\equiv 0$.
Combining Lemma  \ref{lem:partial_yg-constant-formal} and Corollary \ref{cor:finite-curves-formal}, the number of unstable curves is finite. This contradicts the conclusion of Lemma \ref{lem:curve-splitting-formal} since $\mathcal{S}$ is not a single graph.    
\end{proof}

%$\\\\\\\\\\\\\\\\\\\\\\\\\\\\\\$

%$\\\\\\\\\\\\\\\\\\\\\\\\\\\\\\\\\\$

\section{Transversality II: The case $\partial_y^2 g\equiv 0$}
\label{sec:transversality-II}

The previous section established transversality bounds under the condition that $\partial_y^2 g(x,y) \not \equiv 0$, combining Proposition \ref{pro:non-degenecay-H_u-formal} and \ref{pro:Separation-nozero-case}. It remains to prove them in the setting where $\partial_y^2 g(x,y) \equiv 0$. \\

{
In this case, it suffices to study the first derivative of $\Delta_{\mathbf{a},\mathbf{b}}$.
The functions $\theta_{\mathbf{a}}$ have then an explicit series representation; see \eqref{eq:theta-series-lambda(x)}. This gives a simple decomposition of
$\Delta_{\mathbf{u}\mathbf{a},\mathbf{u}\mathbf{b}}'(x)$ in terms of
$\Delta_{\mathbf{a},\mathbf{b}}(\mathbf{u}[x])$ and
$\Delta_{\mathbf{a},\mathbf{b}}'(\mathbf{u}[x])$.
Therefore we should study the  projections of
$\bigl(
\Delta_{\mathbf{a},\mathbf{b}}(\mathbf{u}[x]),
\Delta_{\mathbf{a},\mathbf{b}}'(\mathbf{u}[x])
\bigr)$
and derive a uniform lower bound for
$|\Delta_{\mathbf{u}\mathbf{a},\mathbf{u}\mathbf{b}}'(x)|$; see
\eqref{eq:normalized-derivative-common-prefix} and
\eqref{eq:normalized-derivative-decomposition}.
The following proposition gives the required estimate.
}

\begin{proposition}\label{pro:Separation-zero-case}
Assume that $\mathcal{S}$ does not consist of a single graph and that $\partial_y^2 g(x,y)\equiv 0$.
 Then there exists $N\in\Z_{>0}$ and $c\in(0,1)$ such that the following holds. For all $x\in [0,1]$ and $\mathbf{u}\in\mathcal{A}^*$ with $|\mathbf{u}|\ge N$, there exists 
    $\mathbf{a},\mathbf{b}\in\mathcal{A}^N$
    so that
    $$ |(\Delta_{\mathbf{u}\mathbf{a}\mathbf{c} ,\mathbf{u}\mathbf{b}\mathbf{d}})'(x)|\ge c (r_{\min}\kappa_{\min})^{|\mathbf{u}|}\qquad\text{for each }\mathbf{c},\mathbf{d}\in\mathcal{A}^* .$$   
\end{proposition}

{Recall that $r_{\min}$ and $\kappa_{\min}$ are defined in
\eqref{eq:r_min-r_max} and \eqref{eq:kappa_min-max}, respectively.}

In this section, we assume that $\mathcal{S}$ does not consist of a single graph and that $\partial_y^2 g(x,y)\equiv 0$.
Then there exist real analytic functions $\lambda,\phi:\mathbb{S}^1\to\mathbb{R}$ satisfying
$g(x,y)=\lambda(x)y+\phi(x).$
Since $0<|\partial_y g(x,y)|<1$, we get
$0<|\lambda(x)|<1$
for all $x\in\mathbb{S}^1$. We also regard $\lambda$ and $\phi$ as $\Z$-periodic real analytic functions on $\mathbb{R}$.

\subsection{Basic formulas in the affine case}
For each $\mathbf{u}\in\mathcal{A}^*$, define
$\lambda_{\mathbf{u}}\colon\R\to\mathbb{R}$ by
\begin{equation}\label{eq:lambda_u(x)}
    \lambda_{\mathbf{u}}(x)
:=
\prod_{k=1}^{|\mathbf{u}|}
\lambda\bigl((\mathbf{u}|_k)[x]\bigr).
\end{equation}
We also write $\mathbf{u}|_0:=\emptyset$ and $\lambda_{\emptyset}\equiv 1$.

By \eqref{def:G_u-formal}, for any
$\mathbf{c}\in\mathcal{A}^*\cup\mathcal{A}^\infty$, we have
\begin{equation}\label{eq:theta-series-lambda(x)}
    \theta_{\mathbf{c}}(x)
    =
    \sum_{n=1}^{|\mathbf{c}|}
    \lambda_{\mathbf{c}|_{n-1}}(x)
    \phi\bigl((\mathbf{c}|_n)[x]\bigr).
\end{equation}
where, when $\mathbf{c}\in\mathcal{A}^\infty$, the right hand side is  an infinite series.

{For every $\mathbf{b},\mathbf{u}\in\mathcal{A}^*$, we have 
\begin{equation}\nonumber (\mathbf{b}\mathbf{u})[x] = \mathbf{u}[\mathbf{b}[x]] \qquad\text{and}\qquad \lambda_{\mathbf{b}\mathbf{u}}(x) = \lambda_{\mathbf{b}}(x) \lambda_{\mathbf{u}}(\mathbf{b}[x]). \end{equation} }
Consequently,
\begin{equation}\label{eq:theta-series-two-symbols}
    \theta_{\mathbf{b}\mathbf{c}}(x)
    =
    \theta_{\mathbf{b}}(x)
    +
    \lambda_{\mathbf{b}}(x)
    \theta_{\mathbf{c}}\bigl(\mathbf{b}[x]\bigr).
\end{equation}

Given $\mathbf{d}\in\mathcal{A}^*\cup\mathcal{A}^\infty$, we have
\begin{equation}\label{eq:derivative-minus-two-symbol-same-strart-lambda(x)-2}
 (\theta_{\mathbf{b}\mathbf{c}}-\theta_{\mathbf{b}\mathbf{d}})'(x) 
 =
 \lambda_{\mathbf{b}}'(x)\cdot 
\bigl(\theta_{\mathbf{c}}
-\theta_{\mathbf{d}}\bigr)
\bigl(\mathbf{b}[x]\bigr)+
 \lambda_{\mathbf{b}}(x)\cdot
\bigl(\mathbf{b}[x]\bigr)'
\bigl(\theta_{\mathbf{c}}'
-\theta_{\mathbf{d}}'\bigr)
\bigl(\mathbf{b}[x]\bigr).
\end{equation}

Since $\lambda_{\mathbf{b}}(x)\neq 0$ and
$\bigl(\mathbf{b}[x]\bigr)'\neq 0$, we may define
\begin{equation}\label{eq:def-t1-t2}
    t_1(\mathbf{b})
    :=
    \frac{\lambda_{\mathbf{b}}'(x)}
    {\sqrt{
        |\lambda_{\mathbf{b}}'(x)|^2
        +
        |\lambda_{\mathbf{b}}(x)
        (\mathbf{b}[x])'|^2
    }},
    \qquad
    t_2(\mathbf{b})
    :=
    \frac{
        \lambda_{\mathbf{b}}(x)
        (\mathbf{b}[x])'
    }
    {\sqrt{
        |\lambda_{\mathbf{b}}'(x)|^2
        +
        |\lambda_{\mathbf{b}}(x)
        (\mathbf{b}[x])'|^2
    }}.
\end{equation}
Then $t_1^2+t_2^2=1$, and
{
\begin{equation}\label{eq:normalized-derivative-common-prefix}
    \frac{
        (\theta_{\mathbf{b}\mathbf{c}}
        -\theta_{\mathbf{b}\mathbf{d}})'(x)
    }
    {\sqrt{
        |\lambda_{\mathbf{b}}'(x)|^2
        +
        |\lambda_{\mathbf{b}}(x)
        (\mathbf{b}[x])'|^2
    }}
    =
    t_1(\mathbf{b})\cdot \Delta_{\mathbf{c},\mathbf{d}}
    \bigl(\mathbf{b}[x]\bigr)
    +
    t_2(\mathbf{b})\cdot
    \Delta_{\mathbf{c},\mathbf{d}}'
    \bigl(\mathbf{b}[x]\bigr).
\end{equation}
}

\subsection{The projection non-degeneracy case} Let
\begin{equation}\label{eq:fail-formula-lambda(x)}
   \Delta:=  \min_{x\in [0,1],\eta\in [0,2\pi]}\max_{\mathbf{c},\mathbf{d}\in\mathcal{A}^{\infty} }\left|\cos\eta\cdot 
\Delta_{\mathbf{c},\mathbf{d}}(x)+\sin\eta\cdot \Delta_{\mathbf{c},\mathbf{d}}'(x)  \right|.
\end{equation}

The  following lemma is the main result of this subsection.

\begin{lemma}\label{lem:non-degeneracy-case}
    If $\Delta>0$, then there exists $N\in\Z_{>0}$ and $c\in(0,1)$ such that the following holds. For each $x\in [0,1]$ and $\mathbf{u}\in\mathcal{A}^*$, there exists 
    $\mathbf{a},\mathbf{b}\in\mathcal{A}^N$
    such that
    $$ |(\Delta_{\mathbf{u}\mathbf{a}\mathbf{c} ,\mathbf{u}\mathbf{b}\mathbf{d}})'(x)|\ge c (r_{\min}\kappa_{\min})^{|\mathbf{u}|}\qquad\text{for any }\mathbf{c},\mathbf{d}\in\mathcal{A}^* .$$  
\end{lemma}

\begin{proof}
Set $c:=\Delta/2$. Choose $r\in(0,1)$, $C_1>1$, $\eta>0$, and $\delta\in(0,1)$ such that the conclusion of Corollary \ref{cor:uniform-holomorphic-extensions} holds with $r$ and $\eta$, and the conclusion of Lemma \ref{lem:uniform-bound-k-derivative} holds with $\delta$, $r$, $k=1$, and $C_1$. Choose $N\in\Z_{>0}$
such that 
$$ 4 r^N\le \frac{\Delta}{8}\qquad\text{and}\qquad C_1 r^N\le \frac{\Delta}{8}.$$

Let $x\in [0,1]$ and $\mathbf{u}\in\mathcal{A}^*$.
Write
$y:=\mathbf{u}[x]$.
Choose $\eta\in [0,2\pi]$ such that
$\cos\eta=t_1(\mathbf{u})$ and $\sin\eta=t_1(\mathbf{u})$ where
$t_1,t_2$ are defined by \eqref{eq:def-t1-t2}. By the definition of $\Delta$, there exist
$\mathbf{a}',\mathbf{b}'\in\mathcal{A}^{\infty} $ satisfying
\begin{equation}\label{eq:lower-bound-1}
    \left|\cos\eta\cdot 
\Delta_{\mathbf{a}' ,\mathbf{b}'}(y)+\sin\eta\cdot \Delta_{\mathbf{a}' ,\mathbf{b}'}'(y)  \right|\ge \Delta.
\end{equation}

Set $\mathbf{a}:=\mathbf{a}'|_N$, $\mathbf{b}:=\mathbf{b}'|_N$. 
Thus for any $\mathbf{c},\mathbf{d}\in\mathcal{A}^*$, by Corollary \ref{cor:uniform-holomorphic-extensions} and $4 r^N\le \frac{\Delta}{8}$, 
\begin{equation}\label{eq:non-degeneracy-case-2}
    |(\Delta_{\mathbf{a}\mathbf{c} ,\mathbf{b}\mathbf{d}}-\Delta_{\mathbf{a}' ,\mathbf{b}'})(y)|\le 8 r^N\le \frac{\Delta}{4}.
\end{equation}
Similarly, applying Lemma \ref{lem:uniform-bound-k-derivative} and $C_1 r^N\le \frac{\Delta}{8}$, we have
$$
    |(\Delta_{\mathbf{a}\mathbf{c} ,\mathbf{b}\mathbf{d}}-\Delta_{\mathbf{a}' ,\mathbf{b}'})'(y)|\le 2C_1 r^N\le \frac{\Delta}{4}.
$$
Combining this with \eqref{eq:lower-bound-1} and
\eqref{eq:non-degeneracy-case-2}, 
\begin{equation}\label{eq:non-degeneracy-case-3}
\begin{split}
    &\left|
    \cos\eta\cdot
    \Delta_{\mathbf{a}\mathbf{c},\mathbf{b}\mathbf{d}}(y)
    +
    \sin\eta\cdot
    \Delta_{\mathbf{a}\mathbf{c},\mathbf{b}\mathbf{d}}'(y)
    \right|
    \\
    &\qquad\ge
    \left|
    \cos\eta\cdot
    \Delta_{\mathbf{a}',\mathbf{b}'}(y)
    +
    \sin\eta\cdot
    \Delta_{\mathbf{a}',\mathbf{b}'}'(y)
    \right|
    \\
    &\qquad\quad
    -
    \left|
    \bigl(
    \Delta_{\mathbf{a}\mathbf{c},\mathbf{b}\mathbf{d}}
    -
    \Delta_{\mathbf{a}',\mathbf{b}'}
    \bigr)(y)
    \right|
    -
    \left|
    \bigl(
    \Delta_{\mathbf{a}\mathbf{c},\mathbf{b}\mathbf{d}}
    -
    \Delta_{\mathbf{a}',\mathbf{b}'}
    \bigr)'(y)
    \right|
    \\
    &\qquad\ge
    \frac{\Delta}{2}.
\end{split}
\end{equation}

Since$
|\lambda_{\mathbf{u}}(x)(\mathbf{u}[x])'|
\ge
(r_{\min}\kappa_{\min})^{|\mathbf{u}|},$
by \eqref{eq:normalized-derivative-common-prefix},
\eqref{eq:non-degeneracy-case-3}, and the choice of $\eta$, we have
$$
\left|
\Delta_{\mathbf{u}\mathbf{a}\mathbf{c},
\mathbf{u}\mathbf{b}\mathbf{d}}'(x)
\right|
\ge
c(r_{\min}\kappa_{\min})^{|\mathbf{u}|}.
$$
Thus the lemma holds.
\end{proof}

\subsection{The projection degeneracy case and proof of Proposition \ref{pro:Separation-zero-case}}
It remains to study the case $\Delta=0$. We will show that in this case $\lambda(x)$ is multiplicatively cohomologous to a constant $\lambda_0$. The dynamics is then  conjugated to the $\lambda(x)=\lambda_0$ setting, in which transversality is easy to characterize.

Set
$$
E
:=
\left\{
x\in\R:
\begin{array}{l}
\text{there exists } r\in\mathbb{R} \text{ satisfying } \\[2mm]
r\Delta_{\mathbf{c},\mathbf{d}}(x)
+
\Delta_{\mathbf{c},\mathbf{d}}'(x)
=0
\quad
\text{for any }
\mathbf{c},\mathbf{d}\in\mathcal{A}^{\infty}
\end{array}
\right\}.
$$

Since $\mathcal{S}$ does not consist of a single graph, for every $x\in\mathbb{R}$, there exist
$\mathbf{c},\mathbf{d}\in\mathcal{A}^{\infty}$ such that
$\theta_{\mathbf{c}}(x)\neq\theta_{\mathbf{d}}(x).$ 
Therefore, for every $x\in\R$, there exists at most one $r(x)\in\mathbb{R}$ satisfying
$$
r(x)\Delta_{\mathbf{c},\mathbf{d}}(x)
+
\Delta_{\mathbf{c},\mathbf{d}}'(x)
=0\qquad\text{for each }\mathbf{c},\mathbf{d}\in\mathcal{A}^{\infty}.
$$
In particular, $r(x)$ is uniquely determined by  $x\in E$.

By Lemma \ref{lem:addition-theta_u-formal} and the uniqueness of $r(x)$ on $E$, if $x\in E$, then 
\begin{equation}\label{eq:r(0)=r(1)}
   x+k\in E\qquad\text{and}\qquad r(x)=r(x+k)\qquad\text{for each }k\in\Z.
\end{equation}

\begin{lemma}
\label{lem:closedness-and-continuity-of-r} 
The set \(E\) is closed in \(\R\), and the function \(r\colon E\to\mathbb{R}\) is continuous.
\end{lemma}

\begin{proof}
We only consider the case $E\neq\emptyset$. Let  \((x_n)_{n\geq 1}\subset E\) be a sequence such that
$x_n\to x_0\in\R$
for some $x_0\in[0,1]$.
Passing to a subsequence,  we can assume that
$r(x_n)\to r_0\in\R\cup\{\infty\}.$
 Since $\mathcal{S}$ does not consist of a single graph, we have $r_0\neq\infty$.
Since the functions $\Delta_{\mathbf{c},\mathbf{d}}(x)$
and 
$\Delta_{\mathbf{c},\mathbf{d}}'(x)$ are continuous on $[0,1]$,  we get
$$
r_0\cdot \Delta_{\mathbf{c},\mathbf{d}}(x_0)
+
\Delta_{\mathbf{c},\mathbf{d}}'(x_0)
=0\qquad\text{for all }\mathbf{c},\mathbf{d}\in\mathcal{A}^{\infty}.
$$
Then $x_0\in E$.
By the uniqueness of $r(x)$, we have  $r(x_0)=r_0$.  Then $E$ is closed set.

By the similar argument as above,  $r(x)$ is continuous function on $E$.
\end{proof}

Since it 
sometimes is convenient to work on $\mathbb{R}$, we also denote by $f$ the lift $\widetilde{f}$ of $f$ to a diffeomorphism of $\mathbb{R}$.

\begin{lemma}\label{lem:fail-formula-lambda(x)}
If $E\neq\emptyset$, then $E=\mathbb{R}$ and 
    \begin{equation}\label{eq:fail-formula-lambda(x)-0}
        r(x)\equiv \frac{\lambda'(x)}{\lambda(x)}+r( f(x))\cdot f'(x).
    \end{equation}
\end{lemma}

\begin{proof}
Fix
a point $x_0\in E$. By \eqref{eq:r(0)=r(1)}, we can assume that $x_0\in [0,1]$.

Let $\mathbf{c},\mathbf{d}\in\mathcal{A}^{\infty}$, and let $\mathbf{b}\in\mathcal{A}$.
Then
\begin{equation}\label{eq:fail-formula-lambda-1}
r(x_0)\cdot \Delta_{\mathbf{b}\mathbf{c},\mathbf{b}\mathbf{d}}(x_0)
+
\Delta_{\mathbf{b}\mathbf{c},\mathbf{b}\mathbf{d}}'(x_0)
=0
\end{equation}

Let $y_0=\mathbf{b}[x_0]$. Thus $f(y_0)=x_0+k$ for some $k\in\mathbb{Z}$. By \eqref{eq:r(0)=r(1)},  $f(y_0)\in E$ and
$$
r(x_0)=r(x_0+k)=r(f(y_0)).
$$

Combining \eqref{eq:theta-series-two-symbols} with \eqref{eq:derivative-minus-two-symbol-same-strart-lambda(x)-2} and
$\lambda_{\mathbf{b}}(x)=\lambda(\mathbf{b}[x])$, we obtain
$$ \Delta_{\mathbf{b}\mathbf{c},\mathbf{b}\mathbf{d}}(x_0)=\lambda(y_0)\cdot \Delta_{\mathbf{c},\mathbf{d}}(y_0)  $$
and
$$\Delta_{\mathbf{b}\mathbf{c},\mathbf{b}\mathbf{d}}'(x) 
 =
 (\mathbf{b}[x_0])'\cdot\lambda'(y_0)\cdot 
\Delta_{\mathbf{c},\mathbf{d}}
\bigl(y_0\bigr)+
 \lambda(y_0)\cdot
\bigl(\mathbf{b}[x_0]\bigr)'
\Delta_{\mathbf{c},\mathbf{d}}'
\bigl(y_0\bigr).$$
Combining these with \eqref{eq:fail-formula-lambda-1} and $ (\mathbf{b}[x_0])'=1/f'(y_0)$ and $r(x_0)=r(f(y_0))$, we get
$$
    \left(  f'(y_0)\cdot r(f(y_0))+\frac{\lambda'(y_0)}{\lambda(y_0)}  \right)\cdot 
    \Delta_{\mathbf{c},\mathbf{d}}(y_0)+\Delta_{\mathbf{c},\mathbf{d}}'
\bigl(y_0\bigr)=0.
$$
By the arbitrariness of \(\mathbf{c},\mathbf{d}\in\mathcal{A}^{\infty}\) and the uniqueness of \(r(y_0)\), we have
$y_0\in E$  and
\begin{equation}\label{eq:fail-formula-lambda-2}
    r(y_0)=f'(y_0)\cdot r(f(y_0))+\frac{\lambda'(y_0)}{\lambda(y_0)}.
\end{equation}

{
Thus $\mathbf{b}[x_0]\in E$ for every $x_0\in E$ and
$\mathbf{b}\in\mathcal{A}$. By induction, 
$\mathbf{b}[x_0]\in E$ for every $\mathbf{b}\in\mathcal{A}^*$.
Since the set
$\{\mathbf{b}[x_0]\}_{\mathbf{b}\in\mathcal{A}^*}$
is dense in $[0,1]$, Lemma \ref{lem:closedness-and-continuity-of-r}
implies that $[0,1]\subset E$. Combining this with \eqref{eq:r(0)=r(1)}, we get
$E=\mathbb{R}$.
}

Since $r(y)$ is a continuous function on $E$, by the similar argument, \eqref{eq:fail-formula-lambda-2} implies that
$$ r(y)=f'(y)\cdot r(f(y))+\frac{\lambda'(y)}{\lambda(y)}\qquad\text{for all }y\in \R.$$
Thus this lemma holds.
\end{proof}

In the next two lemmas, we will handle the case when \eqref{eq:fail-formula-lambda(x)-0} holds.

{
\begin{lemma}\label{lem:r(x)-fromula}
    If $r:\mathbb{R}\to\mathbb{R}$ is $\mathbb{Z}$-periodic and satisfies
    \begin{equation}\label{eq: degeneracy-formula}
        r(x)\equiv \frac{\lambda'(x)}{\lambda(x)}+r( f(x))\cdot f'(x),
    \end{equation}
    then there exist a $\mathbb{Z}$-periodic $C^1$ function $R$ on $\mathbb{R}$ and a constant $\lambda_0$ with $|\lambda_0|\in(0,1)$ with
$$\lambda(x)\equiv\lambda_0\exp{\left(R(x)-R(f(x))\right)}.$$
\end{lemma}
}

\begin{proof}
{
Since $|\lambda(x)|>0$, we only consider the case $\lambda(x)>0$, the other case is the same.
}

Let $b>1$ be the degree of the expanding map $f$. 
By \eqref{eq: degeneracy-formula},
$$\int_0^1r(x)dx= \int_0^1\frac{\lambda'(x)}{\lambda(x)}dx +\int_0^1 r( f(x))\cdot f'(x)dx.$$
    Combining this with $\lambda(0)=\lambda(1)$, we get
    $$  \int_0^1r(x)dx=b  \int_0^1r(x)dx.$$
Therefore $\int_0^1r(x)dx=0$. We  may define the $\Z$-periodic function $R\in C^1(\R;\R)$ by
$$ R(x):= \int_0^xr(s)ds. $$
Then \eqref{eq: degeneracy-formula} implies that 
$$ R'(x) \equiv (\log\lambda(x))'+(R(f(x)))'. $$
{
Since $f(0)=0$, the claim holds with $\lambda_0=\lambda(0)$.
}
\end{proof}

Since $R$ is $\mathbb{Z}$-periodic function, we also regard $R$ as a function on $\mathbb{S}^1$.

\begin{lemma}\label{lem:reduction-to-constant-lambda}
{
Suppose that there exist a function $R\in C^1(\mathbb{S}^1;\mathbb{R})$ and a constant $\lambda_0$ with $|\lambda_0|\in(0,1)$ satisfying $$\lambda(x)\equiv\lambda_0\exp{\left(R(x)-R(f(x))\right)}.$$
}
    Let $\tilde{\phi}(x):=e^{-R(f(x))}\cdot \phi(x)$.
Then  for each $\mathbf{a}\in\mathcal{A}^\infty$, we obtain:
\begin{equation}
    \theta_{\mathbf{a}}(x)=e^{R(x)}\cdot\tilde{\theta}_{\mathbf{a}}(x),
\end{equation}
where the function $\tilde{\theta}_{\mathbf{a}}:\R\to\R$ is defined by
$$
    \tilde{\theta}_\mathbf{a}(x)=\sum_{n=1}^\infty\lambda_0^{n-1}\cdot \tilde{\phi}((\mathbf{a}|_n)[x]).
$$ 
\end{lemma}

\begin{proof}
We shall first consider the maps $\tilde{T},H:\mathbb{S}^1\times\R\to\mathbb{S}^1\times\R$ defined by 
$$  \tilde{T}(x,y):= (f(x),\lambda_0y+ \tilde{\phi}(x))\qquad  H(x,y):=(x, e^{R(x)}y).$$
Recall that $T(x,y)=(f(x),\lambda(x)y+\phi(x))$. Then by $\lambda(x)e^{R(f(x))}=\lambda_0e^{R(x)}$ we have
\begin{equation}
\label{eq:conjugacy-relation}
H\circ \widetilde{T}(x,y)
=
T\circ H(x,y).
\end{equation}

{
Since $H(x,0)=(x,0)$, then for any
$\mathbf{a}\in\mathcal{A}^{\infty}$ and $x\in[0,1]$,  we have
$$
(x,\theta_{\mathbf{a}}(x))
=
\lim_{n\to\infty}
T^n\bigl((\mathbf{a}|_n)[x],0\bigr)
=
\lim_{n\to\infty}
T^n\left(
H\bigl((\mathbf{a}|_n)[x],0\bigr)
\right)
$$
and
$$
(x,\widetilde{\theta}_{\mathbf{a}}(x))
=
\lim_{n\to\infty}
\widetilde{T}^{\,n}
\bigl((\mathbf{a}|_n)[x],0\bigr).
$$
}

Combining these with \eqref{eq:conjugacy-relation}, we  get
$$(x,\theta_{\mathbf{a}}(x))=H(x,\tilde{\theta}_{\mathbf{a}}(x)).$$
Thus $\theta_{\mathbf{a}}(x)=e^{R(x)}\cdot\tilde{\theta}_{\mathbf{a}}(x)$. Therefore this lemma holds.
\end{proof}

{We now reduce the problem once again to a projection estimate for $\widetilde{\theta}_{\mathbf{a}}$.}

Under the assumption of Lemma \ref{lem:reduction-to-constant-lambda}, for $\mathbf{b}\in\mathcal{A}^*$ and $\mathbf{c}, \mathbf{d}\in\mathcal{A}^*\cup\mathcal{A}^\infty$, by \eqref{eq:theta-series-two-symbols},
\begin{equation}\nonumber
    \begin{split}
     \left(\theta_{\mathbf{b}\mathbf{c}}-\theta_{\mathbf{b}\mathbf{d}} \right)'(x)&=  \left(e^{R(x)}\cdot\left(     \tilde{\theta}_{\mathbf{b}\mathbf{c}}-\tilde{\theta}_{\mathbf{b}\mathbf{d}} \right)(x)\right)'\\&=\left(e^{R(x)}\cdot\left(     \tilde{\theta}_{\mathbf{c}}-\tilde{\theta}_{\mathbf{d}} \right)(\mathbf{b}[x])\right)'\\&=
     e^{R(x)}\cdot\left( R'(x)\left(     \tilde{\theta}_{\mathbf{c}}-\tilde{\theta}_{\mathbf{d}} \right)(\mathbf{b}[x])+\mathbf{b}[x]' \cdot \left(     \tilde{\theta}_{\mathbf{c}}-\tilde{\theta}_{\mathbf{d}} \right)'(\mathbf{b}[x])\right).
    \end{split}
\end{equation}
 Define
\begin{equation}\label{eq:s_1-s_2}
     s_1(\mathbf{b}):=\frac{R'(x)}{\sqrt{|R'(x)|^2+|\mathbf{b}[x]'|^2}  }\qquad  s_2(\mathbf{b}):= \frac{\mathbf{b}[x]'}{\sqrt{|R'(x)|^2+|\mathbf{b}[x]'|^2}  }, 
\end{equation}
then we have
\begin{equation}\label{eq:normalized-derivative-decomposition}   \frac{\left(\theta_{\mathbf{b}\mathbf{c}}-\theta_{\mathbf{b}\mathbf{d}} \right)'(x)}{e^{R(x)}\cdot\sqrt{|R'(x)|^2+|\mathbf{b}[x]'|^2}}=s_1\cdot \left(     \tilde{\theta}_{\mathbf{c}}-\tilde{\theta}_{\mathbf{d}} \right)(\mathbf{b}[x])+s_2\cdot  \left(     \tilde{\theta}_{\mathbf{c}}-\tilde{\theta}_{\mathbf{d}} \right)'(\mathbf{b}[x]).
\end{equation}

\begin{lemma}\label{lem:uniform-transversality-constant-lambda}
Under the assumptions of Lemma \ref{lem:reduction-to-constant-lambda}, let
$\tilde{\phi}$ and $\tilde{\theta}_{\mathbf{a}}$ be given in that lemma.
Then 
\begin{equation}\label{eq:def-tilde-Delta}
\tilde{\Delta}
:=
\min_{\substack{x\in[0,1]\\ \eta\in[0,2\pi]}}
\max_{\mathbf{c},\mathbf{d}\in\mathcal{A}^{\infty}}
\left|
\cos\eta\,
\bigl(\tilde{\theta}_{\mathbf{c}}-\tilde{\theta}_{\mathbf{d}}\bigr)(x)
+
\sin\eta\,
\bigl(\tilde{\theta}_{\mathbf{c}}'-\tilde{\theta}_{\mathbf{d}}'\bigr)(x)
\right|
>0.
\end{equation}
\end{lemma}

\begin{proof}
Using the same argument as in the proof of Lemma \ref{lem:fail-formula-lambda(x)}, there exists a continuous function
$\widetilde{r}\in C(\mathbb{S}^1,\mathbb{R})$ 
satisfying
\begin{equation}\label{eq:defining-relation-for-r}
\widetilde{r}(x)
\bigl(\widetilde{\theta}_{\mathbf{c}}
-\widetilde{\theta}_{\mathbf{d}}\bigr)(x)
+
\bigl(\widetilde{\theta}_{\mathbf{c}}
-\widetilde{\theta}_{\mathbf{d}}\bigr)'(x)
\equiv 0
\qquad
\text{for each }
\mathbf{c},\mathbf{d}\in\mathcal{A}^{\infty},
\end{equation}
and
$$
\widetilde{r}(x)\equiv f'(x)\widetilde{r}(f(x)).
$$
{
Here the latter equation follows from  that the fiber coefficient of the conjugated system is the constant $\lambda_0$.
}
{
Iterating above, for any $x\in\mathbb{S}^1$ and $n\in\mathbb{N}$, we have
$$
\widetilde{r}(x)
=
(f^n)'(x)\cdot \widetilde{r}(f^n(x)).
$$
}

Then fix a point $y\in\mathbb{S}^1$, we obtain:
    $$ \widetilde{r}(\mathbf{b}[y])=\frac{ \widetilde{r}(y)}{(f^n)'(\mathbf{b}[y])}\qquad\text{for all }\mathbf{b}\in\mathcal{A}^*. $$
Combining this with  that the set $\{\mathbf{b}[y] \}_{\mathbf{b}\in \mathcal{A}^*}$ is dense set on $\mathbb{S}^1$ and
$$|(f^n)'(\mathbf{b}[y])|\ge (s_{\min})^{|\mathbf{b}|}\qquad s_{\min}:=\inf_{t\in \mathbb{S}^1}|f'(t)|>1,$$
we get
$ \widetilde{r}(x)\equiv 0.$ Thus \eqref{eq:defining-relation-for-r} implies that
$$ \bigl(\tilde{\theta}_{\mathbf{c}}
-\tilde{\theta}_{\mathbf{d}}\bigr)'(x)
\equiv 0\qquad\text{for all }\mathbf{c},\mathbf{d}\in\mathcal{A}^{\infty}.$$
{
By Lemma \ref{lem:reduction-to-constant-lambda}, the solenoidal attractor $\widetilde{ \mathcal{S} }$ of the conjugated system does not consist of a single graph. This contradicts Lemma \ref{lem:rigidity-derivatives}.
}
Thus our assumption fails and the lemma holds.
\end{proof}

\begin{proof}[Proof of Proposition \ref{pro:Separation-zero-case}]
By Lemma \ref{lem:non-degeneracy-case}, we only need to prove the case where
$\Delta=0$, equivalently, $E\neq\emptyset$.
By Lemmas \ref{lem:fail-formula-lambda(x)} and \ref{lem:r(x)-fromula},
the assumptions of Lemma \ref{lem:reduction-to-constant-lambda} hold.
Let $\tilde{\phi}$ and $\tilde{\theta}_{\mathbf{a}}$ be defined as in that lemma.
Then, by Lemma \ref{lem:uniform-transversality-constant-lambda}, we have
\begin{equation}\label{eq:uniform-transversality}
    \min_{\substack{x\in[0,1]\\ \eta\in[0,2\pi]}}
\max_{\mathbf{c},\mathbf{d}\in\mathcal{A}^{\infty}}
\left|
\cos\eta\,
\bigl(\tilde{\theta}_{\mathbf{c}}-\tilde{\theta}_{\mathbf{d}}\bigr)(x)
+
\sin\eta\,
\bigl(\tilde{\theta}_{\mathbf{c}}'-\tilde{\theta}_{\mathbf{d}}'\bigr)(x)
\right|
>0.
\end{equation}
Recall that $\tilde{\theta}_{\mathbf{a}}$ is the function associated with the affine fiber map
$$
\tilde{g}(x,y)=\lambda_0 y+\tilde{\phi}(x),
$$
where $\lambda_0$ is a constant with $|\lambda_0|\in (0,1)$.

Therefore, combining the above transversality estimate \eqref{eq:uniform-transversality} with
\eqref{eq:normalized-derivative-decomposition} and \eqref{eq:s_1-s_2},
and applying the similar argument as in the proof of Lemma
\ref{lem:non-degeneracy-case} to the family
$\{\tilde{\theta}_{\mathbf{a}}\}_{\mathbf{a}\in\mathcal{A}^{\infty}}$, 
the claim also holds for the case $\Delta=0$.
\end{proof}

%$\\\\\\\\\\\\\\\\\\\\\\\\\\$

\section{Tree Theorem}\label{sec:proof-tree-theorem}

Given $k\in\Z_{>0}$ and $\mathbf{a}\in\mathcal{A}^*$, define the vector valued function
$$ \Theta_{\mathbf{a},k}(x):=\left( \theta_{\mathbf{a}}^{(1)}(x),\ldots,\theta_{\mathbf{a}}^{(k)}(x)  \right)\qquad\text{for }x\in [0,1].$$

The following tree theorem is the main result of this section.

\begin{theorem}
\label{thm:tree-theorem-general-case-formal}
Suppose that $\mathcal{S}$ does not consist of a single graph.

Let $\psi:\mathbb{S}^1\to\R$ be  a H\"older continuous function.
Then there exist $K\in\Z_{>0}$ and constants $C_0,\gamma,\alpha>0$ such that, for every $n\in\Z_{>0}$ and $\sigma>0$, the following holds.

For $x\in [0,1]$ and $t\in\R^K$, we have
$$\sum_{\mathbf{a}\in\mathcal{G}} 
p_{\mathbf{a}}(x)\le C_0(\sigma^{\gamma}+e^{-\alpha n}), $$
where the set $\mathcal{G}$ is defined by 
\begin{equation}\label{def:G}
    \mathcal{G}:=\left\{ \mathbf{a} \in \mathcal{A}^n\::\: \Theta_{\mathbf{a},K}(x) \in \mathbf{B}(t,\sigma) \right\}.
\end{equation}
\end{theorem}
Here $\mathbf{B}(t,\sigma)$ is the closed ball in $\R^K$ with center at $t$ and radius $\sigma$.
The positive probability vector $\{p_{\mathbf{a}}(x)\}_{\mathbf{a}\in\mathcal{A}^n } $ is defined by \eqref{eq:p_u}. \\

Throughout this section, we assume that $\mathcal{S}$ does not consist of a single graph. Before heading into the proof, let us recall the transversality bound that follows from our two previous section. The proof of the Tree theorem is essentially an iterated application of this multiscale separation result.

\begin{proposition}\label{pro:Separation}
    There exists $K,N\in\Z_{>0}$ and $c\in(0,1)$ such that the following holds. For any $x\in [0,1]$ and $\mathbf{u}\in\mathcal{A}^*$ satisfying $|\mathbf{u}|\ge N$, there exists 
    $k\in\{1,\dots,K   \}$ and
    $\mathbf{a},\mathbf{b}\in\mathcal{A}^N$
    such that
    $$ |(\Delta_{\mathbf{u}\mathbf{a}\mathbf{c} ,\mathbf{u}\mathbf{b}\mathbf{d}})^{(k)}(x)|\ge c (r_{\min}\kappa_{\min})^{|\mathbf{u}|}\qquad\text{for all }\mathbf{c},\mathbf{d}\in\mathcal{A}^* .$$   
\end{proposition}

\begin{proof}
    Combining Proposition \ref{pro:non-degenecay-H_u-formal}, Proposition \ref{pro:Separation-nozero-case} and Proposition \ref{pro:Separation-zero-case}, our claim holds.
\end{proof}

\begin{proof}[Proof of Theorem \ref{thm:tree-theorem-general-case-formal}]
Let $K,N\in\mathbb{Z}_{>0}$ and $c>0$ be such that the conclusions of Lemma \ref{lem:G_a^(k)-G_b^(k)-lower-bound-formal} hold for $K,N$ and $\delta=c$, and the conclusions of Proposition \ref{pro:Separation} hold for $K,N,c$. Define
$$p_{\min}:=\min_{x\in [0,1],\, \mathbf{u}\in\mathcal{A}^N} p_{\mathbf{u}}(x)>0,$$
{
where $p_{\mathbf{u}}(x)$ is defined by \eqref{eq:p_u} with respect to the given potential function $\phi$.
}
Set
$$ \alpha:=\frac{-\log(1-p_{\min} )}{N}\qquad\text{and}\qquad \gamma:=\frac{\log(1-p_{\min})}{N\log r_{\min}\kappa_{\min}}. $$
Let $C_0>1$ be sufficiently large, to be specified later.

Let $x\in[0,1]$, $n\in\mathbb{Z}_{>0}$, and $\sigma>0$.  Let $t=(t_1,\cdots,t_K)\in\R^K$, define our target set as
\begin{equation}\label{eq:G-1}
     \mathcal{G}:=\{ \mathbf{a} \in \mathcal{A}^n\::\: \Theta_{\mathbf{a},K}(x) \in \mathbf{B}(t,\sigma) \},
\end{equation}
where  recall that $ \Theta_{\mathbf{a},K}(x):=\left( \theta_{\mathbf{a}}^{(1)}(x),\ldots,\theta_{\mathbf{a}}^{(K)}(x)  \right).$

We only need to consider the case
$$ n\ge 2N\qquad\text{and}\qquad  c (r_{\min}\kappa_{\min})^{N}>3\sigma,$$
since otherwise the conclusion of this theorem follows by taking $C_0>1$ sufficiently large.

\noindent\textbf{First case.}
We first consider the case $c(r_{\min}\kappa_{\min})^{n-N}\le 3\sigma$, and complete the proof in this case by induction.

\noindent\textbf{Step 1: Eliminating one initial symbol.} By Lemma \ref{lem:G_a^(k)-G_b^(k)-lower-bound-formal}, there exists $\mathbf{a},\mathbf{b}\in\mathcal{A}^N$
  and $k\in\{1,2,\ldots,K\}$
  such that 
  \begin{equation}\label{eq:tree-theorem-1}
      \left| (\theta_{\mathbf{a}\mathbf{c}}-\theta_{\mathbf{b}\mathbf{d}})^{(k)}(x)  \right|\ge c\qquad\text{for all } \mathbf{c},\mathbf{d}\in\mathcal{A}^{n-N}.
  \end{equation}

{
For $\mathbf{a},\mathbf{a}'\in\mathcal{A}^*$, we write
$\mathbf{a}\prec\mathbf{a}'$ if $\mathbf{a}$ is a  prefix of
$\mathbf{a}'$.
}

If there exists $\mathbf{a}'\in\mathcal{A}^n$
such that $\mathbf{a}\prec \mathbf{a}'$ and
$\mathbf{a}'\in\mathcal{G}$,  then for any 
$\mathbf{b}'\in\mathcal{A}^n$
with $\mathbf{b}\prec \mathbf{b}'$,
by \eqref{eq:tree-theorem-1} and $c>3\sigma$, we have
$$ \left| (\theta_{\mathbf{a}'}-\theta_{\mathbf{b}'})^{(k)}(x)  \right|\ge c>3\sigma.$$
{
Since $\mathbf{a}'\in\mathcal{G}$, we have
$\left| \theta_{\mathbf{a}'}^{(k)}(x)-t_k \right| \le \sigma.$ Therefore
}
$$ \left| \theta_{\mathbf{b}'}^{(k)}(x)-t_k \right| \ge 2\sigma.$$
Then by the definition of $\mathcal{G}$, we have $\mathbf{b}'\not\in\mathcal{G}$. In this case,  set
$$ \mathbf{w}:=\mathbf{b}\qquad\text{and}\qquad\mathcal{G}_1:=\mathcal{A}^N\setminus\{\mathbf{b}\}.$$
Therefore we obtain
\begin{equation}\label{eq:tree-theorem-2}
    \mathcal{G}\subset\{ \mathbf{v}\in\mathcal{A}^n: \mathbf{c}\prec\mathbf{v}\text{ for some } \mathbf{c}\in \mathcal{G}_1\}.
\end{equation}

If  $\mathbf{a}'\not\in\mathcal{G}$ for all $\mathbf{a}'\in\mathcal{A}^n$
with $\mathbf{a}\prec \mathbf{a}'$, then let
$$ \mathbf{w}:=\mathbf{a}\qquad\text{and}\qquad\mathcal{G}_1:=\mathcal{A}^N\setminus\{\mathbf{a}\}.$$
Therefore \eqref{eq:tree-theorem-2} holds  by our assumption in this case.

Let $\mathcal{G}_0:=\{\emptyset\}$
and $\mathbf{w}_1:\mathcal{G}_0\to \mathcal{A}^N $ be the map such that $\mathbf{w}_1(\emptyset):=\mathbf{w}$.

\noindent\textbf{Step 2: Eliminating the second remaining symbol in $\mathcal{G}_1$.}
Let $\mathbf{u}\in\mathcal{G}_1$. By Proposition \ref{pro:Separation}, there exist
$k=k(\mathbf{u})\in\{1,\dots,K   \}$ and
    $\mathbf{a}=\mathbf{a}(\mathbf{u}),\mathbf{b}=\mathbf{b}(\mathbf{u})\in\mathcal{A}^N$
    such that
    \begin{equation}\label{eq:tree-theorem-3}
         |(\Delta_{\mathbf{u}\mathbf{a}\mathbf{c} ,\mathbf{u}\mathbf{b}\mathbf{d}})^{(k)}(x)|\ge c (r_{\min}\kappa_{\min})^{|\mathbf{u}|}\qquad\text{for all }\mathbf{c},\mathbf{d}\in\mathcal{A}^{n-2N} .
    \end{equation}
   
Assume that there exists $\mathbf{a}'\in\mathcal{A}^n$ such that $\mathbf{u}\mathbf{a}\prec \mathbf{a}'$ and
$\mathbf{a}'\in\mathcal{G}$.  Let $\mathbf{b}'\in\mathcal{A}^n$ be with $\mathbf{u}\mathbf{b}\prec \mathbf{b}'$. Then
by \eqref{eq:tree-theorem-3} and $  c (r_{\min}\kappa_{\min})^{N}>3\sigma$, we obtain:
$$ \left| (\theta_{\mathbf{a}'}-\theta_{\mathbf{b}'})^{(k)}(x)  \right|\ge c (r_{\min}\kappa_{\min})^{N}> 3\sigma.$$
Combining this with 
$\left| \theta_{\mathbf{a}'}^{(k)}(x)-t_k \right| \le \sigma,$ we have
$$ \left| \theta_{\mathbf{b}'}^{(k)}(x)-t_k \right| > \sigma.$$
Therefore we have $\mathbf{b}'\not\in\mathcal{G}$. In this case, we set
$$ \mathbf{w}_2(\mathbf{u}):=\mathbf{b}\qquad\text{and}\qquad\mathcal{G}_2:=\left\{ \mathbf{u}\mathbf{v}\,:\,\mathbf{u}\in\mathcal{G}_1,\mathbf{v}\in\mathcal{A}^{N}\setminus\{\mathbf{b}\}\right\}.$$

If  $\mathbf{a}'\not\in\mathcal{G}$ for all $\mathbf{a}'\in\mathcal{A}^n$
such that $\mathbf{u}\mathbf{a}\prec \mathbf{a}'$, let
$$ \mathbf{w}_2(\mathbf{u}):=\mathbf{a}\qquad\text{and}\qquad\mathcal{G}_2:=\left\{ \mathbf{u}\mathbf{v}\,:\,\mathbf{u}\in\mathcal{G}_1,\mathbf{v}\in\mathcal{A}^{N}\setminus\{\mathbf{a}\}\right\}.$$

In any case,  we have 
$$   \mathcal{G}\subset\{ \mathbf{v}\in\mathcal{A}^n: \mathbf{c}\prec\mathbf{v}\text{ for some } \mathbf{c}\in \mathcal{G}_2\}.  $$

\noindent\textbf{Step 3: Induction on the $k$-th symbols.} Let $k\ge2$.
Assume that we have defined the
sets $\mathcal{G}_{j}\subset\mathcal{A}^{kN}$ and the functions $\mathbf{w}_j:\mathcal{G}_{j-1}\to\mathcal{A}^N, j\in\{1,\cdots k\}$
such that
\begin{equation}\label{eq:G-including}
    \mathcal{G}\subset\{ \mathbf{v}\in\mathcal{A}^n: \mathbf{c}\prec\mathbf{v}\text{ for some } \mathbf{c}\in \mathcal{G}_j\} 
\end{equation}
and
\begin{equation}\label{def:G_j}
    \mathcal{G}_j:=\left\{
\mathbf{u}\mathbf{v}\::\: \mathbf{u}\in\mathcal{G}_{j-1},\mathbf{v}\in\mathcal{A}^{N}\setminus\{\mathbf{w}_{j}(\mathbf{u})\}
\right\}.
\end{equation}

{
If $c(r_{\min}\kappa_{\min})^{(k+1)N}>3\sigma$, then,  following  the argument in Step 2 and applying Proposition \ref{pro:Separation}, the following holds.
}
For each $\mathbf{u}\in \mathcal{G}_k$, there exists $\mathbf{w}_{k+1}(\mathbf{u})\in\mathcal{A}^N$
such that 
$$ \mathbf{v}'\not\in\mathcal{G}\qquad\text{for each } \mathbf{v}'\in\mathcal{A}^n\text{ with }\mathbf{u} \mathbf{w}_{k+1}(\mathbf{u})\prec\mathbf{v}'.$$
Thus we define
$$ \mathcal{G}_{k+1}:=\left\{ \mathbf{u}\mathbf{v}\,:\,\mathbf{u}\in\mathcal{G}_k,\mathbf{v}\in\mathcal{A}^{N}\setminus\{\mathbf{w}_{k+1}(\mathbf{u})\}\right\}.$$
Then $\mathcal{G}_{k+1}$ satisfies the induction hypothesis for $k+1$.
This process terminates at the $m$-th step, where $m$ is the unique positive integer satisfying 
$$c(r_{\min}\kappa_{\min})^{mN}\le 3\sigma <c(r_{\min}\kappa_{\min})^{(m-1)N}.$$

\noindent\textbf{Step 4: An upper bound for the probability of $\mathcal{G}$.} 
Let $k\in\{1,\cdots,m-2\}$.  
Combining \eqref{def:G_j} with \eqref{eq:p-formula} and the definition of $p_{\min}$, we have
\begin{equation}
    \begin{split}        \sum_{\mathbf{z}\in\mathcal{G}_{k+1}} p_{\mathbf{z}}(x)&=\sum_{\mathbf{u}\in\mathcal{G}_{k}}\sum_{ \mathbf{v}\in\mathcal{A}^N\setminus\{ \mathbf{w}_{k+1}(\mathbf{u}) \} } p_{\mathbf{u}\mathbf{v} }(x) \\&=    \sum_{\mathbf{u}\in\mathcal{G}_{k}} p_{\mathbf{u}}(x)\sum_{ \mathbf{v}\in\mathcal{A}^N\setminus\{ \mathbf{w}_{k+1}(\mathbf{u}) \} }  p_{\mathbf{v}}(\mathbf{u}[x]) \\&\le   (1-p_{\min})\sum_{\mathbf{u}\in\mathcal{G}_{k}} p_{\mathbf{u}}(x).
    \end{split}
\end{equation}
After the induction, we obtain
$$ \sum_{\mathbf{z}\in\mathcal{G}_{m-1}}
 p_{\mathbf{z}}(x)\le  (1-p_{\min})^{m-1}. $$
Combining this and \eqref{eq:G-including}, we have
\begin{equation}\label{eq:key-mass-upper-bound}
    \sum_{\mathbf{a}\in\mathcal{G}}
     p_{\mathbf{a}}(x)\le\sum_{\mathbf{z}\in\mathcal{G}_{m-1}} p_{\mathbf{z}}(x)\le  (1-p_{\min})^{m-1}.
\end{equation}

By $m=\frac{\log\sigma}{N\log r_{\min}\kappa_{\min}}+O(1)$ and $\gamma=\frac{\log(1-p_{\min})}{N\log r_{\min}\kappa_{\min}}$, we have
$$  \sum_{\mathbf{a}\in\mathcal{G}}
     p_{\mathbf{a}}(x)\le (1-p_{\min})^{m-1}=O\left( \sigma^{ \frac{\log (1-p_{\min})}{\log\sigma}\cdot \frac{\log\sigma}{N\log r_{\min}\kappa_{\min}} }\right)=O(\sigma^{\gamma}).  $$

\noindent\textbf{Second case.}
We now consider the case $c(r_{\min}\kappa_{\min})^{n-N}> 3\sigma$.  Following the argument in the first case, we obtain  sets $\mathcal{G}_j\subset\mathcal{A}^{jN}$ and the functions $\mathbf{w}_j:\mathcal{G}_{j-1}\to\mathcal{A}^N, j\in\{1,\cdots m\}$
such that \eqref{eq:G-including} and \eqref{def:G_j} hold,
where $m:=\lfloor\frac{n}{N} \rfloor.$
Thus by $\alpha=\frac{-\log(1-p_{\min})}{N}$,
we obtain:
$$ \sum_{\mathbf{a}\in\mathcal{G}}
    p_{\mathbf{a}}(x)\le (1-p_{\min})^{m}=O\left( e^{ \frac{\log (1-p_{\min})}{N}\cdot n}\right)=O(e^{-\alpha n}).  $$
In any case, we prove this theorem.
\end{proof}

\subsection*{Declaration of AI use}

During the exploratory and manuscript-preparation stages of this work, H.R.
used ChatGPT, developed by OpenAI (GPT-5.6 Sol Pro), through the ChatGPT web
interface during July--September 2026. The tool was used for exploratory
mathematical brainstorming, investigation of candidate formulas and derivative
decompositions, literature searches, and the revision and polishing of selected
mathematical passages, including English and \LaTeX{} editing. It was also used
to explore possible transversality conditions and proof strategies in both the
affine fiber case and the general nonlinear case.

Among the AI-assisted exploratory inputs that informed the final manuscript
were the following. In the special affine-fiber case
$g(x,y)=\lambda(x)y+\phi(x)$ with nonconstant $\lambda$, AI-assisted
discussions contributed to identifying a possible route toward the required
first-order transversality estimate; the resulting argument was subsequently
developed, checked, and rewritten by the authors and is presented in a
separate section. In the study of decompositions separating a common symbolic
prefix, ChatGPT suggested the candidate quantity
$\Lambda_{\mathbf{u};\mathbf{a},\mathbf{b}}$. After the need for an appropriate
normalization had been identified, ChatGPT assisted in exploring a candidate
derivation of a formula for $H_{\mathbf{u}}$. Finally, after the relevant
mathematical framework and the analogy with the method of Gao and Shen
\cite{gao2024low} had been supplied, ChatGPT assisted in formulating a possible
structural consequence of the hypothetical degeneracy
$H_{\mathbf{u}}'\equiv 0$.

The research problem, the overall proof architecture, and the higher-order
derivative strategy were formulated by the authors. All final mathematical
judgments and all proofs in the manuscript were made, checked, and approved by
the authors. Any AI-assisted material retained in the manuscript was
independently rederived or verified and critically evaluated by the authors,
and all references identified through AI-assisted searches were checked
against the original sources. No AI-generated mathematical claim was accepted
without human verification. The authors take full responsibility for the
mathematical content of the manuscript.

\section*{Acknowledgements}
 G.L. is supported by the Mittag-Leffler Institute and the University of Helsinki, and thank Tuomas Sahlsten for his support.
 H.R. was supported by the Israel Science Foundation (grant No.~619/22).
He is grateful to the Department of Mathematics at the Technion for its hospitality, and to Ariel Rapaport for his guidance and support.
 We thank the organizers of the workshop "Focused Workshop on dimension drop for analytic Iterated Function Systems" at Erdös Center in July 2026, where this project was born: Balázs Bárány, Istvan Kolossvary, and Sascha Troscheit.

\bibliographystyle{alpha}
\bibliography{references}

\end{document}